\documentclass[12pt]{amsart}

\usepackage{geometry}
\usepackage{amsfonts, adjustbox, amssymb, amscd}
\numberwithin{equation}{section}

\usepackage{bm}
\usepackage{verbatim}
\usepackage{mathrsfs}
\usepackage{graphicx}
\usepackage{tikz-cd}
\usepackage{subcaption}
\usepackage{listings}
\usepackage{subfiles}
\usepackage[toc,page]{appendix}
\usepackage{mathtools}
\usepackage{comment}
\usepackage{enumerate}
\usepackage{enumitem}
\usepackage[all]{xy}
\usepackage{calligra}
\usepackage{hyperref}

\hypersetup{
    colorlinks=true,
    citecolor=red,
    linkcolor=blue,
    filecolor=magenta,      
    urlcolor=red,
}
\newcommand{\Exc}[0]{{\operatorname{Exc}}}
\newcommand{\Supp}{\operatorname{Supp}}
\newcommand{\Rr}{\mathbb{R}}
\newcommand{\Qq}{\mathbb{Q}}
\newcommand{\Ff}{\mathcal{F}}
\newcommand{\Aa}{{\mathfrak{A}}}

\DeclareMathOperator{\Ii}{\Gamma}
\DeclareMathOperator{\Center}{center}
\DeclareMathOperator{\Weil}{Weil}
\DeclareMathOperator{\Nklt}{Nklt}
\DeclareMathOperator{\irr}{irr}

\DeclareMathOperator{\Nlc}{Nlc}
\DeclareMathOperator{\mult}{mult}

\newtheorem{thm}{Theorem}[section]

\newtheorem{lem}[thm]{Lemma}
\newtheorem{cor}[thm]{Corollary}
\newtheorem{prop}[thm]{Proposition}

\theoremstyle{definition}
\newtheorem{defn}[thm]{Definition}
\newtheorem{rem}[thm]{Remark}

\newtheorem{deflem}[thm]{Definition-Lemma}

\newtheorem{nota}[thm]{Notation}
\newtheorem{cons}[thm]{Construction}

\newtheorem*{claim*}{Claim}

\newcommand{\bb}{\bm{b}}

\newcommand{\Dd}{{\bf{D}}}

\newcommand{\Cc}{\mathfrak{C}}
\newcommand{\Mm}{{\bf{M}}}
\newcommand{\Nn}{{\bf{N}}}



\begin{document}

\title{Existence of the minimal model program for log canonical generalized pairs}
\author{Zhengyu Hu and Jihao Liu}

\date{\today}

\begin{abstract}
We introduce linearly decomposable (LD) generalized pairs, which serve as a workable substitute for rational decompositions in the non-NQC setting. Using LD generalized pairs, together with a refinement of special termination and Koll\'ar-type gluing theory, we prove the existence of flips for log canonical generalized pairs without assuming the klt condition, the NQC condition, or $\mathbb Q$-factoriality. Together with the cone and contraction theorems, this yields the existence of the minimal model program for arbitrary log canonical generalized pairs.
\end{abstract}

\address{Mathematical Sciences Research Center, Chongqing University of Technology, No.69 Hongguang Avenue, Chongqing, 400054, China}
\email{zhengyuhu16@gmail.com}

\address{Department of Mathematics, Peking University, No. 5 Yiheyuan Road, Haidian District, Peking 100871, China}
\email{liujihao@math.pku.edu.cn}

\maketitle

\pagestyle{myheadings}\markboth{\hfill Zhengyu Hu and Jihao Liu \hfill}{\hfill Existence of the minimal model program for log canonical generalized pairs\hfill}

\tableofcontents

\section{Introduction}\label{sec: introduction}

We work over the field of complex numbers $\mathbb C$.

\subsection*{Background}

The theory of generalized pairs, introduced by Birkar and Zhang \cite{BZ16} in the study of the effective Iitaka fibration conjecture, has become a standard framework in birational geometry.  Roughly speaking, a generalized pair $(X,B,\Mm)$ keeps track not only of a boundary $\mathbb R$-divisor $B$ on $X$ but also of a nef part $\Mm$ living on a higher model, and this extra flexibility is well-suited for inductive constructions.  The basic ideas already appear in work of Birkar and the first author \cite{BH14}, and generalized pairs are closely related to adjunction-type constructions, including subadjunction and the canonical bundle formula \cite{Kaw98,FM00}.  Many major developments in the field, such as the proofs of the Borisov--Alexeev--Borisov conjecture \cite{Bir19,Bir21} and the M\textsuperscript{c}Kernan--Shokurov conjecture \cite{Bir23}, depend on this framework.  While the recent proof of Prokhorov--Shokurov's $\bb$-semi-ampleness conjecture \cite{BFMT25} clarifies the behavior of the moduli part in the canonical bundle formula, a systematic minimal model theory for log canonical generalized pairs remains a distinct and natural problem: in many contexts, the nef part is not assumed---and often not expected---to be semi-ample, e.g. termination of flips (cf. \cite{HM20,CT23,Mor25}), geometry of varieties with anti-nef canonical class (cf. \cite{FS23,Bir24}), K\"ahler and analytic MMP (cf. \cite{DH23,DHY23,DH24}), and foliations (cf. \cite{LLM23,Cas+24,LMX24,Cas+25a}).

A central problem in minimal model theory is to determine for which classes of structures one can run a minimal model program.  For klt generalized pairs, the existence of the minimal model program follows formally from the minimal model program for usual klt pairs, as explained in \cite{BZ16}, but the log canonical (lc) case has remained open for a long time.  As usual, the existence of the minimal model program can be divided into three parts: the cone theorem, the contraction theorem, and the existence of flips.  For generalized pairs, the cone theorem was proved in the NQC case (i.e.\ when the nef part is a positive linear combination of nef $\bb$-Cartier $\bb$-divisors) in \cite{HL23}, and in full generality in \cite{CHLX23}.  The contraction theorem was proved in the NQC case in \cite{Xie24} (see also \cite{CLX23}), and in full generality in \cite{CHLX23}.  On the other hand, flips are known to exist in the $\mathbb Q$-factorial NQC case \cite{HL23}, in the general NQC case \cite{LX23}, and in the $\mathbb Q$-factorial (not necessarily NQC) case \cite{CHLX23}.  The remaining case was the non-NQC, non-$\mathbb Q$-factorial, non-klt setting; see also \cite{TX24} for further remarks on the existence of minimal models of NQC log canonical generalized pairs.

\subsection*{Main results}

In this paper, we resolve this remaining case and prove the existence of flips for log canonical generalized pairs in full generality, i.e.\ without assuming $\mathbb Q$-factoriality, the NQC condition, or the klt condition.  In particular, this closes the last open case needed to run the MMP for log canonical generalized pairs.

\begin{thm}\label{thm: flip nonnqc-g}
    Let $(X,B,\Mm)/U$ be a log canonical generalized pair and let
    $f: X\rightarrow Z$
    be a $(K_X+B+\Mm_X)$-flipping contraction$/U$. Then the flip
    $f^+: X^+\rightarrow Z$
    of $f$ exists.
\end{thm}

Combining this with the cone and contraction theorems \cite[Theorem~B]{CHLX23}, we obtain the existence of the minimal model program for log canonical generalized pairs.

\begin{thm}\label{thm: eommp nonnqc-g}
Let $(X,B,\Mm)/U$ be a log canonical generalized pair. Then we may run a $(K_X+B+\Mm_X)$-MMP$/U$.
\end{thm}

As further outputs, we show that after adding an ample divisor, the log canonical class of a $\mathbb Q$-factorial lc generalized pair becomes $\mathbb R$-linearly equivalent to the log canonical class of an lc pair (Theorem~\ref{thm: structure of q-factorial gpair}).  We also describe the bases of Fano contractions of generalized pairs (Corollary~\ref{cor: potentially lc base}).

\subsection*{State of the art}

The following table summarizes previous results on the minimal model program for log canonical generalized pairs. Before the present paper, the only remaining case was the existence of flips in the general (non-klt, non-NQC, non-$\mathbb Q$-factorial) setting.

\begin{table}[!htbp]
   \caption{Minimal model program for lc generalized pairs $(X,B,\Mm)$}
    \label{tab:log structures}
\begin{center}
        \begin{tabular}{|c|c|c|c|}
\hline
& Cone theorem & Contraction theorem & Existence of flips \\
\hline
Klt, $\Mm=\bm{0}$ & cf.\ \cite{KMM87} & cf.\ \cite{KMM87} & \cite{BCHM10} \\
\hline
$\Mm=\bm{0}$ & \cite{Amb03}; \cite{Fuj11} &  \cite{Amb03}; \cite{Fuj11} & \cite{Bir12,HX13}  \\
\hline
Klt & \cite{BZ16} & \cite{BZ16} & \cite{BZ16} \\
\hline
NQC $\mathbb Q$-factorial & \cite{HL23} & \cite{HL23}  & \cite{HL23} \\
\hline
NQC & \cite{HL23} & \cite{Xie24}; \cite{CLX23}  & \cite{LX23} \\
\hline
$\mathbb Q$-factorial & \cite{CHLX23} & \cite{CHLX23} & \cite{CHLX23} \\
\hline
General case & \cite{CHLX23} & \cite{CHLX23} & This paper \\
\hline
\end{tabular}
\end{center}
\end{table}

\subsection*{Strategy and intermediate results}

As in the classical case, to construct flips it is convenient to reduce to the existence of good minimal models in certain special situations.  For lc pairs, two basic criteria of this type are known: the existence of good minimal models for relatively potentially log Calabi--Yau pairs \cite[Theorem~1.1]{Bir12} (see also \cite[Theorem~1.1]{Has19}), and the existence of good minimal models in families \cite[Theorem~1.1]{HX13} (see also \cite[Theorem~1.2]{Has19}).  For NQC generalized pairs, the analogues of these theorems were proved in \cite[Theorems~1.1 and~1.3]{LX23}.  In the non-NQC setting, we establish the following analogues of \cite[Theorem~1.1]{Bir12} and \cite[Theorem~1.1]{HX13}.

\begin{thm}\label{thm: bir12 1.1 nonnqc-g}
Let $(X,B,\Mm)/U$ be a generalized pair and let $A\geq 0$ be an $\Rr$-divisor such that $(X,B+A,\Mm)/U$ is lc and
$K_X+B+A+\Mm_X\sim_{\mathbb R,U}0.$

Then $(X,B,\Mm)/U$ has a good log minimal model or a log Mori fiber space.
\end{thm}

\begin{thm}\label{thm: hx13 1.1 nonnqc-g}
Let $(X,B,\Mm)/U$ be an LD generalized pair (cf.\ Definition \ref{defn: ld gpair}), $U^0\subset U$ a non-empty open subset, and set $(X^0,B^0,\Mm^0):=(X,B,\Mm)\times_UU^0$. Assume that
\begin{enumerate}
\item $(X^0,B^0,\Mm^0)/U^0$ has a good minimal model, and
\item any lc center of $(X,B,\Mm)$ intersects $X^0$.
\end{enumerate}
Then $(X,B,\Mm)/U$ has a good log minimal model.
\end{thm}

Theorem~\ref{thm: bir12 1.1 nonnqc-g} is a direct analogue of \cite[Theorem~1.1]{Bir12} for generalized pairs.  In Theorem~\ref{thm: hx13 1.1 nonnqc-g} we impose an additional condition ``LD'', standing for ``linearly decomposable''.  Conceptually, the LD condition provides a workable replacement for rational decompositions in the absence of the NQC hypothesis: it requires that the log canonical class varies in a rational affine subspace and the generalized pair remains lc on a neighborhood, so that one can still obtain controlled convex decompositions into $\mathbb Q$-classes. 

Most importantly, any lc generalized pair polarized by an ample $\mathbb R$-divisor admits an LD structure up to $\mathbb R$-linear equivalence (Lemma~\ref{lem: +ample is ld}). Using this input, we obtain the following structural theorem for $\mathbb Q$-factorial lc generalized pairs.

\begin{thm}\label{thm: structure of q-factorial gpair}
    Let $(X,B,\Mm)/U$ be a $\mathbb Q$-factorial  lc generalized pair and $A$ an ample$/U$ $\mathbb R$-divisor on $X$. Then there exists an lc pair $(X,\Delta)$ such that
    $$K_X+\Delta\sim_{\mathbb R,U}K_X+B+\Mm_X+A.$$
\end{thm}
See Theorem~\ref{thm: structure of mx rcartier gpair} for a more general statement.  Combining Theorem~\ref{thm: structure of q-factorial gpair} with Prokhorov--Shokurov's $\bb$-semi-ampleness theorem \cite{BFMT25}, we obtain the following corollary, which can be viewed as a generalized pair analogue of \cite[Theorem~1.6]{HH20}.

\begin{cor}\label{cor: potentially lc base}
    Let $(X,B,\Mm)/U$ be a $\mathbb Q$-factorial lc generalized pair and $f: X\rightarrow Z$ a contraction$/U$ such that $-(K_X+B+\Mm_X)$ is ample$/Z$. Then $Z$ is potentially lc, i.e.\ there exists an lc pair $(Z,\Delta_Z)$.
\end{cor}

\begin{rem}
It is natural to ask whether our main results remain valid for projective morphisms between compact K\"ahler varieties (or complex analytic spaces) when one assumes that $\Mm$ is only a nef $\bb$-$(1,1)$-class instead of a nef $\bb$-divisor.

There is evidence in this direction.  On the one hand, the existence of the minimal model program for compact K\"ahler generalized pairs with $\mathbb Q$-factorial projective klt ambient variety was proved in \cite{DH24}, relying on the Graf--Kirschner decomposition theorem \cite[Proposition~4.2]{GK20}.  On the other hand, the minimal model program for log canonical \emph{pairs} is now available for projective morphisms between complex analytic spaces (cf.\ \cite{EH24,EH25,Fuj24a,Fuj24b,Fuj25,Has25b,EH26}), building on the klt case (cf.\ \cite{Fuj22,LM22,DHP24}). It seems plausible that combining these analytic developments with the methods of the present paper may lead to an existence theorem for the minimal model program of log canonical generalized pairs associated with projective morphisms between compact K\"ahler varieties (or complex analytic spaces). We also note that adjunction and the canonical bundle formula in the K\"ahler generalized pair setting are available; see \cite{HP24}. We do not pursue this direction here.

A more ambitious question is the existence of the (generalized) log canonical minimal model program for compact K\"ahler varieties, i.e.\ the transcendental minimal model program.  There has been substantial progress in the klt case in dimension $\leq 4$ (cf.\ \cite{DH23,DHY23,DHP24}), but the log canonical case appears to remain widely open.  In higher dimensions, the transcendental minimal model program seems very difficult, although the recent breakthrough of Ou \cite{Ou25} suggests new possibilities.
\end{rem}

\subsection*{Sketch of the proof}

We focus on Theorem~\ref{thm: flip nonnqc-g}.  Recall that the NQC case of Theorem~\ref{thm: flip nonnqc-g} was proved in \cite[Theorem~1.2]{LX23}, and our goal is to extend the strategy of \cite{LX23} beyond the NQC and $\mathbb Q$-factorial settings.  The proof of \cite[Theorem~1.2]{LX23} can be organized into three steps:
\begin{enumerate}
    \item Reduce the existence of flips to a question on the existence of good minimal models (in our notation, reduce Theorem~\ref{thm: flip nonnqc-g} to Theorem~\ref{thm: bir12 1.1 nonnqc-g}).
    \item Under the assumptions of Theorem~\ref{thm: bir12 1.1 nonnqc-g}, construct a (log) minimal model.
    \item Show that the (log) minimal model obtained in (2) is in fact good.
\end{enumerate}
The reduction in (1) is standard, so the main input lies in (2) and (3).  In the NQC setting, (2) was carried out in \cite{LX25}, building on the MMP for log abundant generalized pairs developed in \cite{Has22a}, while (3) was achieved in \cite{LX23} using Koll\'ar's gluing theory \cite{Kol13}.  We follow the same outline in the non-NQC case, but several essential obstacles appear.

\smallskip

\noindent\emph{(I) Lack of rational decompositions.}
The arguments of \cite{LX23} use rational decompositions of generalized pairs to apply gluing theory and certain finite generation statements (cf.\ \cite[Theorem~1.5]{LX23}).  Here the NQC condition is crucial: NQC generalized pairs admit rational decompositions
\[
(X,B,\Mm)=\sum a_i (X,B_i,\Mm_i),
\]
where $a_i\in(0,1]$, $\sum a_i=1$, and each $(X,B_i,\Mm_i)$ is a $\mathbb Q$-generalized pair (cf.\ \cite[Proposition~3.20]{HL22}).  Such decompositions fail in general for non-NQC generalized pairs.

\smallskip

\noindent\emph{(II) Special termination.}
Special termination is a key ingredient in \cite{LX25}.  This is because special termination is formulated in terms of difficulty functions, and in the non-NQC setting there is no satisfactory notion of ``coefficients of the nef part'' that would make the usual difficulty function well-defined.

\smallskip

\noindent\emph{(III) ACC-type inputs.}
For non-NQC generalized pairs, we do not have the ACC for lc thresholds or the global ACC (cf.\ \cite{HMX14,BZ16}) in the form needed for the abundant-case arguments.  In the NQC case, neither \cite{LX23} nor \cite{LX25} invokes ACC statements explicitly, but \cite[Theorem~7.1]{LX25} relies on \cite[Theorem~3.5]{Has22a}, whose proof uses ACC and the global ACC for NQC generalized pairs in an essential way.  As a result, the corresponding steps of \cite{LX25} do not directly extend to the non-NQC case.

\smallskip

To overcome these issues, we introduce several new ideas.

\smallskip

\noindent\emph{(a) Linear decompositions of the log canonical class.}
Even when $(X,B,\Mm)$ itself does not admit a rational decomposition, the log canonical $\mathbb R$-divisor $K_X+B+\Mm_X$ can still satisfy useful arithmetic constraints, especially after adding an ample $\Rr$-divisor $A$.  In many places it is enough to control the class $K_X+B+A+\Mm_X$ (rather than decomposing both $B$ and $\Mm$) so as to obtain convex decompositions into $\mathbb Q$-classes.  Motivated by this, we introduce the class of \emph{linearly decomposable} (\emph{LD}) generalized pairs (Section~\ref{sec: ld gpair}) and establish its basic properties.

For obstacle (I), since subadjunction and related adjunction tools for non-NQC generalized pairs are available in \cite{CHLX23}, we can establish a Koll\'ar-type gluing theory for LD generalized pairs (Theorem~\ref{thm: bir12 1.7 g LD}). This linear decomposition serves as a substitute for the NQC decomposition used in \cite[Proposition~3.20]{HL22}.  Moreover, the finite generation arguments in \cite{LX23} are only used in \cite[Proof of Theorem~5.1, Step~2]{LX23} to reduce to the general type case.  The gluing arguments in \cite[Proof of Theorem~5.1, Step~3]{LX23} do not require general type; the only point where ``general type'' is used is \cite[Lemma~3.5]{LX23}, and we replace it with an argument analogous to \cite[Theorem~1.7]{Bir12}.  See Theorem~\ref{thm: bir12 1.7 g} below.

\smallskip

\noindent\emph{(b) Special termination via codimension $2$ control.}
For obstacle (II), the key observation is that the invariants entering special termination are codimension $2$ in nature.  Concretely, if we run an lc generalized pair MMP on a $\mathbb Q$-factorial klt variety $X$, the relevant coefficients arise from surface plt points.  For a surface $\epsilon$-plt pair $(X,S)$, local Cartier indices are bounded by $\left\lceil\frac{1}{\epsilon}\right\rceil$ (cf.\ \cite[16.6 Proposition]{Kol+92}), and the $\epsilon$-plt condition is preserved under the MMP.  On the other hand, when $K_X+B+\Mm_X$ is a $\mathbb Q$-Cartier $\mathbb Q$-divisor (as happens for LD generalized pairs after suitable perturbations), its Weil index is preserved along the MMP, hence Cartier indices at codimension $2$ points are uniformly bounded. 
This allows us to define an appropriate difficulty function for LD generalized pairs and prove special termination (Section~\ref{sec: spe tof}).  Unlike the classical difficulty function \cite{Fuj07}, our difficulty depends on the given generalized pair and on an induction hypothesis. We refer the reader to \cite[Subsection 4.2]{Hu25} for an embryonic form of this type of special termination.

\smallskip

\noindent\emph{(c) Avoiding ACC for lc thresholds.}
For obstacle (III), we avoid ACC-type arguments altogether: instead we follow the approach of \cite{Has19}, where the pair versions of Theorems~\ref{thm: bir12 1.1 nonnqc-g} and~\ref{thm: hx13 1.1 nonnqc-g} are proved without ACC inputs.  Similar methods also appear in an earlier arXiv version of \cite{HL23} and later in \cite{Cas+25a}.  In our setting, together with (a) and (b), this yields Theorem~\ref{thm: hx13 1.1 nonnqc-g} and the LD case of Theorem~\ref{thm: bir12 1.1 nonnqc-g}.

It is also worth comparing the role of ``(log) abundant generalized pairs" in this paper and their roles in literature.  Besides relying on ACC-type results, the approach in \cite{Has22a,LX25} uses detailed properties of the MMP for (log) abundant NQC generalized pairs, with ideas that originate from the pair case, notably Hashizume's work \cite{Has24} (see also \cite{Hu17,HH20}).  In the pair case, log abundance is crucial because nef and log abundant lc pairs admit good minimal models \cite{FG14,HX16}.  However, counterexamples already appear in the NQC generalized setting (cf.\ \cite{JLX25}), and the non-NQC case behaves even worse; accordingly, only a limited portion of these abundant-case techniques applies in our argument.  See Section~\ref{sec: abundant gmm} for further discussion.

\smallskip

Finally, the general case of Theorem~\ref{thm: bir12 1.1 nonnqc-g} follows from the LD case by a standard reduction, after replacing $\Mm$ so that $K_X+B+A+\Mm_X\sim_{\Rr,U}0$.  This yields Theorem~\ref{thm: flip nonnqc-g}; combined with the cone and contraction theorems \cite[Theorem~B]{CHLX23} it gives Theorem~\ref{thm: eommp nonnqc-g}.

We also use Theorem~\ref{thm: hx13 1.1 nonnqc-g} to deduce the non-NQC analogue of \cite[Lemma~5.9]{HL23} and then follow the arguments of \cite{HL23} to obtain Theorem~\ref{thm: structure of q-factorial gpair}.  Finally, Corollary~\ref{cor: potentially lc base} is an immediate consequence of Theorem~\ref{thm: structure of q-factorial gpair} together with \cite[Theorem~7.6 and~7.2.3]{BFMT25}.
\subsection*{Structure of the paper}
In Section~\ref{sec: preliminary} we recall preliminary material.  In Section~\ref{sec: ld gpair} we introduce LD generalized pairs and establish their basic properties.  In Sections~\ref{sec: models} and~\ref{sec: mmp with scaling} we discuss models and the behavior of LD generalized pairs under the MMP with scaling.  In Section~\ref{sec: spe tof} we prove special termination for LD generalized pairs.  In Section~\ref{sec: abundant gmm} we study the MMP for generalized pairs with abundant canonical class.  In Section~\ref{sec: glue} we establish Koll\'ar-type gluing theory for LD generalized pairs.  In Section~\ref{sec: hx13 1.1} we prove Theorem~\ref{thm: hx13 1.1 nonnqc-g}.  In Section~\ref{sec: eommp} we prove Theorems~\ref{thm: flip nonnqc-g}, \ref{thm: eommp nonnqc-g}, and \ref{thm: bir12 1.1 nonnqc-g}.  Finally, in Section~\ref{sec: structure theorem} we prove Theorem~\ref{thm: structure of q-factorial gpair} and Corollary~\ref{cor: potentially lc base}.

\subsection*{Acknowledgments}

This work is supported by the National Key R\&D Program of China \#2024YFA1014400. The first author is supported by Overseas High-Level Young Talent Recruitment Programs and partially supported by Chongqing Natural Science Foundation Innovation and Development Joint Fund CSTB2023NSCQ-LZX0031. The authors would like to thank Professor Caucher Birkar, Paolo Cascini, Christopher D. Hacon, Kenta Hashizume, Chen Jiang, Wenhao Ou, Calum Spicer, and Zheng Xu for useful discussions.
\section{Preliminaries}\label{sec: preliminary}

We adopt the standard notation and terminology for the minimal model program from \cite{Sho92,KM98,BCHM10} and use them freely. For generalized pairs, we follow the notation and definitions in \cite{BZ16,HL23} with minor differences. To simplify notation, we sometimes write a generalized pair in the form $\Aa$. This convention was introduced in \cite{Cas+25a}, and our notation is consistent with \cite{Cas+25a}. See Notation \ref{nota: simple notation} below.

\subsection{Special notation}

\begin{nota}
Let $X\rightarrow U$ be a projective morphism between normal quasi-projective varieties. A \emph{contraction}$/U$ $f: X\rightarrow Y$ is a projective morphism$/U$ such that $f_\ast \mathcal{O}_X=\mathcal{O}_Y$. For any birational map$/U$ $h: X\dashrightarrow X'$, we denote by $\Exc(h)$ the reduced divisor supported on the codimension one part of the exceptional locus of $h$.
\end{nota}

\begin{defn}
    Let $\bm{v}=(v_1,\dots,v_m)\in\mathbb R^m$ be a vector. The \emph{rational envelope} of $\bm{v}$ is the minimal $\mathbb Q$-affine subspace $V$ in $\mathbb R^m$ such that $\bm{v}\in V$. For example, if $\bm{v}=(\sqrt{2},1-\sqrt{2})\in\mathbb R_{xy}^2$, then the rational envelope of $\bm{v}$ is $(x+y=1)$.
\end{defn}

\begin{defn}
    Let $X\rightarrow U$ be a projective morphism from a normal quasi-projective variety to a variety.  Let $D$ be an $\Rr$-Cartier $\Rr$-divisor on $X$ and $\phi: X\dashrightarrow X'$ a birational map$/U$. We say that $\phi$ is $D$-negative (resp. $D$-non-positive, $D$-trivial, $D$-non-negative, $D$-positive) if $\phi$ does not extract any divisor, $D':=\phi_\ast D$ is $\Rr$-Cartier, and for any  resolution of indeterminacy $p: W\rightarrow X$ and $q: W\rightarrow X'$ of $\phi$, we may write
    $$p^\ast D=q^\ast D'+F$$
    where $F\geq 0$  and $\Supp p_\ast F=\Exc(\phi)$ (resp. $F\geq 0$, $F=0$, $F\leq 0$, $F\leq 0$ and $\Supp p_\ast F=\Exc(\phi)$).
\end{defn}

\begin{defn}\label{defn: irr}
    Let $\Ii\subset\mathbb R$ be a set. Let $X$ be a normal variety, and let $D$ be an $\mathbb R$-divisor on $X$. Write $D=\sum d_iD_i$ where $D_i$ are the irreducible components of $D$. We denote by $D^{\irr}:=\sum_{d_i\not\in\mathbb Q}d_iD_i$. We denote by $||D||_{\infty}:=\sup\{0,|d_i|\}$. We denote by $D\in\Ii$ if $d_i\in\Ii$ for any $i$.
\end{defn}

\subsection{Generalized pairs}

\begin{defn}[$\bb$-divisors]
We use the notation as in \cite[Definition 2.4]{HL23}. Let $X$ be a normal quasi-projective variety. A $\bb$-divisor $\Mm$ on $X$ is a (possibly infinite) $\mathbb R$-linear combination of divisorial valuations $\nu_P$ over $X$
$$\Mm=\sum_Pr_P\nu_P$$
such that for any birational map $\phi: X\dashrightarrow Y$, 
$$\{P\mid \Center_{Y}P\text{ is a divisor}, r_P\not=0\}$$
is a finite set. By definition, $\Mm$ is also a $\bb$-divisor on $Y$. We say that $\Mm$ is a $\mathbb Q$-$\bb$-divisor if $r_P\in\mathbb Q$ for any prime divisor $P$ over $X$. For any birational map $\phi: X\dashrightarrow Y$, we denote by
$$\Mm_{Y}:=\sum_{P\mid \Center_{Y}P\text{ is a divisor}}r_P\left(\Center_{Y}P\right).$$
We say that $\Mm$ \emph{descends to} $X$ if $\Mm_{X}$ is $\mathbb R$-Cartier, and for any projective birational morphism $h: X'\rightarrow X$, we have $\Mm_{X'}=h^*\Mm_X$. For any $\mathbb R$-Cartier $\mathbb R$-divisor $D$ on $X$, we denote by $\overline{D}$ the $\bb$-divisor $\Mm$ such that $\Mm$ descends to $X$ and $\Mm_X=D$. We denote by $\bm{0}:=\overline{0}$. We say that $\Mm$ is \emph{$\bb$-Cartier} if there exists a birational map $\phi: X\dashrightarrow Y$ such that $\Mm$ descends to $Y$ and $\Mm_Y$ is Cartier. 

For any projective morphism $\pi: X\rightarrow U$ to a normal quasi-projective variety $U$, we say that $\Mm$ is \emph{nef$/U$} if there exists a birational map$/U$ $\phi: X\dashrightarrow Y$ such that $\Mm$ descends to $Y$ and $\Mm_Y$ is nef$/U$, and if $U$ is a closed point, then we say that $\Mm$ is \emph{nef}. For any non-negative real number $s$, we say that $\Mm$ is \emph{$s$-NQC}$/U$ if $\Mm=\sum \mu_i\Mm_i$, where each $\mu_i\geq s$ and each $\Mm_i$ is a nef$/U$ $\bb$-Cartier $\bb$-divisor. For any $\mathbb R$-Cartier $\mathbb R$-divisor $D$ on $X$, we say that $D$ is $s$-NQC$/U$ if $D=\sum d_iD_i$, where each $d_i\geq s$ and each $D_i$ is a nef$/U$ Cartier divisor. If $\Mm$ (resp. $D$) is $0$-NQC$/U$, then we say that $\Mm$ (resp. $D$) is NQC$/U$.
\end{defn}

\begin{defn}\label{defn: restriction b divisor}
We will use two types of restrictions of $\bb$-divisors in this paper. Let $X$ be a normal quasi-projective variety and $\Dd$ a $\bb$-divisor on $X$.
\begin{enumerate}
    \item Let $V$ be a non-empty open subset of $X$.  We define the \emph{restricted $\bb$-divisor} of $\Dd$ on $V$, which is denoted by $\Dd|_{V}$, in the following way. 

    For any birational morphism $\pi: W\to V$, there exists a birational morphism $\pi': Y\rightarrow X$ such that $W\subset Y$ and $\pi'|_W=\pi$. We let $(\Dd|_V)_{W}=(\Dd_Y)|_W$. It is easy to see that this definition is independent of the choice of $Y$ and defines a $\bb$-divisor.
    \item Suppose that $\Dd$ descends to a birational model of $X$. Let $S$ be a prime divisor on $X$ and $\nu: S^\nu\rightarrow S$ the normalization of $S$. The \emph{restricted $\bb$-divisor} of $\Dd$ on $S^\nu$, which is denoted by $\Dd|_{S^\nu}$, is defined in the following way. 

     Let $f: Y\rightarrow X$ be a log resolution of $(X,S)$ such that $\Dd$ descends to $Y$. Let $S_Y:=f^{-1}_*S$. Then there exists an induced birational morphism $f_S: S_Y\rightarrow S^\nu$ such that $\nu\circ f_S=f|_{S_Y}$.
     We define $$\Dd|_{S^\nu}:=\overline{\Dd_Y|_{S_Y}}.$$
     It is clear that $\Dd|_{S^\nu}$ is well-defined and is independent of the choice of $Y$.
\end{enumerate}
\end{defn}

\begin{defn}[Generalized pairs]
A \emph{generalized pair} $(X,B,\Mm)/U$ is the datum of a normal quasi-projective variety $X$, a projective morphism $X\rightarrow U$ to a normal quasi-projective variety $U$, an $\Rr$-divisor $B\geq 0$ on $X$, and a nef$/U$ $\bb$-divisor $\Mm$, such that $K_X+B+\Mm_X$ is $\Rr$-Cartier.  We say that $(X,B,\Mm)/U$ is \emph{NQC} if $\Mm$ is NQC$/U$. We say that $(X,B)/U$ is a \emph{pair} if $\Mm=\bm{0}$.

If $B=0$, or if $\Mm=\bm{0}$, or if $U$ is clear from context, then we may drop $B$, $\Mm$, or $U$, respectively. If $U=\{pt\}$ then we also drop $U$ and say that $(X,B,\Mm)$ is \emph{projective}. If we allow $B$ to have negative coefficients, then we shall add the prefix ``sub-". If $B$ is a $\Qq$-divisor and $\Mm$ is a $\Qq$-$\bb$-divisor, then we shall add the prefix ``$\Qq$-".
\end{defn}

We adopt the following notation for simplicity of the arguments. Such notation was first introduced in \cite{Cas+25a} due to the lengthy writing of the canonical $\mathbb R$-divisor.

\begin{nota}\label{nota: simple notation}
Let $(X,B,\Mm)/U$ be a sub-generalized pair. In many scenarios, the notation $(X,B,\Mm)/U$ is rather inconvenient and makes arguments hard to follow. At times, we do not need to rely on the full information that the (sub-)generalized pair conveys. In view of this, throughout the paper we will usually denote (sub-)generalized pairs in the form of ``$\Aa:=(X,B,\Mm)$" or simply say that ``$\Aa/U$ is a (sub-)generalized pair" without having to fully recall $X,B,\Mm$. 

In the following, we let $\Aa/U=(X,B,\Mm)/U$ be a sub-generalized pair.

For any $\Rr$-divisor $D$ on $X$ and nef$/U$ $\bb$-divisor $\Nn$ on $X$ such that $D+\Nn_X$ is $\Rr$-Cartier, we define $(\Aa,D,\Nn):=(X,B+D,\Mm+\Nn)$. If $D=0$ then we may drop $D$, and if $\Nn=\bm{0}$ then we may drop $\Nn$. We define
$$K_{\Aa}:=K_{(X,B,\Mm)}:=K_X+B+\Mm_X$$
to be the \emph{canonical $\Rr$-divisor} of $\Aa$. 
Moreover, 
$X,B,\Mm$ are called the \emph{ambient variety}, \emph{boundary part}, and \emph{nef part} (or \emph{moduli part}) of $\Aa$ respectively. The $\mathbb R$-divisor
$B+\Mm_X$ is called the \emph{generalized boundary} of $\Aa$. We say that $\Aa/U$ is \emph{of general type} if $K_{\Aa}$ is big$/U$. 
\end{nota}

\begin{defn}
Let $\Aa/U:=(X,B,\Mm)/U$ be a sub-generalized pair. For any birational map$/U$ $\phi: X\dashrightarrow X'$, we define $\phi_*\Aa:=(X',\phi_*B,\Mm)$ and say that $\phi_*\Aa$ is the \emph{image} of $\Aa$ on $X'$. For any projective birational morphism $h: X'\rightarrow X$, we define 
$h^*\Aa:=(X',B',\Mm)$
where $B'$ is the unique $\Rr$-divisor such that $K_{h^*\Aa}=h^*K_{\Aa}$. For any prime divisor $E$ on $X'$, we denote by
$$a(E,\Aa):=-\mult_EB'$$
the \emph{discrepancy} of $E$ with respect to $\Aa$. We say that $\Aa$ is \emph{sub-lc} (resp. \emph{sub-klt}) if $a(E,\Aa)\geq -1$ (resp. $>-1$) for any prime divisor over $X$. If $B\geq 0$ and $\Aa$ is sub-lc (resp. sub-klt), then we say that $\Aa$ is lc (resp. klt). We say that $X$ is \emph{potentially klt} if $(X,\Delta)$ is klt for some $\Delta$.

An \emph{nklt place} (resp. \emph{lc place}, \emph{nlc place}) of $\Aa$ is a prime divisor $E$ over $X$ such that $a(E,\Aa)\leq -1$ (resp. $a(E,\Aa)=-1$, $a(E,\Aa)<-1$). An \emph{nklt center} (resp. \emph{non-trivial lc center}, \emph{nlc center}) of $\Aa$ is the image of an nklt place (resp. \emph{lc place}, \emph{nlc place}) of $\Aa$ on $X$. An \emph{lc center} of $\Aa$ is either a non-trivial lc center of $\Aa$ or $X$ itself. $X$ is called \emph{the trivial lc center} of $\Aa$. The union of all nklt centers (resp. all nlc centers) is called the \emph{nklt locus} (resp. \emph{nlc locus}) of $\Aa$ and is denoted by $\Nklt(\Aa)$ (resp. $\Nlc(\Aa)$).

We say that $\Aa$ is \emph{dlt} if $\Aa$ is lc, and for any lc center $V$ of $\Aa$ with generic point $\eta_V$, $(X,B)$ is log smooth near $\eta_V$, and $\Mm$ descends to $X$ near $\eta_V$. We say that $\Aa$ is \emph{qdlt} if there exists an open subset $V\subset X$ such that
\begin{itemize}
    \item $(V,\Supp B|_V)$ is $\Qq$-factorial log toroidal (cf. \cite[Definition 2.1]{ACSS21}), and
    \item $V$ contains the generic point of any lc center of $\Aa$, and the generic point of any lc center of $\Aa$ is the generic point of an lc center of $(V,B|_V)$.
\end{itemize}
By \cite[Lemma 7.1.2]{CHLX23}, $\Aa$ is qdlt if and only if for any lc center of $\Aa$ with generic point $\eta$, $\Aa$ is $\Qq$-factorial toroidal near $\eta$.

A \emph{log resolution} $h: Y\rightarrow X$ of $\Aa$ is a log resolution of $(X,\Supp B)$ such that $\Mm$ descends to $Y$.
\end{defn}

\begin{defn}
    Let $\Aa/U:=(X,B,\Mm)/U$ be a sub-generalized pair. For any non-negative real number $s$, we say that $\Aa/U$ is \emph{$s$-KNQC} if $K_{\Aa}$ is $s$-NQC$/U$. We say that $\Aa/U$ is KNQC if $K_{\Aa}$ is NQC$/U$.
\end{defn}

\begin{defn}
Let $\Aa/U:=(X,B,\Mm)/U$ be a sub-generalized pair. For any non-empty open subset $U^0\subset U$ with $X^0:=X\times_UU^0$, we denote by 
$$\Aa\times_UU^0:=\Aa|_{X^0}:=(X^0,B^0,\Mm^0):=(X,B,\Mm)\times_UU^0:=(X^0,B|_{X^0},\Mm|_{X^0}).$$
Further assume that $B\geq 0$ and $\widetilde{S}$ is an irreducible component of $B^{=1}$ with normalization $S$. By \cite[Definition 4.7]{BZ16}, there exists a generalized pair $(S,B_S,\Mm^S)/U$ induced by adjunction
$$K_S+B_S+\Mm^S_S=(K_X+B+\Mm_X)|_S$$
We denote by $\Aa|_S:=(S,B_S,\Mm^S)$. By \cite[Remark 4.8]{BZ16}, if $\Aa$ is lc, then $\Aa|_S$ is lc. By \cite[Lemma 2.9]{HL22}, if $\Aa$ is dlt, then $\Aa|_S$ is dlt. 

If $\Aa$ is dlt and $V$ is a dimension $\geq 1$ non-trivial lc center of $\Aa$, then $V$ is an irreducible component of $\cap_{i=1}^k S_i$ where $S_i$ are irreducible components of $\lfloor B\rfloor$ and each $S_i$ is normal by \cite[Lemma 2.6]{HL22}. Then for any $1\leq j\leq k$, there exists an irreducible component $V_j$ of $X\cap\bigcap_{i=1}^jS_i$ such that
$$V_j\subset V_{j-1},\quad  1\leq j\leq k,\quad  V_k=V,\quad V_0=X.$$
By \cite[Lemma 2.9, Proposition 2.10]{HL22}, we may inductively define $\Aa_{V_j}:=\Aa_{V_{j-1}}|_{V_j}$ so that $\Aa_{V_j}$ is dlt for any $j$. We denote by $\Aa_V:=\Aa_{V_k}:=\Aa|_V$.
\end{defn}

\begin{defn}
For any sub-generalized pairs $\Aa_i/U=(X,B_i,\Mm_i)/U$ and real numbers $a_i\in [0,1]$ such that $\sum a_i=1$, we denote by
$$\sum a_i\Aa_i:=\left(X,\sum a_iB_i,\sum a_i\Mm_i\right).$$  
\end{defn}

\begin{deflem}
Let $\Aa/U$ be an lc generalized pair. A \emph{$\mathbb Q$-factorial (q)dlt modification} of $\Aa$ is a projective birational 
morphism $h: X'\rightarrow X$ such that $\Aa':=h^*\Aa$ is $\mathbb Q$-factorial (q)dlt and $h$ only extracts lc places of $\Aa$. Existence of $\mathbb Q$-factorial dlt modification for lc generalized pairs was shown in \cite[Proposition 3.10]{HL22}.
\end{deflem}

The following lemma is well-known. For the reader's convenience, we provide a proof.

\begin{lem}\label{lem: qdlt is potentially klt}
    Let $(X,B,\Mm)/U$ be a qdlt generalized pair. Then $X$ is potentially klt.
\end{lem}
\begin{proof}
Let $V$ be an open subset of $X$ such that
    \begin{itemize}
        \item $(V,\Supp B|_V)$ is $\mathbb Q$-factorial toroidal, and
        \item  $V$ contains the generic point of any lc center of $\Aa$, and the generic point of any lc center of $\Aa$ is the generic point of an lc center of $(V,B|_V)$.
    \end{itemize}
If $V=\emptyset$, then $\Aa$ is klt, hence $X$ is potentially klt by \cite[Lemma 3.4]{HL22}. Therefore, we may assume that $V\not=\emptyset$.
    
    It essentially follows from the same lines of the proof of \cite[Proposition~2.43]{KM98}. Let $S:=\lfloor B\rfloor$. Then there exists a very ample Cartier divisor $H$ on $X$ such that $\mathcal{L}:=\mathcal{O}_X(H+S)$ is generated by global sections. Choose an integer $m\gg 0$ and set $\epsilon:=\frac{1}{m}$. Let $D'$ be a general member of $|\mathcal{L}|$ and set $D:=\frac{1}{m}D'$. Then 
    $$K_X+B+\Mm_X-\epsilon S+\epsilon D$$ 
    is $\mathbb R$-Cartier, hence $(X,B-\epsilon S+\epsilon D,\Mm)$ is a generalized pair. We let $h: X'\rightarrow X$ be a log resolution of $(X,B+\epsilon D,\Mm)$. We let $E_i$ be the $h$-exceptional divisors and let $\eta_i$ be the generic point of $\Center_XE_i$ for each $i$. 
    
    Since there are only finitely many $E_i$, for $0<\epsilon\ll 1$, if $E_i$ is not an lc place of $(X,B,\Mm)$, then $E_i$ is not an nklt place of $(X,B-\epsilon S+\epsilon D,\Mm)$. If $E_i$ is an lc place of $(X,B,\Mm)$, then $\Center_XE_i$ is a stratum of $S$ and $V$ is $\mathbb Q$-factorial klt near $\eta_i$. Since $D'$ is general and $\mathcal{L}$ is globally generated, $\Supp D$ does not contain $\Center_XE_i$. Therefore, 
    $$a(E_i,X,B-\epsilon S+\epsilon D,\Mm)>a(E_i,B,\Mm)=-1.$$
    Therefore, $E_i$ is not an lc place of  $(X,B-\epsilon S+\epsilon D,\Mm)$ for $0<\epsilon\ll 1$. Thus $(X,B-\epsilon S+\epsilon D,\Mm)$ is klt. By \cite[Lemma 3.4]{HL22}, $X$ is potentially klt.
\end{proof}

\subsection{Relative Nakayama-Zariski decomposition}
We need the following definition of the relative Nakayama-Zariski decomposition as in \cite{LX25}. The projective version can be found in \cite{Nak04}.
\begin{defn}
    Let $\pi\colon X\rightarrow U$ be a projective morphism between normal quasi-projective varieties, $D$ a pseudo-effective$/U$ $\Rr$-Cartier $\Rr$-divisor on $X$, and $P$ a prime divisor on $X$. We define $\sigma_{P}(X/U,D)$ as in \cite[Definition 3.1]{LX25} by considering $\sigma_{P}(X/U,D)$ as a number in  $[0,+\infty)\cup\{+\infty\}$. We define $$N_{\sigma}(X/U,D)=\sum_Q\sigma_Q(X/U,D)Q$$
    where the sum runs through all prime divisors on $X$ and consider it as a formal sum of divisors with coefficients in $[0,+\infty)\cup\{+\infty\}$. 
    
    We say that $D$ is \emph{movable$/U$} if $N_{\sigma}(X/U,D)=0$.
\end{defn}

\subsection{Perturb generalized pair to non-lc pair}

We need the following lemma, which allows us to perturb an lc generalized pair to a (possibly) non-lc pair with well-controlled non-klt locus.

\begin{lem}\label{lem: generalized lc to non-lc pair}
Let $\Aa/U=(X,B,\Mm)/U$ be a $\mathbb Q$-factorial qdlt generalized pair. Set $S:=\lfloor B\rfloor$. 
\begin{enumerate}
    \item If $B+\Mm_X$ is big$/U$, then there exists a pair $(X,\Delta)$ and an ample$/U$ $\Rr$-divisor $A$, such that $K_X+\Delta+A\sim_{\mathbb R,U}K_{\Aa}$, and $$\Nlc(X,\Delta)=\Nklt(X,\Delta)=\Nklt(\Aa)=S.$$
    \item If $K_{\Aa}$ is big$/U$, then there exists a pair $(X,\Delta)$, an ample$/U$ $\Rr$-divisor $A$, and a real number $\mu>1$, such that $K_X+\Delta+A\sim_{\mathbb R,U}\mu K_{\Aa}$, and $$\Nlc(X,\Delta)=\Nklt(X,\Delta)=\Nklt(\Aa)=S.$$
\end{enumerate} 
\end{lem}
\begin{proof}
    (1) Write $B+\Mm_X=2H+E$ where $H$ is ample$/U$ and $E\geq 0$. Then there exists $\epsilon>0$ such that $(X,(1-\epsilon)B-S+\epsilon E,(1-\epsilon)\Mm)$ is klt. By \cite[Lemma 3.4]{HL22} there exists a klt pair $(X,\Delta_0)$ such that
    $$K_X+\Delta_0\sim_{\mathbb R,U}K_X+(1-\epsilon)B-S+\epsilon E+(1-\epsilon)\Mm_X+\epsilon H.$$
    We let $0<\delta\ll\epsilon$ be a real number such that $A:=\epsilon H-\delta S$ is ample and let 
    $$\Delta:=\Delta_0+(1+\delta)S,$$
    then $(X,\Delta)$ and $A$ satisfy our requirements.

    (2) Write $K_{\Aa}=2H+E$ where $H$ is ample$/U$ and $E\geq 0$. Then there exists $\epsilon>0$ such that $(X,B-S+\epsilon E,\Mm)$ is klt. By \cite[Lemma 3.4]{HL22} there exists a klt pair $(X,\Delta_0)$ such that
    $$K_X+\Delta_0\sim_{\mathbb R,U}K_X+B-S+\epsilon E+\Mm_X+\epsilon H.$$
    We let $0<\delta\ll\epsilon$ be a real number such that $A:=\epsilon H-\delta S$ is ample and let 
    $$\Delta:=\Delta_0+(1+\delta)S,$$
    then $\mu:=1+\epsilon$, $(X,\Delta)$, and $A$ satisfy our requirements.
\end{proof}

\begin{lem}\label{lem: reduction to Nlc locus}
Let $\Aa/U=(X,B,\Mm)/U$ be a $\Qq$-factorial dlt generalized pair with $S:=\lfloor B\rfloor$. Assume that $K_{\Aa}$ is nef$/U$, either $K_{\Aa}$ or $B+\Mm_X$ is big$/U$, and $K_{\Aa}|_S$ is semi-ample$/U$. Then $K_{\Aa}$ is semi-ample$/U$.
\end{lem}
\begin{proof}
By Lemma \ref{lem: generalized lc to non-lc pair}, there exists a pair $(X,\Delta)/U$, an ample$/U$ $\mathbb R$-divisor $A$, and a real number $\mu\geq 1$, such that $K_X+\Delta+A\sim_{\mathbb R,U}\mu K_{\Aa}$ and $\Nlc(X,\Delta)=S$. By \cite[Theorem 2.23]{Has25a}, $K_X+\Delta+A$ is semi-ample$/U$, so $L$ is semi-ample$/U$.
\end{proof}

\subsection{Iitaka dimensions}

We refer the reader to \cite[Section 2]{Cho08}, \cite[Definition 2.3]{Hu20}, \cite[Definition 2.6]{HH20} for the definition of invariant Iitaka dimensions. We need the following lemma on basic properties of numerical dimensions and invariant Iitaka dimensions. 

\begin{lem}[cf. {\cite[Lemma 2.3]{LX25}}]\label{lem: property of numerical and Iitaka dimension} Let $\pi: X\rightarrow U$ be a projective morphism between normal quasi-projective varieties and $D$ an $\Rr$-Cartier $\Rr$-divisor on $X$. Then:
\begin{enumerate}
    \item $D$ is big$/U$ if and only if $\kappa_{\sigma}(X/U,D)=\dim X-\dim U$.
    \item Let $D_1,D_2$ be two $\Rr$-Cartier $\Rr$-divisors on $X$. Suppose that $D_1\sim_{\mathbb R,U}E_1\geq 0$ and $D_2\sim_{\mathbb R,U}E_2\geq 0$ for some $\Rr$-divisors $E_1,E_2$ such that $\Supp E_1=\Supp E_2$. Then $\kappa_{\sigma}(X/U,D_1)=\kappa_{\sigma}(X/U,D_2)$ and $\kappa_{\iota}(X/U,D_1)=\kappa_{\iota}(X/U,D_2)$.
    \item Let $f: Y\rightarrow X$ be a surjective birational morphism and $D_Y$ an $\Rr$-Cartier $\Rr$-divisor on $Y$ such that $D_Y=f^*D+E$ for some $f$-exceptional $\Rr$-divisor $E\geq 0$. Then $\kappa_{\sigma}(Y/U,D_Y)=\kappa_{\sigma}(X/U,D)$ and $\kappa_{\iota}(Y/U,D_Y)=\kappa_{\iota}(X/U,D)$. 
    \item Let $g: Z\rightarrow X$ be a surjective morphism from a normal variety such that $Z$ is projective over $U$. Then $\kappa_{\sigma}(Z/U,g^*D)=\kappa_{\sigma}(X/U,D)$ and $\kappa_{\iota}(Z/U,g^*D)=\kappa_{\iota}(X/U,D)$.
    \item  Let $\phi: X\dashrightarrow X'$ be a sequence of steps of a $D$-MMP$/U$ and let $D':=\phi_*D$. Then $\kappa_{\sigma}(X/U,D)=\kappa_{\sigma}(X'/U,D')$ and $\kappa_{\iota}(X/U,D)=\kappa_{\iota}(X'/U,D')$
\end{enumerate}
\end{lem}

\section{Linearly decomposable generalized pair}\label{sec: ld gpair}

In this section, we introduce the class of \emph{linearly decomposable} (\emph{LD}) generalized pairs and study their basic properties.

\subsection{Definition of LD generalized pair}

\begin{defn}[Linearly decomposable generalized pair]\label{defn: ld gpair}
    Let $\Aa/U:=(X,B,\Mm)/U$ be an lc generalized pair. Write $$K_{\Aa}=K_X+B+\Mm_X=:D=\sum_{i=1}^m c_iD_i$$
    where $D_i\geq 0$ are distinct Weil divisors, $c_i\not=0$ and $D_i\not=0$ for any $i$, and $c_i\leq c_j$ if $i\leq j$. Let $\bm{c}:=(c_1,\dots,c_m)\in\mathbb R^m$ and let $V$ be the rational envelope of $\bm{c}$ in $\mathbb R^m$. Let $$D(\bm{v}):=\sum_{i=1}^mv_iD_i,\quad  B(\bm{v}):=B+D(\bm{v})-D,\quad  \text{and}\quad  \Aa(\bm{v}):=(X,B(\bm{v}),\Mm)$$ for any $\bm{v}=(v_1,\dots,v_m)\in V$.
    
    By our construction, for a fixed choice of $K_X$, $m,\bm{c},V,D(\cdot),B(\cdot)$, $\Aa(\cdot)$ are uniquely determined by the $\mathbb R$-divisor $K_{\Aa}$.

    $\{V,\bm{c},\Aa(\cdot)\}$ is called the \emph{polytopal} of $\Aa$. We say that $\Aa/U$ is \emph{linearly decomposable} (\emph{LD} for short) if there exists an open neighborhood $V_0$ of $\bm{c}$ in $V$, such that for any $\bm{v}\in V_0$, $\Aa(\bm{v})$ is lc. In particular, $B(\bm{v})\geq 0$ for any $\bm{v}\in V_0$. It is clear that the polytopal of $\Aa$ does not depend on $U$, hence LD is a property which does not depend on $U$. Therefore, if $\Aa/U$ is LD, then we also say that $\Aa$ is LD.
\end{defn}

The following equivalent definition of LD generalized pair is useful in many places.

\begin{lem}\label{lem: equivalent definition of LD}
     Let $\Aa/U=(X,B,\Mm)/U$ be a generalized pair. Then the following three conditions are equivalent:
     \begin{enumerate}
         \item $\Aa$ is LD.
         \item There exist $a_1,\dots,a_m\in (0,1]$ that are $\mathbb Q$-linearly independent and lc generalized pairs $\Aa_i/U=(X,B_i,\Mm)/U$ such that each $K_{\Aa_i}$ is a $\mathbb Q$-divisor, $\sum a_i=1$, and $\sum a_i\Aa_i=\Aa$.
        \item There exist $a_1,\dots,a_m\in (0,1]$ and lc generalized pairs $\Aa_i/U=(X,B_i,\Mm)/U$ such that each $K_{\Aa_i}$ is a $\mathbb Q$-divisor, $\sum a_i=1$, and $\sum a_i\Aa_i=\Aa$.
     \end{enumerate}
\end{lem}
\begin{proof}
Let $\{V,\bm{c},\Aa(\cdot)\}$ be the \emph{polytopal} of $\Aa$.

(1)$\Rightarrow$(2): Let $V_0\ni\bm{c}$ be an open subset of $V$ such that $\Aa(\bm{v})$ is lc for any $\bm{v}\in V_0$. Let $m:=\dim V+1$ and let $\bm{v}_1,\dots,\bm{v}_{m}\in V_0$ be $\mathbb Q$-vectors such that $\bm{c}$ is contained in the interior of the convex hull spanned by $\bm{v}_1,\dots,\bm{v}_{m}$. Then there exist unique positive real numbers $a_1,\dots,a_m\in (0,1]$ that are $\mathbb Q$-linearly independent, such that $\sum_{i=1}^ma_i=1$ and $\sum_{i=1}^ma_i\bm{v}_i=\bm{c}$. We may let $B_i:=B(\bm{v}_i)$ for each $i$.

(2)$\Rightarrow$(3): Obvious.

(3)$\Rightarrow$(1): Let $V'$ be the minimal $\mathbb Q$-affine subspace spanned by $K_{\Aa_i}$ in $\Weil_{\mathbb R}(X)$. Then $K_{\Aa}\in V'$, so $V\subset V'$ by the definition of rational envelope. Let $\mathcal{H}$ be the convex hull spanned by $K_{\Aa_i}$  in $V'$ and let $V_0$ be the interior of $\mathcal{H}\cap V$. Then $V_0\ni\bm{c}$ is an open subset of $V$, and for any $\bm{v}\in V_0$, $\Aa(\bm{v})$ is lc. Thus $\Aa$ is LD.
\end{proof}

\subsection{Lc strata of LD generalized pair}

\begin{lem}\label{lem: LD preserves lc center}
  Let $\Aa/U$ be an LD generalized pair with polytopal $\{V,\bm{c},\Aa(\cdot)\}$. Then:
  \begin{enumerate}
\item For any lc place $E$ of $\Aa$, we have $a(E,\Aa(\bm{v}))=-1$ for any $\bm{v}\in V$.
\item Assume that $\Aa$ is lc (resp. dlt, qdlt, klt). Then there exists an open neighborhood $V'\ni\bm{c}$ in $V$, such that $\Aa(\bm{v})$ is lc (resp. dlt, qdlt, klt) and $\Nklt(\Aa(\bm{v}))=\Nklt(\Aa)$  for any $\bm{v}\in V'$.
  \end{enumerate}
\end{lem}
\begin{proof}
Since $\Aa$ is LD, there exists an open neighborhood $V_0\ni\bm{c}$ in $V$ such that $\Aa(\bm{v})$ is lc for any $\bm{v}\in V_0$.

(1) Suppose (1) does not hold. Then there exists $\bm{v}\in V$ such that $a(E,\Aa(\bm{v}))\not=0$. Possibly replacing $\bm{v}$ with $2\bm{c}-\bm{v}$, we may assume that $a(E,\Aa(\bm{v}))<-1$. Thus
    $$a(E,\Aa(\bm{c}+t(\bm{v}-\bm{c})))<-1$$
    for any $t>0$. Since $V_0\ni\bm{c}$ is an open subset of $V$, there exists $0<t\ll 1$ such that $\bm{c}+t(\bm{v}-\bm{c})\in V_0$. Thus $\Aa(\bm{c}+t(\bm{v}-\bm{c}))$ is lc, a contradiction.
    
(2) By (1), we have that $\Nklt(\Aa)\subset\Nklt(\Aa(\bm{v}))$ for any $\bm{v}\in V_0$. We let
    $$V':=\left\{\bm{v}\middle| \bm{v}=\frac{1}{2}(\bm{c}+\bm{u}),\bm{u}\in V_0\right\}.$$
    By linearity of discrepancies and the definitions of lc, dlt, qdlt and klt, $V_0$ satisfies our requirements.
\end{proof}

The next lemma shows that LD is a property that is preserved under $\mathbb Q$-factorial (q)dlt modifications.

\begin{lem}\label{lem: LD under dlt model}
Let $\Aa/U$ be an lc generalized pair and $h: Y\rightarrow X$ a $\mathbb Q$-factorial qdlt modification of $\Aa$. Then $\Aa$ is LD if and only if $\Aa_Y:=h^*\Aa$ is LD.
\end{lem}
\begin{proof}
We write $\Aa=(X,B,\Mm)$ and write
$$K_X+B+\Mm_X=:D=\sum_{i=1}^mc_iD_i$$
where $D_i\geq 0$ are distinct Weil divisors, $c_i\not=0,D_i\not=0$ for each $i$, and $c_i\leq c_j$ if $i\leq j$. Let $E$ be the reduced $h$-exceptional divisor. Write $h^*\Aa=(Y,B_Y,\Mm)$. If there exists an index $i_0$ such that $c_{i_0}=1$, then we define $m':=m$,  $D_i':=h^{-1}_*D_i$ if $i<i_0$ or $i>i_0$, and $D_{i_0}':=h^{-1}_*D_{i_0}+E$. If $c_i\neq 1$ for any $i$, then we let $i_0$ be the index such that $c_{i_0}<1<c_{i_0+1}$ (here we consider $c_{0}:=-\infty$ and $c_{m+1}:=+\infty)$, let $m':=m+1$, and $D_i':=h^{-1}_*D_i$ if $i\leq i_0$, $D_{i_0+1}':=E$, and $D_i':=h^{-1}_*D_{i-1}$ if $i\geq i_0+2$. Then we have
$$K_Y+B_Y+\Mm_Y=h^*D=\sum_{i=1}^{m'}c_iD_i'.$$
Let $\bm{c}:=(c_1,\dots,c_m),\bm{c}':=(c_1,\dots,c_{m'})$, and $V,V'$ the rational envelopes of $\bm{c},\bm{c}'$ in $\mathbb R^m,\mathbb R^{m'}$ respectively. Then there exists a one-to-one correspondence $V\xleftrightarrow{\phi} V'$:
$$(v_1,\dots,v_m)\xleftrightarrow{\phi} (v_1,\dots,v_m) \text{ if }m'=m$$
and
$$(v_1,\dots,v_m)\xleftrightarrow{\phi} (v_1,\dots,v_{i_0},1,v_{i_0+1},\dots,v_m) \text{ if } m'=m+1.$$ 

For any $\bm{v}=(v_1,\dots,v_m)\in V$, $\bm{v}'=(v_1',\dots,v_{m'}')\in V'$, we define $D(\bm{v}):=\sum_{i=1}^mv_iD_i$, $B(\bm{v}):=B+D(\bm{v})-D$, $D_Y(\bm{v}'):=\sum_{i=1}^{m'}v_i'D_i'$, and $B_Y(\bm{v}'):=B_Y+D_Y(\bm{v}')-h^*D$.

By \cite[Lemmas 5.3, 5.4]{HLS24}, we have
$$K_Y+B_Y(\phi(\bm{v}))+\Mm_Y=h^*(K_X+B(\bm{v})+\Mm_X)$$
for any $\bm{v}\in V$, hence $\Aa(\bm{v}):=(X,B(\bm{v}),\Mm)$ is lc if and only if $\Aa_Y(\bm{v}):=(Y,B_Y(\bm{v}),\Mm)$ is lc. We have that $\{V,\bm{c},\Aa(\cdot)\}$ and $\{V',\bm{c}',\Aa_Y(\cdot)\}$ are the polytopal of $\Aa$ and $h^*\Aa$ respectively. Since $\phi$ is a linear isomorphism between finite dimensional affine spaces, the lemma follows.
\end{proof}

The following lemma is crucial to characterize LD generalized pairs. Roughly speaking, it indicates that LD is a property which only relies on the irrational part of the nef part. Recall that for any $\mathbb R$-divisor $D$, $D^{\irr}$ stands for the irrational part of $D$ (see Definition \ref{defn: irr}).

\begin{lem}\label{lem: ld and irrationality}
  Let $\Aa/U:=(X,B,\Mm)/U$ be an lc generalized pair. Then:
  \begin{enumerate}
      \item If $\Aa$ is LD, then $\Supp\Mm_X^{\irr}\subset\Supp\{B\}$.
      \item If $\Aa$ is qdlt and $\Supp\Mm_X^{\irr}\subset\Supp\{B\}$, then $\Aa$ is LD.
  \end{enumerate}
\end{lem}
\begin{proof}
  Let $\{V,\bm{c},\Aa(\cdot)\}$ be the polytopal of $\Aa$ and write $\Aa(\bm{v})=(X,B(\bm{v}),\Mm)$ for any $\bm{v}\in V$.

(1) Let $V_0\ni\bm{c}$ be an open subset of $V$ such that $\Aa(\bm{v})$ is lc for any $\bm{v}\in V_0$. Let $S$ be an  irreducible component of $\Supp\Mm_X^{\irr}$. If $S\not\subset\Supp\{B\}$, then either $S$ is a component of $\lfloor B\rfloor$, or $S$ is not a component of $B$. In particular, $S$ is a component of $(K_X+B+\Mm_X)^{\irr}$. Therefore, there exists $\bm{v}_1,\bm{v}_2\in V_0$ such that $\mult_S(B(\bm{v}_1)-B)>0$ and $\mult_S(B(\bm{v}_2)-B)<0$. If $S$ is a component of $\lfloor B\rfloor$, then $\mult_SB(\bm{v}_1)>1$ so $\Aa(\bm{v}_1)$ is not lc, which is not possible. If $S$ is not a component of $B$, then $\mult_SB(\bm{v}_2)<0$, so  $\Aa(\bm{v}_2)$ is not lc, which is not possible.

(2) By Lemma \ref{lem: qdlt is potentially klt}, possibly replacing $X$ with a small $\mathbb Q$-factorialization, we may assume that $X$ is $\mathbb Q$-factorial klt.  Write
$$K_X+B+\Mm_X=:D=\sum_{i=1}^mc_iD_i$$
where $\bm{c}=(c_1,\dots,c_m)$ and $D_1,\dots,D_m\geq 0$ are distinct Weil divisors. Let $D(\bm{v}):=\sum_{i=1}^mv_iD_i$ for any $\bm{v}\in\mathbb R^m$, then $B(\bm{v})=B+D(\bm{v})-D$ for any $\bm{v}\in V$. 

By our assumption, 
\begin{align*}
  \Supp(B(\bm{v})-B)&=\Supp(D(\bm{v})-D)\subset\Supp(K_X+B+\Mm_X)^{\irr}\\
  &\subset\Supp B^{\irr}\cup\Supp\Mm_X^{\irr}\subset\Supp\{B\}  
\end{align*}
for any $\bm{v}\in V$. Since $(X,B,\Mm)$ is $\mathbb Q$-factorial qdlt, there exists $\epsilon\in (0,1)$ such that $(X,B+\epsilon\Supp\{B\},\Mm)$ is qdlt and $\{B\}\in\{0\}\cup [\epsilon,1)$. Let $V_0\ni\bm{c}$ be an open subset of $V$ such that $||B(\bm{v})-B||_{\infty}<\epsilon^2$ for any $\bm{v}\in V_0$. Then
$$B+\epsilon\Supp\{B\}\geq B(\bm{v})\geq B-\epsilon\Supp\{B\}\geq 0,$$
hence $(X,B(\bm{v}),\Mm)$ is qdlt for any $\bm{v}\in V_0$. Therefore, $(X,B,\Mm)$ is LD. 
\end{proof}

The following lemma shows that a non-trivial combination of an lc generalized pair and an LD generalized pairs is LD.

\begin{lem}\label{lem: combination of LD with lc g-pair}
    Let $\Aa/U:=(X,B,\Mm)/U$ be an lc generalized pair and $C$ an $\mathbb R$-Cartier $\mathbb R$-divisor on $X$. Assume that $(\Aa,C)$ is LD, then $\Aa(s):=(\Aa,sC)$ is LD for any $s\in (0,1]$.
\end{lem}
\begin{proof}
We only need to show that $\Aa(s)$ is LD for any $s\in (0,1)$. Fix $t\in (0,1)$. Let $h: X'\rightarrow X$ be a $\mathbb Q$-factorial dlt modification of $\Aa(t)$ and let $\Aa'(s):=(h^{-1}_*\Aa(s),\Exc(h))$ for any $s\in [0,1]$. Then $\Aa'(t)=h^*\Aa(t)$ is $\mathbb Q$-factorial dlt. Since $h$ only extract lc places of $\Aa(t)$ and since $\Aa(0)$ and $\Aa(1)$ are lc, by linearity of discrepancies, we have $\Aa'(s)=h^*\Aa(s)$ is lc for any $s\in [0,1]$. Since $\Aa'(t)$ is $\mathbb Q$-factorial dlt and $\Aa'(0)$ and $\Aa'(1)$ are lc, by linearity of discrepancies again, $\Aa'(s)$ is dlt for any $s\in (0,1)$.

Write $\Aa'(s):=(X',B'(s),\Mm)$ for any $s\in [0,1]$. Since 
$$\Supp B\geq B\geq 0\quad  \text{and}\quad \Supp(B+C)\geq B+C\geq 0,$$ 
we have
$$\Supp\{B+C\}\subset\Supp\{B+sC\}$$
for any $s\in (0,1)$, hence $\Supp\{B'(1)\}\subset\Supp\{B'(s)\}$ for any $s\in (0,1)$.

By Lemma \ref{lem: LD under dlt model}, $\Aa'(1)$ is LD. By Lemma \ref{lem: ld and irrationality}(1), 
$$\Supp\Mm_{X'}^{\irr}\subset\Supp\{B'(1)\}\subset\Supp\{B'(s)\}$$
for any $s\in (0,1)$. By Lemma \ref{lem: ld and irrationality}(2), $\Aa'(s)$ is LD for any $s\in (0,1)$. By Lemma \ref{lem: LD under dlt model}, $\Aa(s)$ is LD for any $s\in (0,1)$.
\end{proof}

\subsection{Nef and semi-ampleness of LD generalized pairs}

Usually, the canonical class of a klt generalized pair may be nef but not NQC (cf. \cite[Example 3.19]{HL22}). However, this cannot happen for LD generalized pairs (see Lemma \ref{lem: ld nef impleis nqc}).

\begin{lem}\label{lem: ld preserves trivialness}
    Let $\Aa/U$ be an LD generalized pair with polytopal $\{V,\bm{c},\Aa(\cdot)\}$. Let $h: X\rightarrow Z$ a projective morphism$/U$ such that $K_{\Aa}\sim_{\mathbb R,Z}0$. 
    
    Then $K_{\Aa(\bm{v})}\sim_{\mathbb R,Z}0$ for any $\bm{v}\in V$.
\end{lem}
\begin{proof}
The lemma is an immediate consequence of \cite[Lemma 5.3]{HLS24}.
\end{proof}

The following lemma shows that the ample model of an LD generalized pairs is stable under decomposition.

\begin{lem}\label{lem: semi-ampleness ld}
Let $\Aa/U$ be an LD generalized pair with polytopal $\{V,\bm{c},\Aa(\cdot)\}$. Assume that $K_{\Aa}$ is semi-ample$/U$ with ample model$/U$ $h: X\rightarrow Z$ (cf. \cite[Definition 3.6.5]{BCHM10}). Then there exists an open subset $V_0\ni\bm{c}$ of $V$, such that $K_{\Aa(\bm{v})}$ is semi-ample$/U$ with ample model$/U$ $h: X\rightarrow Z$.
\end{lem}
\begin{proof}
    We have $K_{\Aa}\sim_{\mathbb R,U}h^*L$ for some ample$/U$ $\Rr$-Cartier $\mathbb R$-divisor $L$. By Lemma \ref{lem: ld preserves trivialness}, there exists a continuous function $L(\cdot): V\rightarrow\Weil_{\mathbb R}(Z)$ such that $K_{\Aa(\bm{v})}\sim_{\mathbb R,U}h^*L(\bm{v})$ for any $\bm{v}\in V$. Since $L=L(\bm{c})$ is ample$/U$, there exists an open subset $V_0\ni\bm{c}$ of $V$, such that $L(\bm{v})$ is ample$/U$ for any $\bm{v}\in V_0$. Such $V_0$ satisfies our requirements.
\end{proof}

\begin{lem}\label{lem: ld nef impleis nqc}
Let $\Aa/U$ be an LD generalized pair with polytopal $\{V,\bm{c},\Aa(\cdot)\}$. Assume that $K_{\Aa}$ is nef$/U$. Then there exists an open subset $V_0\ni\bm{c}$ of $V$ such that $\Aa(\bm{v})$ is lc and $K_{\Aa(\bm{v})}$ is nef$/U$ for any $\bm{v}\in V_0$. In particular, $\Aa/U$ is KNQC.
\end{lem}
\begin{proof}
Since $\Aa$ is LD, there exists an open neighborhood $V_1\ni\bm{c}$ such that $\Aa(\bm{v})$ is lc for any $\bm{v}\in V_1$. Let $n:=\dim V+1$. Pick $\mathbb Q$-vectors $\bm{v}_1,\dots,\bm{v}_{n}\in V_1$ such that $\bm{c}$ is contained in the interior of the convex hull $V_2$ spanned by $\bm{v}_1,\dots,\bm{v}_{n}$.  Then there exist positive real numbers $a_1,\dots,a_{n}$ such that $\sum_{i=1}^{n}a_i\bm{v}_i=\bm{c}$ and $\sum_{i=1}^{n}a_i=1$. We let $I$ be a positive integer such that $IK_{\Aa(\bm{v}_i)}$ is Cartier for each $i$. Let $d:=\dim X$ and $a_0:=\min_{1\leq i\leq n}\{a_i\}$. 

Consider the set
    $$\Ii:=\left\{\sum a_i\gamma_i\middle|\gamma_i\in [-2dI,+\infty)\cap\mathbb Z\right\}\cap (0,+\infty).$$
    We have $\gamma_0:=\inf\{\gamma\mid\gamma\in\Ii\}>0$. We let $V_0$ be the interior of the set
    $$\left\{\frac{1}{2d+\gamma_0}(2d\bm{c}+\gamma_0\bm{v})\Bigg| \bm{v}\in V_2\right\}.$$

    We show that $V_0$ satisfies our requirement. By construction, $\Aa(\bm{v})$ is lc for any $\bm{v}\in V_0$, so it suffices to show that $K_{\Aa(\bm{v})}$ is nef$/U$ for any $\bm{v}\in V_0$. We let $R$ be an extremal ray in $\overline{NE}(X/U)$. There are three cases.

    \medskip

    \noindent\textbf{Case 1}. $K_{\Aa}\cdot R=0$. In this case, by \cite[Lemma 5.3]{HLS24}, $K_{\Aa(\bm{v})}\cdot R=0$ for any $\bm{v}\in V$, so $K_{\Aa(\bm{v})}\cdot R=0$ for any $\bm{v}\in V_1$.

\medskip

    \noindent\textbf{Case 2}. $K_{\Aa(\bm{v}_i)}\cdot R\geq 0$ for any $i$. In this case, $K_{\Aa(\bm{v})}\cdot R\geq 0$ for any $\bm{v}\in V_2$, so $K_{\Aa(\bm{v})}\cdot R\geq 0$ for any $\bm{v}\in V_0$.

    \medskip

     \noindent\textbf{Case 3}.  $K_{\Aa}\cdot R>0$ and $K_{\Aa(\bm{v}_j)}\cdot R<0$ for some $j$. In this case, by the cone theorem \cite[Theorem 2.2.1]{CHLX23}, $R$ is spanned by a curve $C$ such that $K_{\Aa(\bm{v}_i)}\cdot C\geq -2d$ for any $i$. In particular, $K_{\Aa(\bm{v})}\cdot C\geq -2d$ for any $\bm{v}\in V_2$. Moreover, we have
     $$IK_{\Aa(\bm{v}_i)}\cdot C\in [-2dI,+\infty)\cap\mathbb Z,$$
    so
    $$IK_{\Aa(\bm{c})}\cdot C\in\Ii_0.$$
    For any $\bm{v}\in V_0$, there exists $\bm{v}'\in V_2$ such that $(2d+\gamma_0)\bm{v}=2d\bm{c}+\gamma_0\bm{v}'$. We have
$$IK_{\Aa(\bm{v})}\cdot C=\frac{\gamma_0}{2d+\gamma_0}IK_{\Aa(\bm{v}')}\cdot C+\frac{2d}{2d+\gamma_0}IK_{\Aa(\bm{c})}\cdot C\geq\frac{\gamma_0}{2d+\gamma_0}\cdot (-2d)+\frac{\gamma_0}{2d+\gamma_0}\cdot\gamma_0=0,$$
    so $IK_{\Aa(\bm{v})}\cdot R\geq 0$. The lemma follows.
\end{proof}

\subsection{LD generalized pair under standard transformations}

We show that LD is a property that is preserved under adjunction and the minimal model program.

\begin{lem}\label{lem: ld adjunction}
Let $\Aa/U=(X,B,\Mm)/U$ be an LD generalized pair and let $\widetilde{S}$ be an irreducible component of $\lfloor B\rfloor$ with normalization $S$. Then $\Aa|_S$ is RD.
\end{lem}
\begin{proof}
By Lemma \ref{lem: equivalent definition of LD}, there exist $a_1,\dots,a_m\in (0,1]$ and lc generalized pairs $\Aa_i/U:=(X,B_i,\Mm)/U$ such that $\sum_{i=1}^ma_i=1$, $\sum_{i=1}^ma_i\Aa_i=\Aa$, and $K_{\Aa_i}$ is $\mathbb Q$-Cartier for each $i$. Then $S$ is an lc place of $\Aa_i$ for each $i$. Let $\Aa_{i,S}:=\Aa_i|_S$ for each $i$. Then $\Aa_{i,S}$ is lc and $K_{\Aa_{i,S}}$ is $\mathbb Q$-Cartier for each $i$, and we have $\sum_{i=1}^ma_i\Aa_{i,S}=\Aa|_S$. By Lemma \ref{lem: equivalent definition of LD}, $\Aa|_S$ is LD.
\end{proof}

\begin{lem}\label{lem: LD preserve under mmp}
Let $\Aa/U$ be an LD generalized pair and $\phi: X\dashrightarrow X'$ a sequence of steps of a $K_{\Aa}$-MMP$/U$ (resp. a $K_{\Aa}$-trivial map$/U$). Then $\phi_*\Aa$ is LD.
\end{lem}
\begin{proof}
Let $\{V,\bm{c},\Aa(\cdot)\}$ be the polytopal of $\Aa$. There exists an open neighborhood $V_0\ni\bm{c}$ such that $\phi$ is a sequence of steps of a $K_{\Aa(\bm{v})}$-MMP$/U$ for any $\bm{v}\in V_0$ (resp. is  a $K_{\Aa(\bm{v})}$-trivial map$/U$). Since $\Aa$ is LD, possibly shrinking $V_0$, we may assume that $\Aa(\bm{v})$ is lc for any $\bm{v}\in V_0$. Thus $\phi_*\Aa(\bm{v})$ is lc for any $\bm{v}\in V_0$.

Pick $\mathbb Q$-vectors $\bm{v}_1,\dots,\bm{v}_m\in V_0$ such that $\bm{c}$ is contained in the interior of the convex hull spanned by $\bm{v}_1,\dots,\bm{v}_m$, where $m=\dim V+1$. Then there exist unique positive real numbers $a_1,\dots,a_m$ such that $\sum_{i=1}^ma_i=1$ and $\sum_{i=1}^ma_i\bm{v}_i=\bm{c}$. Then $\Aa_i':=\phi_*\Aa(\bm{v}_i)$ are lc generalized pairs, $K_{\Aa_i'}$ is a $\mathbb Q$-divisor for each $i$, and $\sum_{i=1}^ma_i\Aa_i'=\phi_*\Aa$. By Lemma \ref{lem: equivalent definition of LD}, $\phi_*\Aa$ is LD.
\end{proof}

\subsection{LD property under ample polarization}

In this subsection, we show that for any lc generalized pair $(X,B,\Mm)$ and any ample $\mathbb R$-divisor $A$ on $X$, the divisor $K_X+B+A+\Mm_X$ can be realized as the canonical class of an LD generalized pair, polarized by another ample class. We will not use the results of this subsection in the rest of the paper; nevertheless, they are conceptually natural and intuitive for the introduction of LD generalized pairs. Indeed, Lemma \ref{lem: +ample is ld} below could immediately reduce the question of the existence of flips (Theorem \ref{thm: flip nonnqc-g}) to the category of LD generalized pairs.

\begin{lem}\label{lem: ld plus general base point free}
    Let $\Aa/U$ be an LD generalized pair, $|H_i|_U$ base-point-free$/U$ linear systems, and $c_1,\dots,c_n\in [0,1]$ real numbers. Then $\left(\Aa,\sum_{i=1}^nc_iA_i\right)$ is LD for any general $A_i\in |H_i|_U$.  
\end{lem}
\begin{proof}
Write $\Aa=(X,B,\Mm)$. By Lemma \ref{lem: equivalent definition of LD}, we may write $\Aa=\sum b_j\Aa_j$ such that each $\Aa_j:=(X,B_j,\Mm)$ is lc, $K_{\Aa_j}$ is $\mathbb Q$-Cartier, and $\sum b_j=1$. Let $\bm{c}=(c_1,\dots,c_n)$ and $V$ the rational envelope of $\bm{c}$ in $\mathbb R^n$. We may pick $\bm{c}_1,\dots,\bm{c}_k\in V\cap\mathbb Q^n$ such that $\bm{c}$ is contained in the interior of the convex hull spanned by $\bm{c}_1,\dots,\bm{c}_k$. Let $a_1,\dots,a_k\in (0,1]$ be real numbers such that $\sum_{l=1}^ka_i=1$ and $\sum_{l=1}^ka_l\bm{c}_l=\bm{c}$. Write $\bm{c}_l=(c_{l,1},\dots,c_{l,n})$ for any $l$. Then
$$K_{\Aa}+\sum_{i=1}^nc_iA_i=\sum_j\sum_{l=1}^k b_ja_l\left(K_{\Aa_j}+\sum_{i=1}^nc_{l,i}A_i\right)$$
and each $\left(\Aa_j,\sum_{i=1}^nc_{l,i}A_i\right)$ is lc. By Lemma \ref{lem: equivalent definition of LD}, $\left(\Aa,\sum_{i=1}^nc_iA_i\right)$ is LD.
\end{proof}

\begin{lem}\label{lem: +ample is ld}
    Let $\Aa/U$ be an lc generalized pair and $A$ an ample$/U$ $\Rr$-divisor on $X$. Then there exist three ample$/U$ $\Rr$-divisors $H,J$ and $L$ on $X$ satisfying the following.
    \begin{enumerate}
        \item $A=H+J+L$.
        \item $K_{\Aa}+H$ and $L$ are $\mathbb Q$-divisors.
        \item $\left(\Aa,J'+L',\overline{H}\right)$ is LD for some $J'\in |J|_{\mathbb R/U}$ and $L'\in |L|_{\mathbb Q/U}$.
    \end{enumerate}
\end{lem}
\begin{proof}
    We write $K_{\Aa}=D_0+\sum_{i=1}^c r_iD_i$ where $1,r_1,\dots,r_c$ are $\mathbb Q$-linearly independent and $D_0,\dots,D_c$ are $\mathbb Q$-divisors. By \cite[Lemma 5.3]{HLS24}, each $D_i$ is $\mathbb Q$-Cartier. Let $H_0\geq 0$ be an ample$/U$ $\mathbb Q$-divisor such that $A-H_0$ is ample$/U$. Then there exists a real number $\epsilon>0$ such that $H_0+\sum_{i=1}^c \epsilon_iD_i$ and $A-H_0+\sum_{i=1}^c \epsilon_iD_i$ are ample for any $\epsilon_i\in (-\epsilon,\epsilon)$. 

    Pick $r_i'\in\mathbb Q$ such that $|r_i'-r_i|<\epsilon$ and let
    $$H:=H_0+\sum_{i=1}^c (r_i'-r_i)D_i.$$
    Then $H$ is ample$/U$, $A-H$ is ample$/U$, and
    $$K_{\Aa}+H=D_0+H_0+\sum_{i=1}^cr_i'D_i$$
    is a $\mathbb Q$-divisor. Let $L$ be an ample$/U$ $\mathbb Q$-divisor such that $J:= A-H-L$ is ample$/U$. Then (1-2) follows.

    We let $m>0$ be a sufficiently divisible integer such that $\mathcal{O}_X(mL)$ is globally generated$/U$ and let $L':=\frac{1}{m}G$ for some general $G\in |mL|_{U}$. We write $J=\sum_{i=1}^ms_iJ_i$ where $s_i\in (0,1)$ and $\mathcal{O}_X(J_i)$ is globally generated$/U$ for each $i$, and let $J':=\sum_{i=1}^ms_iG_i$ where $G_i\in |J_i|_{U}$ are general elements. Then $\left(\Aa,\sum_{i=1}^mG_i+L',\overline{H}\right)$ is lc.
    
    Let $V$ be the rational envelope of $\bm{s}:=(s_1,\dots,s_m)\in\mathbb R^m$ and let $\bm{v}_1,\dots,\bm{v}_{m+1}\in V\cap\mathbb Q^m\cap (0,1)^m$ be vectors such that $\bm{s}$ is contained in the interior of the convex hull spanned by $\bm{v}_1,\dots,\bm{v}_{m+1}$. Then there exist unique positive real numbers $a_1,\dots,a_{m+1}$ such that $\sum_{i=1}^{m+1}a_i=1$ and $\sum_{i=1}^{m+1}a_i\bm{v}_i=\bm{s}$. Let
$$\Aa(\bm{v}):=\left(\Aa,\sum_{i=1}^mv_iG_i+L',\overline{H}\right)$$
    for any  $\bm{v}=(v_1,\dots,v_{m})\in\mathbb R^m$, then $$\sum_{i=1}^{m+1}a_i\Aa(\bm{v}_i)=\left(\Aa,J'+L',\overline{H}\right)$$
    and  $K_{\Aa(\bm{v}_i)}$ is a $\mathbb Q$-divisor for each $i$. By Lemma \ref{lem: equivalent definition of LD}, $\left(\Aa,J'+L',\overline{H}\right)$ is LD.
\end{proof}

\section{Models}\label{sec: models}

In this section we recall some basic properties of models of generalized pairs. This section aligns with \cite[Section 2]{Bir12}, \cite[Section 3]{HL23}, \cite[Section 4]{LMX24}, and \cite[Section 3]{Cas+25b}. Some of the results in this section are special cases of results in \cite{LMX24,Cas+25b} where the results were stated for the more general category of algebraically integrable adjoint foliated structures. For the reader's convenience, we still provide the precise statements in this section.

\subsection{Basic properties of models}

\begin{defn}[Log birational model]\label{defn: log birational model}
Let $\Aa/U$ be an lc generalized pair with ambient variety $X$ and $\phi: X\dashrightarrow X'$ a birational map$/U$. Let $\Aa':=(\phi_*\Aa,\Exc(\phi^{-1}))$. If $K_{\Aa}$ is $\mathbb R$-Cartier, then we say that $\Aa'/U$ is a \emph{log birational model} of $\Aa/U$.
\end{defn}

\begin{defn}[Models]\label{defn: minimal model}
    Let $\Aa/U$ be a generalized pair with ambient variety $X$ and $\Aa'/U$ a log birational model of $\Aa/U$ with ambient variety $X'$ associated with birational map$/U$ $\phi: X\dashrightarrow X'$. Consider the following conditions:
    \begin{enumerate}
        \item $K_{\Aa'}$ is nef$/U$.
        \item $K_{\Aa'}$ is semi-ample$/U$.
        \item  $a(D,\Aa)\leq a(D,\Aa')$ for any exceptional$/X'$ prime divisor $D$ on $X$.
        \item  $a(D,\Aa)<a(D,\Aa')$ for any exceptional$/X'$ prime divisor $D$ on $X$.
        \item $\Aa$ is $\mathbb Q$-factorial qdlt.
    \end{enumerate}
    If conditions (1)(3) (resp. (1)(4), (2)(3), (2)(4), (1)(4)(5), (2)(4)(5)) are satisfied, then we say that $\Aa'/U$ is a \emph{bs-weak lc model} (resp. \emph{bs-minimal model, bs-semi-ample model, bs-good minimal model, log minimal model, good log minimal model}) of $\Aa/U$. If condition (4) (resp. (4)(5)) holds and $f: X'\rightarrow Z$ is a $K_{\Aa'}$-Mori fiber space$/U$, then we say that $f: \Aa'\rightarrow Z$ is a bs-Mori fiber space (resp. log Mori fiber space) of $\Aa/U$. 

    Here ``bs-" stands for ``in the sense of Birkar-Shokurov". If $\phi$ does not extract any divisor, then we may remove the prefix ``bs-" in the definitions above.
\end{defn}

\begin{lem}\label{lem: Bir12 2.6}
Let $\Aa/U$ be an lc generalized pair with ambient variety $X$ and let $\Aa'/U$ be a bs-weak lc model of $\Aa/U$ with ambient variety $X'$ and associated with birational map$/U$ $\phi: X\dashrightarrow X'$. Let $p: W\rightarrow X$ and $q: W\rightarrow X'$ be a resolution of indeterminacy of $\phi$. Assume that
$$p^*K_{\Aa}=q^*K_{\Aa'}+E,$$
then $E\geq 0$ and is exceptional$/X'$.
\end{lem}
\begin{proof}
This is a special case of \cite[Lemma 3.3]{Cas+25b}.
\end{proof}

\begin{lem}\label{lem: Bir12 2.7}
Let $\Aa/U$ be an lc generalized pair. Let $\Aa_1/U$ and $\Aa_2/U$ be two bs-weak lc models of $\Aa/U$. Let $X,X_1,X_2$ be the ambient varieties of $\Aa,\Aa_1,\Aa_2$ respectively with induced birational maps $\phi: X_1\dashrightarrow X_2$. Let $h_1: W\rightarrow X_1$ and $h_2: W\rightarrow X_2$ be two birational morphisms such that $\phi\circ h_1=h_2$. Then:
\begin{enumerate}
    \item $h_1^*K_{\Aa_1}=h_2^*K_{\Aa_2}.$
    \item If $K_{\Aa_2}$ is semi-ample$/U$, then $K_{\Aa_1}$ is semi-ample$/U$.
    \item If $K_{\Aa_2}$ is ample$/U$, then $\phi$ is a morphism.
\end{enumerate}
\end{lem}
\begin{proof}
This is a special case of \cite[Lemma 3.4]{Cas+25b}.
\end{proof}

\begin{lem}\label{lem: Cas+25b 3.5}
    Let $r$ be a positive real number. Let $\Aa_1/U$ and $\Aa_2/U$ be two lc generalized pairs with the same ambient variety $X$ such that
    $$K_{\Aa_2}\equiv_U rK_{\Aa_1}.$$
    Let $\Aa_1'/U$ be a weak lc model (resp. minimal model) of $\Aa_1/U$ with ambient variety $X'$ and induced birational map $\phi: X\dashrightarrow X'$. Let $\Aa_2':=\phi_*\Aa$. Then:
    \begin{enumerate}
        \item $\Aa_2'/U$ is a weak lc model (resp. minimal model) of $\Aa_2/U$.
        \item If $\Aa_1'/U$ is a semi-ample model (resp. good minimal model) of $\Aa_1/U$ and
    $$K_{\Aa_2}\sim_{\mathbb R,U} rK_{\Aa_1},$$
      then  $\Aa_2'/U$ is a semi-ample model (resp. good minimal model) of $\Aa_2/U$.
    \end{enumerate}
\end{lem}
\begin{proof}
It is a special case of \cite[Lemma 3.5]{Cas+25b}.
\end{proof}

\begin{thm}[Very exceptional MMP]\label{thm: very exceptional mmp}
Let $\Aa/U$ be an lc algebraically integrable adjoint foliated structure with ambient variety $X$. Assume that $X$ is potentially klt, and $K_{\Aa}\sim_{\mathbb R,U}E\geq 0$ for some very exceptional$/U$ (cf. \cite[Definition 3.1]{Bir12}) $\mathbb R$-divisor $E$. Then we may run a $K_{\Aa}$-MMP$/U$ with scaling of an ample$/U$ $\Rr$-divisor and any such MMP terminates with a good minimal model $\Aa'/U$ of $\Aa/U$ such that $K_{\Aa'}\sim_{\mathbb R,U}0$.

In particular, the reduced divisor contracted by this $K_{\Aa}$-MMP$/U$ is exactly $\Supp E$.
\end{thm}
\begin{proof}
It is a special case of \cite[Theorem 3.6]{Cas+25b}.
\end{proof}

\subsection{Log toroidal model}

\begin{defn}[Log toroidal model]\label{defn: log toroidal models}
Let $\Aa/U$ be an lc generalized pair with ambient variety $X$ and $h: X'\rightarrow X$ a projective birational morphism. Let $E\geq 0$ be an $h$-exceptional $\mathbb R$-divisor and let $\Aa':=(h^*\Aa,E)$. Write $\Aa'=(X',B',\Mm)$. We say that $h: \Aa'\rightarrow\Aa$ is a \emph{log toroidal model} if 
\begin{enumerate}
    \item $(X',\Supp B'\cup\Exc(h))$ is $\mathbb Q$-factorial log toroidal, $B'\geq 0$, and $\Mm$ descends to $X'$, 
    \item $\Aa'$ is lc, and
    \item for any $h$-exceptional prime divisor $D$ such that $a(D,\Aa)>-1$, $D$ is an irreducible component of $E$.
\end{enumerate}
In addition:
\begin{itemize}
\item If $(X',\Supp B'\cup\Exc(h))$ is log smooth, then we say that $h: \Aa'\rightarrow\Aa$ is a \emph{log smooth model}.
\item If $\Supp E\subset\Supp\{B'\}$, then we say that $h: \Aa'\rightarrow\Aa$ is a \emph{proper log toroidal model}.
\item If $h: \Aa'\rightarrow\Aa$ is a log smooth model and a proper log toroidal model, then we say that $h: \Aa'\rightarrow\Aa$ is a proper log smooth model.
\item If  $h: \Aa'\rightarrow\Aa$ is a (proper) log toroidal model (resp. (proper) log smooth model), then we say that $\Aa'$ is a (proper) log toroidal model (resp. (proper) log smooth model) of $\Aa$.
\end{itemize}
\end{defn}

\begin{defn}
    Let $\Aa/U=(X,B,\Mm)/U$ be a generalized pair and $\pi: X\rightarrow Z$ a contraction. We say that $\Aa/Z$ is \emph{relatively log toroidal} if the following conditions hold.
    \begin{enumerate}
        \item $X$ has at most $\mathbb Q$-factorial quotient singularities.
        \item $(X,\Sigma)$ is log toroidal for some reduced divisor $\Sigma$ such that $\Supp B\subset\Sigma$. In particular, $X$ is klt.
        \item $\pi$ is a contraction.
        \item There exists a log smooth pair $(Z,\Sigma_Z)$ such that $\pi: (X,\Sigma)\rightarrow (Z,\Sigma_Z)$ is an equidimensional toroidal contraction. In particular, $Z$ is smooth.
        \item $\Mm$ descends to $X$.
    \end{enumerate}
Note that we do not require that $\pi$ is a contraction$/U$.
\end{defn}
\begin{deflem}\label{deflem: foliated log resolution}
  Let $\Aa/U=(X,B,\Mm)/U$ be a generalized pair and $\phi: X\dashrightarrow Z$ a dominant map. A \emph{log toroidal modification} of $\Aa$ with respect to $\phi$ is a projective birational morphism $h: X'\rightarrow X$, such that there exists a projective birational morphism $h_Z: Z'\rightarrow Z$ and a contraction $f: X'\rightarrow Z'$ with $\phi\circ h=h_Z\circ f$, so that
  $$(X',h^{-1}_*B+\Exc(h),\Mm)/Z$$
  is relatively log toroidal. Log toroidal modification always exists: indeed, a log toroidal modification of $\Aa$ is a foliated log resolution of $(X,\Ff,B,\Mm)$ where $\Ff$ is the foliation induced by $\phi$, which is known to exist by \cite[Lemma 6.2.4]{CHLX23}.

  We say that $f: X'\rightarrow Z'$ is \emph{associated to} $h:\Aa'\rightarrow\Aa$.
\end{deflem}

\begin{deflem}\label{deflem: eoltm}
    Let $\Aa/U$ be an lc generalized pair with ambient variety $X$, $\phi: X\dashrightarrow Z$ a dominant map, and $h: X'\rightarrow X$ a log toroidal modification of $\Aa$ with respect to $\phi$. Then there exists a proper log toroidal model $h: \Aa'\rightarrow\Aa$, such that if $\Aa$ is LD, then $\Aa'$ is LD. We say that $h: \Aa'\rightarrow\Aa$ is a \emph{proper log toroidal model with respect to $\phi$}. If $\Aa'$ is LD, then we say that $h: \Aa'\rightarrow\Aa$ is a \emph{proper LD log toroidal model with respect to $\phi$}.
\end{deflem}
\begin{proof}
 Write $\Exc(h)=E_1+E_2$ where $E_1\wedge E_2=0$, any irreducible component of $E_1$ is an lc place of $\Aa$, and no irreducible component of $E_2$ is an lc place of $\Aa$. Write $E_2=\sum_{i=1}^m F_i$ where $F_i$ are the irreducible components of $E_2$.
    
    If $\Aa$ is not LD, then we may choose arbitrary $0<\epsilon\ll 1$, and $\Aa':=(h^{-1}_*\Aa,E_1+(1-\epsilon)E_2)$ satisfies our requirements. Therefore, we may assume that $\Aa$ is LD. 
    
    We let $\{V,\bm{c},\Aa(\cdot)\}$ be the polytopal of $\Aa$. By Lemma \ref{lem: LD preserves lc center}, for any $\bm{v}\in V$, any irreducible component of $E_1$ is an lc place of $\Aa(\bm{v})$. Thus there exists an open neighborhood $V_0\ni\bm{c}$ in $V$, such that for any $\bm{v}\in V_0$, $a(F_i,\Aa(\bm{v}))>-1$ for any $i$ and $a(F,\Aa(\bm{v}))=-1$ for any irreducible component $F$ of $E_1$. Let $n:=\dim V+1$ and let $\bm{v}_1,\dots,\bm{v}_n$ be $\mathbb Q$-vectors in $V_0$ such that $\bm{c}$ is contained in the interior of the convex hull spanned by $\bm{v}_1,\dots,\bm{v}_n$, then there exists unique real numbers $a_1,\dots,a_n\in (0,1]$ such that $\sum_{i=1}^na_i=1$ and $\sum_{i=1}^na_i\bm{v}_i=\bm{c}$. Then there exist $\mathbb Q$-divisors $G_i\geq 0$ on $X'$ such that $\Supp G_i=\Supp E_2$ and $\Aa_i':=(h^*\Aa(\bm{v}_i),E_1+G_i)$ is a proper log toroidal model of $\Aa(\bm{v}_i)$. We let $\Aa':=\sum_{i=1}^n a_i\Aa_i'$. By Lemma \ref{lem: equivalent definition of LD}, $\Aa'$ is LD and $h: \Aa'\rightarrow\Aa$ is a proper log toroidal model of $\Aa$.
\end{proof}

\begin{lem}\label{lem: eolsm}
Let $\Aa/U$ be an lc generalized pair. Then $\Aa$ has a proper log smooth model. Moroever, if $\Aa$ is LD, then $\Aa$ has a proper LD log toroidal model.
\end{lem}
\begin{proof}
The lemma follows from Definition-Lemma \ref{deflem: eoltm} by letting $\phi$ be the identity morphism.
\end{proof}

\begin{lem}\label{lem: Bir12 2.8}
Let $\Aa/U$ be an lc generalized pair and $\Aa'$ a log toroidal model of $\Aa$. Then any bs-weak lc model (resp. bs-minimal model, bs-semi-ample model, bs-good minimal model, log minimal model, good log minimal model) of $\Aa'/U$ is a bs-weak lc model (resp. bs-minimal model, bs-semi-ample model, bs-good minimal model, log minimal model, good log minimal model) of $\Aa/U$. 
\end{lem}
\begin{proof}
This is a special case of \cite[Lemma 3.8]{Cas+25b}.
\end{proof}

\subsection{Models under birational transformations}

\begin{lem}\label{lem: foliation lsm has lmm}
Let $\Aa/U$ be an lc generalized pair and $\Aa'/U$ a bs-weak lc model of $\Aa/U$. Let $\Aa_W$ be a log toroidal model of $\Aa$ and assume that the induced birational map $\phi: W\dashrightarrow X'$ is a morphism, where $W$ and $X'$ are the ambient varieties of $\Aa_W$ and $\Aa'$ respectively. 

Then we may run a $K_{\Aa_W}$-MMP$/X'$ with scaling of an ample$/X'$ $\Rr$-divisor and any such MMP terminates with a good minimal model $\Aa_Y/X'$ of $\Aa_W/X'$ such that $\Aa_Y=q^*\Aa'$, where $Y$ is the ambient variety of $\Aa_Y$ and $q: Y\rightarrow X'$ is the induced morphism.

In particular, $\Aa_Y/U$ is a $\mathbb Q$-factorial qdlt minimal model of $\Aa_W/U$.
\end{lem}
\begin{proof}
This is a special case of \cite[Lemma 3.10]{Cas+25b}.
\end{proof}

\begin{lem}\label{lem: g-pair weak glc imply lmm}
Let $\Aa/U$ be an lc generalized pair. If $\Aa/U$ has a bs-weak lc model (resp. bs-semi-ample model), then $\Aa/U$ has a log minimal model (resp. good log minimal model).
\end{lem}
\begin{proof}
This is a special case of \cite[Lemma 3.12]{Cas+25b}.
\end{proof}

\begin{lem}\label{lem: Cas+25b 3.13}
Let $\Aa/U$ be an lc generalized pair. Let $h: X'\rightarrow X$ be a projective birational morphism and $\Aa'/U$ an lc generalized pair such that $\Aa'=(h^*\Aa,E)$ for some $E\geq 0$ that is exceptional$/X$. Then:
\begin{enumerate}
    \item Any bs-weak lc model of $\Aa/U$ is a bs-weak lc model of $\Aa'/U$.
    \item If $\Aa/U$ has a bs-weak lc model (resp. bs-semi-ample model), then $\Aa'/U$ has a log minimal model (resp. good log minimal model).
\end{enumerate}
\end{lem}
\begin{proof}
   This is a special case of \cite[Lemma 3.13]{Cas+25b}. 
\end{proof}

\begin{lem}\label{lem: mm preserved under dlt model}
Let $\Aa/U$ be an lc generalized pair. Let $h: X'\rightarrow X$ be a birational morphism which only extracts lc places of $\Aa$ and $\Aa':=h^*\Aa$. Then any bs-weak lc model (resp. bs-minimal model) of $\Aa'/U$ is a bs-weak lc model (resp. bs-minimal model) of $\Aa/U$.
\end{lem}
\begin{proof}
This is a special case of \cite[Lemma 3.14]{Cas+25b}.
\end{proof}

\subsection{Models under MMP}

The following theorem shows that existence of a bs-weak lc model is equivalent to the existence of a minimal model when the ambient variety is klt.

\begin{thm}\label{thm: bswlc imply mm}
    Let $\Aa/U$ be an lc generalized pair with ambient variety $X$. Assume that $X$ is potentially klt and $\Aa/U$ has a KNQC bs-weak lc model (resp. bs-semi-ample model). Then any sequence of steps of a $K_{\Aa}$-MMP$/U$ with scaling of an ample$/U$ $\mathbb R$-divisor terminates with a minimal model (resp. good minimal model) of $\Aa/U$. In particular, $\Aa/U$ has a $\mathbb Q$-factorial minimal model (resp.  $\mathbb Q$-factorial good minimal model).
\end{thm}
\begin{proof}
By \cite[Theorem 4.6]{Cas+25b}, any sequence of steps of a $K_{\Aa}$-MMP$/U$ with scaling of an ample$/U$ $\mathbb R$-divisor terminates with a minimal model $\Aa'/U$ of $\Aa/U$. By Lemma \ref{lem: Bir12 2.7}, if $\Aa/U$ has a bs-semi-ample model, then $\Aa'/U$ is a good minimal model of $\Aa/U$. Since $X$ is potentially klt, the ambient variety $X'$ of $\Aa'$ is potentially klt. Thus there exists a small $\mathbb Q$-factorialization $\pi: X''\rightarrow X'$. $\pi^*\Aa'/U$ is a $\mathbb Q$-factorial minimal model (resp.  $\mathbb Q$-factorial good minimal model) of $\Aa/U$.
\end{proof}

\begin{lem}\label{lem: minimal model same after running mmp}
    Let $\Aa/U$ be an lc generalized pair and let $\phi: X\dashrightarrow X'$ be a $K_{\Aa}$-negative map$/U$ with $\Aa':=\phi_*\Aa$. Then:
    \begin{enumerate}
        \item Any (bs-)minimal model of $\Aa/U$ is a (bs-)minimal model of $\Aa'/U$.
        \item Any (bs-)minimal model of $\Aa'/U$ is a (bs-)minimal model of $\Aa/U$.
    \end{enumerate}
\end{lem}
\begin{proof}
This is a special case of \cite[Lemma 3.15]{Cas+25b}.
\end{proof}

\begin{lem}\label{lem: scaling basic properties}
Let $\Aa/U$ be an lc generalized pair. Let 
$$\phi_i: \Aa_i\dashrightarrow \Aa_{i+1},\quad \Aa_1:=\Aa$$ 
be a sequence of steps of a $K_{\Aa}$-MMP$/U$ with scaling of some $\mathbb R$-divisor $C$. Let $X_i$ be the ambient variety of $\Aa_i$ and $C_i$ be the image of $C$ on $X_i$ for each $i$, and let
$$\lambda_i:=\inf\{s\geq 0\mid K_{\Aa_i}+sC_i\text{ is nef}/U\}$$
be the scaling numbers of this MMP. Assume this MMP does not terminate. Let $\lambda:=\lim_{i\rightarrow+\infty}\lambda_i$. Then for any $i\gg 0$, we have the following.
\begin{enumerate}
    \item $\phi_i$ is a flip.
    \item If $\lambda=0$, then $K_{\Aa_i}$ is movable$/U$.
    \item Assume that $\lambda=0$ and $\Aa/U$ has a bs-minimal model (resp. minimal model) $\Aa'/U$ with ambient variety $X'$. Let $\psi_i: X_i\dashrightarrow X'$ be the induced birational map. Then $\psi_i$ does not contract any divisor (resp. $\psi_i$ is small).
\end{enumerate}
\end{lem}
\begin{proof}
This is a special case of \cite[Lemma 3.16]{Cas+25b}.
\end{proof}

\subsection{LD models}

The following lemma shows that LD is a property that is preserved under taking minimal models or weak lc models.

\begin{lem}\label{lem: ld minimal model is ld}
    Let $\Aa/U$ be an LD generalized pair and $\Aa'/U$ a bs-weak lc model of $(X,B,\Mm)/U$. Then $\Aa'/U$ is LD. In particular, $K_{\Aa'}$ is NQC$/U$.
\end{lem}
\begin{proof}
Let $X$ and $X'$ be the ambient variety of $\Aa$ and $\Aa'$ respectively and $\phi: X\dashrightarrow X'$ be the associated birational map$/U$. Let $\{V,\bm{c},\Aa(\cdot)\}$ be the polytopal of $\Aa$ and let $\Aa'(\bm{v}):=(\phi_*\Aa(\bm{v}),\Exc(\phi^{-1}))$ for any $\bm{v}\in V$. Let $p: W\rightarrow X$ and $q: W\rightarrow X'$ be a resolution of indeterminacy of $\phi$ and write
$$p^*K_{\Aa(\bm{v})}=q^*K_{\Aa'(\bm{v})}+E(\bm{v})$$
for some $\mathbb R$-divisor $E(\bm{v})$. Then $E(\cdot): V\rightarrow\Weil_{\mathbb R}(W)$ is $\mathbb Q$-affine, and we have $E(\bm{c})\geq 0$. Thus $E(\bm{v})\geq 0$ for any $\bm{v}\in V$ and $||\bm{v}-\bm{c}||\ll 1$. Since $\Aa(\bm{v})$ is lc for any $\bm{v}\in V$ such that $||\bm{v}-\bm{c}||\ll 1$, $\Aa'(\bm{v})$ is lc for any $\bm{v}\in V$ such that $||\bm{v}-\bm{c}||\ll 1$. Thus $\Aa'=\Aa'(\bm{c})$ is LD. The in particular part follows from Lemma \ref{lem: ld nef impleis nqc}.
\end{proof}

\section{Minimal model program with scaling}\label{sec: mmp with scaling}

The goal of this section is to prove that, for LD generalized pairs, existence of minimal models is equivalent to the termination of MMP with scaling (Theorem \ref{thm: eomm implies tof with scaling}). Such a result was proved for lc pairs in \cite[Theorem 1.9]{Bir12}, for NQC generalized pairs in \cite[Theorem 4.1]{HL22} and \cite[Theorem 4.1]{LT22}, and more generally for algebraically integrable adjoint foliated structures in \cite[Theorem 4.1]{Cas+25b}. We also need to show that a similar result holds for MMP over a non-empty open subset (Lemma \ref{lem: termination over open subset}), and construct a special MMP with scaling numbers always decreasing to $0$ (Lemma \ref{lem: construct mmp with scaling to 0}).

\subsection{Existence of MMP with scaling}

\begin{lem}[{cf. \cite[Lemma 4.1]{Hu25}}]\label{lem: can run scaling}
Let $\Aa/U$ be an lc generalized pair and $(\Aa,C)/U$ an LD generalized pair for some $\mathbb R$-divisor $C$. Assume that $K_{\Aa}+C$ is nef$/U$. Let
$$\lambda:=\inf\{t\mid t\geq 0, K_{\Aa}+tC\text{ is nef}/U\}.$$
Then either $\lambda=0$ and $K_{\Aa}$ is nef$/U$, or there exists a $K_{\Aa}$-negative extremal ray $R\in\overline{NE}(X/U)$ such that $(K_{\Aa}+\lambda C)\cdot R=0$.
\end{lem}
\begin{proof}
Since nef$/U$ is a closed condition, $K_{\Aa}+\lambda C$ is nef$/U$.  We may assume that $\lambda>0$. Let $\Aa(s):=\left(\Aa,\frac{1+s}{2}\lambda C\right)$ for any $s\in\mathbb R$. By Lemma \ref{lem: combination of LD with lc g-pair}, $\Aa(0)$ and $\Aa(1)$ are LD. By Lemmas \ref{lem: equivalent definition of LD} and \ref{lem: ld nef impleis nqc}, we may write 
$$\Aa(0)=\sum_{i=1}^m a_i\Aa_i,\quad\text{ and }\quad \Aa(1)=\sum_{j=1}^nc_j\Cc_j,$$
such that $a_i,b_j\in (0,1]$, $\sum_{i=1}^ma_i=\sum_{j=1}^nc_j=1$, $\Aa_i,\Cc_j$ are lc and have the same nef part as $\Aa$, and $K_{\Aa_i},K_{\Cc_j}$ are nef$/U$ $\mathbb Q$-divisors. Let $I$ be a positive integer such that $IK_{\Aa_i}$ and $IK_{\Cc_j}$ are Cartier for any $i,j$. 

By \cite[Theorem 2.2.1]{CHLX23}, we may let $\{R_k\}_{k\in\Lambda}$ be the set of $K_{\Aa(0)}$-negative extremal rays in $\overline{NE}(X/U)$, and $C_k$ a curve which generates $R_k$ such that $$ -2\dim X\leq K_{\Aa_i}\cdot C_k,\quad -2\dim X\leq K_{\Cc_j}\cdot C_k$$ 
for any $i,j,k$. Let $\alpha_k:=K_{\Aa(0)}\cdot C_k$ and $\beta_k:=K_{\Aa(1)}\cdot C_k$ for each $k$. Then for each $k$, we have
$$-2\dim X\leq \alpha_k=K_{\Aa(0)}\cdot C_k=\sum_{i=1}^m\frac{a_il_{i,k}}{I}<0\quad \text{and}\quad \beta_k=K_{\Aa(1)}\cdot C_k=\sum_{j=1}^n\frac{c_jl_{i,k}'}{I}\geq 0,$$
where $l_{i,j},l_{i,k}'$ are integers, each $l_{i,j}\geq -2I\dim X$, and each $l_{i,j}'\geq 0$. Therefore, 
$$\Ii:=\{\alpha_k\}_{k\in\Lambda}\quad \text{and}\quad \Ii':=\{\beta_k\}_{k\in\Lambda}$$
are DCC sets. For any $k\in\Lambda$, let 
$$t_k:=-\frac{\alpha_k}{\beta_k-\alpha_k}=\frac{1}{1+\frac{\beta_k}{-\alpha_k}},$$ 
then $\{t_k\}_{k\in\Lambda}$ is an ACC set. We have
\begin{align*}
  1&=\inf\{s\geq 0\mid K_{\Aa(s)}\text{ is nef}/U\}=\inf\{s\geq 0\mid K_{\Aa(s)}\cdot C_k\geq 0,\forall k\}\\
  &=\sup\{t_k\}_{k\in\Lambda}=\max\{t_k\}_{k\in\Lambda}=t_{k_0}  
\end{align*}
for any $k_0\in\Lambda$. Thus $\beta_{k_0}=0$, and we take $R=R_{k_0}$.
\end{proof}

\subsection{MMP with scaling and existence of minimal models}

\begin{thm}\label{thm: eomm implies tof with scaling}
Let $\Aa/U$ be an lc generalized pair and $C\geq 0$ an $\mathbb R$-divisor on $X$ such that $(\Aa,C)$ is lc. Let 
$$\phi_i: \Aa_i\dashrightarrow\Aa_{i+1}, \Aa_1:=\Aa$$ be a $K_{\Aa}$-MMP$/U$ with scaling of $C$, $C_i$ the image of $C$ on $X_i$ for each $i$, and let
$$\lambda_i:=\inf\{s\geq 0\mid K_{\Aa_i}+sC_i\text{ is nef}/U\}$$
    be the scaling numbers, and let $\lambda_0:=1$. Then:
    \begin{enumerate}
        \item If $\Aa$ is LD, then $(\Aa_i,sC_i)$ is LD for any $i$ and any $s\in [0,\lambda_{i-1}]$ unless $s=\lambda_{i-1}=1$ and $(\Aa,C)$ is not LD.
        \item Assume that one of the following conditions hold.
        \begin{enumerate}
            \item $\Aa$ is LD and $\Aa/U$ has a bs-weak lc model.
            \item $K_{\Aa}$ is not pseudo-effective$/U$.
            \item $\Aa/U$ has a KNQC bs-weak lc model.
        \end{enumerate}
Then either this MMP terminates, or $\lim_{i\rightarrow+\infty}\lambda_i>0$.
    \end{enumerate}
\end{thm}
\begin{proof}
Let $\Aa_i(s):=(\Aa_i,sC_i)$ for any $s\in [0,1]$.

(1) We apply induction on $i$. When $i=1$, (1) follows from  Lemma \ref{lem: combination of LD with lc g-pair} as $\Aa_1(0)$ is LD and $\Aa_1(1)$ is lc. By induction on $i$, we may assume that $(\Aa_i,sC_i)$ is LD for any $s\in [0,\lambda_{i-1})$ and $(\Aa_i,\lambda_{i-1}C_i)$ is LD when $\lambda_{i-1}<1.$

We have that $\phi_i$ is a step of a $(K_{\Aa_i}+sC_i)$-MMP$/U$ for any $s\in [0,\lambda_i)$ and is $(K_{\Aa_i}+\lambda_{i}C_i)$-trivial. Since $\lambda_i\leq\lambda_{i-1}$, $(\Aa_i,sC_i)$ is LD for any $s\in [0,\lambda_i)$, and $(\Aa_i,\lambda_iC_i)$ is LD when $\lambda_i<\lambda_{i-1}$, or $\lambda_i=\lambda_{i-1}<1$, or $\lambda_i=\lambda_{i-1}=1$ and $(\Aa,C)$ is LD. By Lemma \ref{lem: LD preserve under mmp}, $(\Aa_{i+1},sC_{i+1})$ is LD for any $s\in [0,\lambda_i)$, and $(\Aa_{i+1},\lambda_iC_{i+1})$ is LD when $\lambda_i<1$ or $(\Aa,C)$ is LD. We are done by induction on $i$.

(2) We may assume that this MMP does not terminate and $\lim_{i\rightarrow+\infty}\lambda_i=0$.

If (2.b) holds, then there exists $\lambda>0$ such that $K_{\Aa}+\lambda C$ is not pseudo-effective$/U$, hence $K_{\Aa_i}+\lambda C_i$ is not pseudo-effective$/U$ for any $i$. We have $\lim_{i\rightarrow+\infty}\lambda_i>\lambda>0$. 

If (2.a) or (2.c) holds, then by Lemmas \ref{lem: ld nef impleis nqc} and \ref{lem: ld minimal model is ld}, (2.c) holds, and we get a contradiction to \cite[Theorem 4.1]{Cas+25b}.
\end{proof}

\begin{lem}\label{lem: termination over open subset}
Let $\Aa/U$ be a $\mathbb Q$-factorial LD qdlt generalized pair, $A$ an ample$/U$ $\mathbb R$-divisor on $X$, and $U^0\subset U$ a non-empty open subset. Let $\Aa^0:=\Aa\times_UU^0$. Assume that $\Aa^0/U^0$ has a bs-weak lc model. Then we may run a $K_{\Aa}$-MMP$/U$ with scaling of $A$, and any such MMP
$$\phi_i: \Aa_i\dashrightarrow\Aa_{i+1}, \Aa_1:=\Aa$$
terminates with a model $\Aa_n/U$ such that $\Aa_n^0/U^0:=\left(\Aa_n\times_UU^0\right)/U^0$ is a $\mathbb Q$-factorial qdlt minimal model of $\Aa^0/U^0$.
\end{lem}
\begin{proof}
Write $A=\sum r_iH_i$ where each $r_i\in (0,1)$ and each $H_i$ is ample and globally generated$/U$. Possibly replacing $A$ with $\sum r_iH_i'$ where each $H_i'\in |H_i|_U$ is general, we may assume that $(\Aa,A)$ is lc $\mathbb Q$-factorial qdlt and $K_{\Aa}+A$ is nef$/U$. Let $X_i$ be the ambient variety of $\Aa_i$ and $A_i$ the image of $A$ on $X_i$ for each $i$. We may assume that this MMP does not terminate. Let
$$\lambda_i:=\inf\{t\mid t\geq 0, K_{\Aa_i}+tA_i\text{ is nef}/U\}$$
be the scaling numbers of this MMP and let $\lambda:=\lim_{i\rightarrow+\infty}\lambda_i$. 

Assume that $\lambda>0$. Since $X$ is $\mathbb Q$-factorial klt, by \cite[Lemma 3.4]{HL22}, $K_{\Aa}+\frac{\lambda}{2}A\sim_{\mathbb R,U}K_{X_1}+\Delta$ for some klt pair $(X_1,\Delta)$ such that $\Delta$ is big$/U$, and $\phi_i$ is a sequence of steps of a $(K_{X_1}+\Delta)$-MMP$/U$ with scaling of $A$, which terminates by \cite[Corollary 1.4.2]{BCHM10}, a contradiction. Therefore, $\lambda=0$. 

We let $\Aa_i^0:=\Aa_i\times_UU^0,X_i^0:=X_i\times_UU^0$ for each $i$, $\phi_i^0:=\phi_i|_{X^0}$, $A_i^0:=A_i|_{X^0}$, and let
$$\mathcal{N}:=\{i\in\mathbb N^+\mid \phi_i\text{ is not the identity map over }U^0\}.$$ 
Since $\Aa/U$ is $\mathbb Q$-factorial qdlt, $\Aa_i^0$ is $\mathbb Q$-factorial qdlt for any $i$. Since each $\phi_i$ does not extract any divisor and is $K_{\Aa_i}$-negative, the induced birational map $X\dashrightarrow X_i$ does not extract any divisor and is $K_{\Aa}$-negative for any $i$, hence the induced birational map $X^0\dashrightarrow X_i^0$ does not extract any divisor and is $K_{\Aa^0}$-negative for any $i$.

There are two cases:

\medskip

\noindent\textbf{Case 1}. $\mathcal{N}$ is a finite set. In this case, let $n:=\max\{j\mid j\in\mathcal{N}\}+1$. Then $\phi_i^0$ is the identity map over $U^0$ for any $i\geq n$. In this case, $K_{\Aa_i^0}+\lambda_iA_i^0$ is nef$/U^0$ for any $i\geq n$, hence $K_{\Aa_n^0}+\lambda_iA_n^0$ is nef$/U^0$ for any $i\geq n$. Since $\lambda=0$, $K_{\Aa_n^0}$ is nef$/U^0$. Thus $\Aa_n^0/U^0$ is a $\mathbb Q$-factorial qdlt minimal model of $\Aa^0/U^0$.

\medskip

\noindent\textbf{Case 2}. $\mathcal{N}$ is not a finite set. We may write $\mathcal{N}=\{n_i\}_{i=1}^{+\infty}$ such that $n_i<n_{i+1}$ for each $i$, then we get a sequence of induced birational maps
$$\Aa^0\cong \Aa^0_{n_1}\dashrightarrow \Aa^0_{n_2}\dashrightarrow\dots\dashrightarrow \Aa^0_{n_i}\dashrightarrow\dots,$$
which is a sequence of the $K_{\Aa^0}$-MMP$/U^0$ with scaling of $A^0$. Since $\Aa$ is LD, $\Aa^0$ is LD. Since $\Aa^0/U^0$ has a bs-weak lc model, by Theorem \ref{thm: eomm implies tof with scaling}, this MMP terminates, a contradiction.
\end{proof}

\subsection{A special MMP with strictly decreasing scaling numbers}

\begin{lem}\label{lem: trivial mmp}
Let $d$ be a positive integer and $\epsilon$ a positive real number. Then $\delta:=\frac{\epsilon}{2d+\epsilon}$ satisfies the following. Let $\Aa/U$ be an $\epsilon$-KNQC lc generalized pair of dimension $d$ with ambient variety $X$. Let $C$ be an $\mathbb R$-divisor on $X$ such that $(\Aa,C)/U$ is lc. Then for any $s\in [0,\delta)$, any sequence of steps of a $(K_{\Aa}+sC)$-MMP$/U$ $\phi: X\dashrightarrow X'$ is $K_{\Aa}$-trivial, $(\phi_*\Aa,\phi_*C)/U$ is lc, and $\phi_*\Aa/U$ is lc and $\epsilon$-KNQC.
\end{lem}
\begin{proof}
We may assume that $\phi$ is a single step of a $(K_{\Aa}+sC)$-MMP$/U$. Write $K_{\Aa}=\sum r_iH_i$ where each $r_i\geq\epsilon$ and each $H_i$ is nef$/U$ Cartier. We have
$$K_{\Aa}+C+\sum\left(\frac{1}{s}-1\right)r_iH_i=\frac{1}{s}(K_{\Aa}+sC)$$
and $\left(\frac{1}{s}-1\right)r_i>2d$ for any $i$. Then $\phi$ is also a step of a $(K_{\Aa}+C)$-MMP$/U$. By the length of extremal rays \cite[Theorem 2.2.1]{CHLX23}, $\phi$ is $H_i$-trivial for any $i$. By the contraction theorem \cite[Theorem 2.2.6]{CHLX23}, $\phi_*H_i$ is nef$/U$ Cartier for each $i$. Thus $\phi$ is $K_{\Aa}$-trivial and $\phi_*K_{\Aa}$ is $\epsilon$-NQC$/U$, hence $\phi_*\Aa/U$ is $\epsilon$-KNQC. Since $\phi$ is also a step of a $(K_{\Aa}+C)$-MMP$/U$, $(\phi_*\Aa,\phi_*C)/U$ is lc. The lemma follows.
\end{proof}

\begin{lem}\label{lem: construct mmp with scaling to 0}
    Let $\Aa/U$ be a $\mathbb Q$-factorial lc generalized pair with ambient variety $X$. Assume that $X$ is klt. Let $H\geq 0$ be an $\mathbb R$-divisor on $X$ such that $(\Aa,H)/U$ is lc, $K_{\Aa}+H$ is NQC$/U$, and $(\Aa,sH)/U$ has a KNQC bs-weak lc model for any $s\in (0,1]$. Then we may run a $K_{\Aa}$-MMP$/U$ with scaling of $H$ 
    $$\phi_i: \Aa_i\dashrightarrow \Aa_{i+1}, \Aa_1:=\Aa$$
    with scaling numbers 
    $$\lambda_i:=\inf\{t\mid t\geq 0, K_{\Aa_i}+tH_i\text{ is nef}/U\},$$
    where $X_i$ is the ambient variety of $\Aa_i$ and $H_i$ is the image of $H$ on $X_i$, such that either this MMP terminates, or $\lim_{i\rightarrow+\infty}\lambda_i=0$.
\end{lem}
\begin{proof}
We run the MMP in the following way. Suppose that we have already obtained a sequence of steps of $K_{\Aa}$-MMP$/U$ with scaling of $H$
$$\phi_i: \Aa_i\dashrightarrow \Aa_{i+1}, \Aa_1:=\Aa, 1\leq i\leq n-1,$$
$X_i$ is the ambient variety of $\Aa_i$, and $H_i$ is the image of $H$ on $X_i$ for each $i$. If $\lambda_n=0$ then we are done. When $\lambda_n>0$, we have that $K_{\Aa_n}+H_n$ is NQC$/U$. By Lemma \ref{lem: trivial mmp}, we may pick $0<\mu\ll 1$ so that any sequence of steps of a $(K_{\Aa_n}+(\lambda_n-\mu)H_n)$-MMP$/U$ is $(K_{\Aa_n}+\lambda_n H_n)$-trivial.

Since the induced birational map $X\dashrightarrow X_n$ is also a sequence of steps of a $(K_{\Aa}+(\lambda_n-\mu)H)$-MMP$/U$ and $(\Aa,(\lambda_n-\mu)H)/U$ has a good minimal model, $(\Aa_n,(\lambda_n-\mu)H_n)/U$ has a good minimal model. By \cite[Theorem 4.6]{Cas+25b}, the $(K_{\Aa_n}+(\lambda_n-\mu)H_n)$-MMP$/U$ with scaling of an ample divisor $\phi_{i}: X_i\dashrightarrow X_{i+1}, n\leq i\leq m$ terminates with a KNQC minimal model $(\Aa_m,(\lambda_n-\mu)H_m)/U$, where $\Aa_i$ and $H_i$ are the images of $\Aa_n$ and $H_n$ on $X_i$ for any $n\leq i\leq m$. We have that $\phi_i, n\leq i\leq m$ is a sequence of steps of a $K_{\Aa_n}$-MMP$/U$ with scaling of $H_n$. Thus $\phi_i, 1\leq i\leq m$ is a sequence of steps of a $K_{\Aa}$-MMP$/U$ with scaling of $H$. Since $K_{\Aa_m}+(\lambda_n-\mu)H_m$ is nef$/U$, we have $\lambda_m\leq\lambda_n-\mu<\lambda_n$.

By repeating this process, we obtain a sequence of steps of a 
$K_{\Aa}$-MMP$/U$ with scaling of $H$
$$\phi_i: \Aa_i\dashrightarrow \Aa_{i+1},\Aa_1:=\Aa$$
such that either this MMP terminates, or $\lambda_i>\lambda:=\lim_{j\rightarrow+\infty}\lambda_j$ for any $i$. If $\lambda>0$, then this also an infinite sequence of steps of a $(K_{\Aa}+\lambda H)$-MMP$/U$ with scaling of $H$ with scaling numbers $(\lambda_i-\lambda)$, and $\lim_{i\rightarrow+\infty}(\lambda_i-\lambda)=0$. This contradicts \cite[Theorem 4.1]{Cas+25b}. Thus $\lambda=0$.
\end{proof}

\subsection{Lifting MMP with scaling}

The goal of this section is to show that, for generalized pairs, MMP (with scaling) can be lifted to a sequence of steps of $\mathbb Q$-factorial (q)dlt MMP (with scaling). This important for proving the special termination in Section \ref{sec: spe tof} but is also of independent interest.

\begin{prop}\label{prop: lift mmp afs}
Let $\Aa/U$ be an lc generalized pair with ambient variety $X$ and $C$ an $\mathbb R$-divisor on $X$. Let $h: X'\rightarrow X$ be a $\mathbb Q$-factorial qdlt modification of $\Aa'$ and $\Aa':=h^*\Aa$.

Let
$$\phi_i: X_i\dashrightarrow X_i,\quad X_1:=X,\quad  1\leq i\leq n-1$$
be a sequence of birational maps$/U$ which does not extract any divisor, $\Aa_i,C_i$ the images of $\Aa,C$ on $X_i$ respectively. Assume that there exist projective morphisms$/U$
$$f_i: X_i\rightarrow T_i\quad \text{and}\quad f_i^+: X_{i+1}\rightarrow T_i,\quad 1\leq i\leq n-1$$
such that $f_i^+\circ\phi_i=f_i$, $-K_{\Aa_i}$ is ample$/T_i$, and $K_{\Aa_{i+1}}$ is ample$/T_i$. 

Then there exists a sequence of birational maps$/U$ 
$$\phi_i': X_i'\dashrightarrow X_{i+1}',\quad X_1':=X', \quad 1\leq i\leq n-1$$ and birational morphisms 
$$h_i: X_i'\rightarrow X_{i},\quad  h_1:=h,\quad 1\leq i\leq n-1$$ satisfying the following. Let $\Aa_i'$ be the image of $\Aa'$ on $\Aa_i'$, then for each $i$:
\begin{enumerate}
    \item $h_i$ is a $\mathbb Q$-factorial qdlt modification of $\Aa_i$ and $\Aa_i'=h_i^*\Aa_i$. Moreover, if $h$ is a $\mathbb Q$-factorial dlt modification of $\Aa$, then $h_i$ is a $\mathbb Q$-factorial dlt modification of $\Aa_i$.
    \item $h_{i+1}\circ\phi_i'=\phi_i\circ h_i$.
    \item $\phi_i'$ is a $K_{\Aa_i'}$-MMP$/T_i$ that is not the identity morphism unless $\phi_i$ is the identity morphism, and $\Aa'_{i+1}/T_i$ is a $\mathbb Q$-factorial good minimal model of $\Aa_i'/T_i$.
    \item $\{\phi_i'\}_{i=1}^{n-1}$ is a sequence of steps of a $K_{\Aa'}$-MMP$/U$.
    \item For each $i$, let
    $$\lambda_i:=\inf\{s\geq 0\mid K_{\Aa_i}+sC_i\text{ is nef}/U\}$$
    and assume that $K_{\Aa_i}+\lambda_iC_i\equiv_{T_i}0.$ Then:
    \begin{enumerate}
        \item $\phi_i'$ is a sequence of steps of a $K_{\Aa'}$-MMP$/U$ with scaling of $C':=h^*C$, and the scaling numbers of this MMP are all equal to $\lambda_i$.
        \item $\{\phi_i'\}_{i=1}^{n-1}$ is a sequence of  steps  of a $K_{\Aa'}$-MMP$/U$ with scaling of $C'$.
    \end{enumerate} 
    Here ``a sequence of steps" can be $0$ step.
\end{enumerate}
\end{prop}
\begin{proof}
Suppose that we have already constructed $\phi_i'$ for $1\leq i\leq j-1$ and $h_i$ for $i\leq j$ which satisfy our requirements. When $j=1$ this immediately from our assumption. If $j=n$ then we are done. Thus we may assume that $1\leq j\leq n-1$. In the following, we construct $\phi_j'$ and $h_{j+1}$ which satisfy our requirements. Let $C_j'$ be the image of $C'$ on $X_j'$.

Since $\phi_j$ is the ample model$/T_j$ of $K_{\Aa_j}$, $\phi_j\circ h_j$ is the ample model$/T_j$ of $K_{\Aa_j'}$. Since $\phi_j\circ h_j$ does not extract any divisor, by the negativity lemma, $\Aa_{j+1}/T_n$ is a semi-ample model of $\Aa_j'/T_n$. Since $\Aa_j'$ is $\mathbb Q$-factorial qdlt, $X_j'$ is $\mathbb Q$-factorial klt. By Theorem \ref{thm: bswlc imply mm}, we may run a $K_{\Aa_j'}$-MMP$/T_j$ with scaling of an ample divisor $A_j$ which terminates with a good minimal model $\Aa_{j+1}'/T_j$ of $\Aa_j'/T_j$. Let $X_{j+1}'$ be the ambient variety of $\Aa_{j+1}$. Since $\Aa_j'$ is $\mathbb Q$-factorial qdlt, $\Aa_{j+1}'/T_j$ is $\mathbb Q$-factorial qdlt, and if $\Aa_j'$ is dlt, then $\Aa_{j+1}'$ is dlt. Moreover, the induced birational map $h_{j+1}: X_{j+1}'\dashrightarrow X_{j+1}$ is the ample model$/T_j$ of $K_{\Aa_{j+1}'}$. Since $K_{\Aa_{j+1}'}$ is semi-ample$/T_j$, $h_{j+1}$ is a contraction and 
$$K_{\Aa_{j+1}'}=h_{j+1}^*\left(h_{j+1,*}K_{\Aa_{j+1}'}\right)=h_{j+1}^*K_{\Aa_{j+1}},$$
hence $\Aa_{j+1}'=h_{j+1}^*\Aa_j$. 

We let $\phi_j'$ be the induced birational map $X_j'\dashrightarrow X_{j+1}'$. If $\phi_j$ is not the identity morphism, then $\phi_{j}'$ is not the identity morphism, otherwise $h_j^*\Aa_j=h_{j+1}^*\Aa_{j+1}$ which contradicts that $-K_{\Aa_j}$ is ample$/T_j$ and $K_{\Aa_{j+1}}$ is ample$/T_j$.

Moreover, under the condition of (5), $K_{\Aa_j'}+\lambda_jC_j'\equiv_{T_j}0$ and 
$$K_{\Aa_j'}+\lambda_jC_j'=h_j^*(K_{\Aa_j}+\lambda_jC_j)$$
is nef$/U$. Thus $\phi_{j}'$ is also a $K_{\Aa_j'}$-MMP$/U$ with scaling of $C_j'$ and the scaling numbers are all equal to $\lambda_j$. 

Therefore, we have constructed $\phi_{j}'$ and $h_{j+1}$ which satisfy our requirements. The proposition follows by induction on $j$.
\end{proof}

\section{Special termination}\label{sec: spe tof}

The goal of this section is to establish a special termination result for LD generalized pairs which can be seen as generalizations of the special termination results in \cite[Section 4.2]{Hu25}. As we explained in the introduction, the special termination here is different from the special termination as in \cite{Fuj07} as the difficulty function set we construct may rely on the original generalized pair and inductive hypothesis. More precisely, we need the coefficient set $\Ii_k$ as in the following lemma.

\begin{lem}\label{lem: coefficient ld special termination}
Let $d$ be a positive integer and $1\leq k\leq d-1$ an integer. Let $\Aa/U=(X,B,\Mm)/U$ be a $\mathbb Q$-factorial dlt LD generalized pair. Let
$$\phi_i: \Aa_i\dashrightarrow \Aa_{i+1}, \Aa_1:=\Aa$$
be a sequence of steps of a $K_{\Aa}$-MMP$/U$. Assume that
\begin{enumerate}
    \item $\phi_i$ is an isomorphism near the generic point of any lc center of $\Aa_i$, and
    \item $\phi_i$ is an isomorphism near any lc center of $\Aa_i$ of dimension $\leq k-1$.
\end{enumerate}
Then there exists a finite set $\Ii_k\subset [0,1]$ satisfying the following. Write $\Aa_i=(X_i,B_i,\Mm)$ for each $i$. Let $S$ be an lc center of $\Aa$ of dimension $k$, $S_i$ the image of $S$ on $X_i$ for each $i$, and $\Aa_{S_i}:=\Aa_i|_{S_i}$. Write
$$\Aa_{S_i}=(S_i,B_{S_i},\Mm^i).$$
Then $B_{S_i}\in\Ii_k$.
\end{lem}
\begin{proof}
Since $\Aa$ is LD, there exist a positive integer $I$, $a_1,\dots,a_m\in (0,1]$ and lc generalized pairs $\Aa^j=(X,B^j,\Mm), 1\leq j\leq m$, such that $\sum_{j=1}^ma_j=1$, $\sum_{j=1}^ma_j\Aa^j=\Aa$, $\Aa^j$ is lc, and $IK_{\Aa^j}$ is a Weil divisor for each $j$. Since each $\Aa^j$ is lc, $\lfloor B\rfloor\subset\lfloor B^j\rfloor$ for any $j$, and $B^j\geq 0$ for any $j$. We let $\Aa_i^j,B_i^j$ be the images of $\Aa^j,B^j$ on $X_j$ respectively for any $i,j$. Then $IK_{\Aa_i^j}$ are Weil divisors for any $i,j,l$. Since $\Aa_{S_1}$ is dlt,
$$\mathcal{S}:=\{E\mid E\text{ is a prime divisor over }S_1,a(E,\Aa_{S_1})<0,\Center_{S_1}E\not\subset\lfloor B_{S_1}\rfloor\}$$
is a finite set. We let
$$\epsilon:=\inf\{a(E,\Aa_{S_1})\mid E\in\mathcal{S}\}+1=\min\{a(E,\Aa_{S_1})\mid E\in\mathcal{S}\}+1,$$
then $\epsilon>0$. By assumption (1), there exists a unique $\bb$-divisor $\Mm^S$ such that $\Mm^i=\Mm^S$ for any $i$. We let $g: S'\rightarrow S$ be a projective birational morphism such that $\Center_XE$ is a prime divisor for any $E\in\mathcal{S}$ and let 
$$\Ii:=\{\mult_E\Mm^S_{S'}\mid E\in\mathcal{S}\}.$$
Then $\Ii$ is a finite set.

We let $\phi_i: X_i\xrightarrow{f_i}Z_i\xleftarrow{f_i^+}X_{i+1}$ be each step of this MMP and let $S_i\xrightarrow{g_i} T_i\xleftarrow{g_i^+}S_{i+1}$ be the induced morphisms, where $g_i=f_i|_{S_i}$, $g_i^+=f_i^+|_{S_{i+1}}$, and $T_i$ is the normalization of $f_i(S_i)$. Then $K_{\Aa_{S_i}}$ is anti-ample$/T_i$ and $K_{\Aa_{S_{i+1}}}$ is ample$/T_i$, so by the negativity lemma, we have
$$a(E,\Aa_{S_i})\leq a(E,\Aa_{S_{i+1}})$$
for any prime divisor $E$ over $S=S_1$.

Let $D$ be an irreducible component of $B_{S_i}$ for some $i$ such that $\mult_DB_{S_i}<1$. Then $D$ is not contained in $\lfloor B_{S_i}\rfloor$. By assumption (2), so $\Center_{S_1}D$ is not contained in $\lfloor B_{S_1}\rfloor$. We have 
$$a(D,\Aa_{S_1})\leq a(D,\Aa_{S_i})<0,$$
which implies that $D\in\mathcal{S}$. Thus
$$\epsilon\leq a(D,\Aa_{S_1})+1\leq a(D,\Aa_{S_i})+1<1.$$
Let $N$ be the minimal positive integer such that $NP$ is Cartier near the generic point of $D$ on $X_i$ for any Weil divisor $P$ on $X$. Then
$$\left(K_{X_i}+\left\lfloor B_i\right\rfloor\right)|_{S_i}=K_{S_i}+\frac{N-1}{N}D$$
near the generic point of $D$, and so 
$$-a(D,\Aa_{S_i})=\mult_DB_{S_i}\geq \frac{N-1}{N},$$
hence
$$\epsilon\leq a(D,\Aa_{S_i})+1\leq\frac{1}{N}.$$
Thus $N\leq\frac{1}{\epsilon}$. Let $u_{i,j}$ be the unique real numbers such that
$$K_{S_i}+u_{i,j}D+\Mm^S_{S_i}=K_{\Aa_i^j}\Big|_{S_i}=\left(K_{X_i}+B_i^j+\Mm_{X_i}\right)\Big|_{S_i}$$
near the generic point of $D$. Then $u_{i,j}\geq 0$ for any $i,j$. Since $NIK_{\Aa_i^j}$ is Cartier near the generic point of $D$, we have
$$NI(u_{i,j}+\mult_D\Mm^S_{S_i})\in\mathbb Z.$$
Since 
$$\mult_D\Mm^S_{S_i}=\mult_{\Center_{S'}D}\Mm^S_{S'}\in\Ii,$$ 
$u_{i,j}$ belongs to a discrete set. We have
$$K_{S_i}+\sum_{j=1}^ma_ju_{i,j}D+\Mm^S_{S_i}=K_{\Aa_i}|_{S_i}=K_{\Aa_{S_i}}$$
near the generic point of $D$. Since $\Aa_{S_i}$ is lc, $\sum_{j=1}^ma_ju_{i,j}\leq 1$. Therefore, $u_{i,j}$ belongs to a bounded set, hence $u_{i,j}$ belongs to a finite set. Therefore,
$$\mult_{D}B_{S_i}=\sum_{j=1}^ma_ju_{i,j}$$
belongs to a finite set. The lemma follows.
\end{proof}

\begin{thm}\label{thm: special termination}
Let $\Aa/U$ be a $\Qq$-factorial LD dlt generalized pair and $C\geq 0$ an $\Rr$-divisor on $X$, such that $(\Aa,C)$ is LD and $K_{\Aa}+C$ is nef$/U$. Let
$$\phi_i: \Aa_i\dashrightarrow \Aa_{i+1}, \Aa_1:=\Aa$$
be a $K_{\Aa}$-MMP$/U$ with scaling of $C$, $X_i$ the ambient variety of $\Aa_i$ and $C_i$ the image of $C$ on $X_i$ for each $i$. Let
$$\lambda_i:=\inf\{t\mid t\geq 0, K_{\Aa_i}+tC_i\text{ is nef}/U\}$$
be the scaling numbers of this MMP. 

Assume that $\lim_{i\rightarrow+\infty}\lambda_i=0$. Assume that for any $i\gg 0$ and any lc center $S_i$ of $\Aa_i$ of dimension $\geq 1$, either $\Aa|_{S_i}$ has a bs-weak model, or $K_{\Aa|_{S_i}}$ is not pseudo-effective$/U$. Write $\Aa_i=(X_i,B_i,\Mm)$ for any $i$. Then $\phi_i$ is an isomorphism near $\lfloor B_i\rfloor$ for any $i\gg 0$.
\end{thm}

\begin{proof}
The proof generally follows from the same line of the proof of \cite[Proof of Theorem 4.2.1]{Fuj07} but we need to replace the difficulty coefficient set by the set given as in Lemma \ref{lem: coefficient ld special termination} inductively. For the reader's convenience, we provide a full proof here.

We may assume that this MMP does not terminate. Let $\Aa_i(s):=(\Aa_i,sC_i)$ for any $s\in [0,1]$ and any $i$ and let $\lambda_0:=1$. By Theorem \ref{thm: eomm implies tof with scaling}, $(\Aa_i,sC_i)$ is LD for any $s\in [0,\lambda_{i-1}]$ for any $i$. We let $\phi_i: X_i\xrightarrow{f_i}Z_i\xleftarrow{f_i^+}X_{i+1}$ be each step of this MMP where $f_i$ is the $K_{\Aa_i}$-negative extremal contraction$/U$. Then for $i\gg 0$, $\phi_i$ is a flip and $f_i,f_{i+1}$ are small. 

\medskip

\noindent\textbf{Step 1.} In this step we reduce the question to the termination near a single non-trivial lc center. Since there are only finitely many lc centers of $\Aa$, by comparison of discrepancies, the number of lc centers of $\Aa_i$ is decreasing, and if the exceptional locus of $f_i$ contains an lc center of $\Aa_i$, then the number of lc centers of $\Aa_{i+1}$ is strictly less than the number of lc centers of $\Aa_i$. Thus for any $i\gg 0$, $\phi_i$ is an isomorphism near the generic point of any lc center of $\Aa_i$. In particular, $\phi_i$ is an isomorphism near any $0$-dimensional lc center of $\Aa_i$ for any $i\gg 0$.

Now we only need to prove that the MMP terminates near the image of any non-trivial lc center $S$ of $\Aa$ on $X_i$ for any $i\gg 0$. We apply induction on $k:=\dim S<\dim X$. Since the $k=0$ case is already done, we may assume that $k>0$, and that $\phi_i$ is an isomorphism near any dimension $\leq k-1$ lc center of $\Aa_i$ for any $i\gg 0$. We let $S_i$ be the birational transform of $S$ on $X_i$ for any $i$, $T_i$ the normalization of $f_i(S_i)$, and $f_{S_i}: S_i\rightarrow T_i$, $f_{S_i}^+: S_{i+1}\rightarrow T_i$ the induced morphisms. Then $f_{S_i}$ and $f_{S_i}^+$ are birational for any $i\gg 0$ as $\phi_i$ is an isomorphism near the generic point of $S_i$. Moreover, for $i\gg 0$, $\Mm|_{S_i}$ is a fixed $\bb$-divisor, and we denote by $\Mm^S:=\Mm|_{S_i}$ for $i\gg 0$. Let $\Aa_{S_i}:=\Aa_i|_{S_i}$ for any $i$.

\medskip

\noindent\textbf{Step 2.} In this step we show that discrepancies of $\Aa_{S_i}$ increase for $i\gg 0$. By the negativity lemma, we have
$$a(E,\Aa_{S_i})\leq a(E,\Aa_{S_{i+1}})$$
for any $i\gg 0$ and any prime divisor $E$ over $S_i$. More precisely, let $W_i$ be the normalization of the main component of $X_i\times_{Z_i}X_{i+1}$ and let $p_i: W_i\rightarrow X_i$ and $q_i: W\rightarrow X_{i+1}$ be the induced birational morphisms, then we have
$$p_i^*K_{\Aa_i}-q_i^*K_{\Aa_{i+1}}:=F_i\geq 0$$
for some $F_i\geq 0$ by applying the negativity lemma for the induced birational morphism $W_i\rightarrow Z_i$. For any $i\gg 0$, since $\phi_i$ is an isomorphism near the generic point of $S_i$, $p_i$ and $q_i$ are isomorphisms over the generic point of $S_i$ and $S_{i+1}$ respectively, and we may let $S_{W_i}$ be the strict transform of $S_i$ on $W_i$ and let $p_{S_i}: S_{W_i}\rightarrow S_i$ and $q_{S_i}: S_{W_i}\rightarrow S_{i+1}$ be the induced birational morphisms. Then we have
$$p_{S_i}^*K_{\Aa_{S_i}}=p_i^*K_{\Aa_i}|_{S_{W_i}}\geq (p_i^*K_{\Aa_i}-F_i)|_{S_{W_i}}=q_i^*K_{\Aa_{i+1}}|_{S_{W_i}}=q_{S_i}^*K_{\Aa_{S_{i+1}}},$$
so $a(E,\Aa_{S_i})\leq a(E,\Aa_{S_{i+1}})$ for any prime divisor $E$ over $S_i$.

\medskip

\noindent\textbf{Step 3.} In this step we construct the difficulty function and use it to show that $\phi_{S_i}$ does not extract any divisor for any $i\gg 0$. Write $\Aa_{S_i}=(S_i,B_{S_i},\Mm^S)$ for any $i\gg 0$. By Lemma \ref{lem: coefficient ld special termination} and induction hypothesis, there exists a fintie set $\Ii_k$, such that for any $i\gg 0$, $B_{S_i}\in\Ii_k$. We define
$$d_i:=\sum_{\alpha\in\Ii}\#\{E\mid a(E,\Aa_{S_i})<1-\alpha\text{ and }\Center_{S_i}E\not\subset\lfloor B_{S_i}\rfloor\}.$$
Since $\Ii_k$ is a finite set and $\Aa_{S_i}$ is dlt for any $i\gg 0$, $d_i\geq 0$ and $d_i<+\infty$ for any $i\gg 0$. 

Let $\phi_{S_i}:=\phi_i|_{S_i}: S_i\dashrightarrow S_{i+1}$ for any $i\gg 0$. By induction hypothesis, $\phi_{S_i}$ is an isomorphism near $\lfloor B_{S_i}\rfloor$ for any $i\gg 0$. Therefore, for any $i\gg 0$ and any prime divisor $E$ over $S_i$, if $\Center_{S_i}(E)\subset\lfloor B_{S_i}\rfloor$ (resp. $\Center_{S_{i+1}}(E)\subset\lfloor B_{S_{i+1}}\rfloor$), then by the negativity lemma, $\phi_i$ (resp. $\phi_i^{-1}$) isomorphism near the generic point of $\Center_{S_i}(E)$ (resp. $\Center_{S_{i+1}}(E)$). Therefore, for any $i\gg 0$, $\Center_{S_i}(E)\subset\lfloor B_{S_i}\rfloor$ if and only if $\Center_{S_{i+1}}(E)\subset\lfloor B_{S_{i+1}}\rfloor$. By \textbf{Step 2}, we have $d_n\geq d_{n+1}$.

Moreover, for $i\gg 0$, assume that $f_{S_{i}}^+$ contracts a divisor $E$. Then $\mult_EB_{S_{i+1}}=\alpha\in\Ii_k$ for some $\alpha$. Since $K_{\Aa_{S_i}}$ is anti-ample$/T_i$ and $K_{\Aa_{S_{i+1}}}$ is ample$/T_i$, by \cite[Lemma 3.38]{KM98}, we have
$$a(E,\Aa_i)<a(E,\Aa_{i+1})=1-\alpha,$$
hence $d_i>d_{i+1}$ in this case.

Therefore, for any $i\gg 0$, $f_{S_{i}}^+$ does not contract any divisor, hence $\phi_{S_i}$ does not extract any divisor.

\medskip

\noindent\textbf{Step 4.} Fix $n\gg 0$. Let $g_n: Y_n\rightarrow S_n$ be a $\mathbb Q$-factorial dlt modification of $\Aa_{S_n}$ and let $\Aa_{Y_n}:=g_n^*\Aa_{S_n}$. By \textbf{Step 3}, $\phi_{S_i}$ does not extract any divisor for any $i\geq n$. Since $\phi_{S_i}$ is not an isomorphism for infinitely many $i\geq n$, by Proposition \ref{prop: lift mmp afs}, there exist birational morphisms $g_i: Y_i\rightarrow S_i$, $i\geq n$ of $\Aa_i$ for any $i\geq n$ and birational maps $\psi_i: Y_i\dashrightarrow Y_{i+1}$, $i\geq n$, such that for any $i\geq n$,
\begin{itemize}
    \item $g_i$ is a $\mathbb Q$-factorial dlt modification of $\Aa_{S_i}$,
    \item $g_{i+1}\circ\psi_i=\phi_{S_i}\circ g_i$ and 
    $$\Aa_{Y_{i+1}}:=g_{i+1}^*\Aa_{S_{i+1}}=\psi_{i,*}\Aa_{Y_i}$$
    for any $i\geq n$, and
    \item $\psi_i$ is a sequence of steps of a $K_{\Aa_{Y_n}}$-MMP$/U$ with scaling of $g_i^*\left(C_i|_{S_i}\right)$ and the scaling numbers of this MMP is $\lambda_i$.
\end{itemize}
Since $\lim_{i\rightarrow+\infty}\lambda_i=0$, $\{\psi_i\}_{i=n}^{+\infty}$ is a sequence of steps of a $K_{\Aa_{Y_n}}$-MMP$/U$ with scaling of $g_n^*\left(C_n|_{S_n}\right)$ and the limit of the scaling numbers of this MMP is $0$.

Since $\Aa_n$ is LD, by Lemma \ref{lem: ld adjunction}, $\Aa_{S_n}$ is LD. By Lemma \ref{lem: LD under dlt model}, $\Aa_{Y_n}$ is LD. Since either $\Aa_{S_n}/U$ has a bs-weak lc model or $K_{\Aa_{S_n}}$ is not pseudo-effective$/U$, by Lemma \ref{lem: Cas+25b 3.13}, either $\Aa_{Y_n}/U$ has a bs-weak lc model, or $K_{\Aa_{Y_n}}$ is not pseudo-effective$/U$. This contradicts Theorem \ref{thm: eomm implies tof with scaling}.
\end{proof}

\section{Reduction via Iitaka fibration}\label{sec: abundant gmm}

In this section, we use the relative Iitaka fibration to reduce the existence of good minimal models for generalized pairs to the abundant case.  We first treat the case of relative Iitaka dimension zero and obtain good minimal models under additional hypotheses.  We then establish the existence of good minimal models for klt abundant generalized pairs, a key input for the gluing theory in Section~\ref{sec: glue}.

\subsection{Good minimal model with relative Iitaka dimension zero}

In this subsection, we study the existence of good minimal models for generalized pairs $\Aa/U$ admitting a contraction$/U$ $\pi: X\rightarrow Z$ such that $\kappa_{\sigma}(X/Z,K_{\Aa})=\kappa_{\iota}(X/Z,K_{\Aa})=0$, where $X$ is the ambient variety of $\Aa$. We prove the existence of good minimal models for such generalized pairs in some cases.

\begin{lem}\label{lem: has19 3.2 step 3 abu ver}
    Let $\Aa/U=(X,B,\Mm)/U$ be a $\mathbb Q$-factorial lc generalized pair such that $X$ is klt and $\pi: X\rightarrow U$ the induced morphism. Assume that:
    \begin{enumerate}
        \item $\pi$ is an equidimensional contraction.
        \item $U$ is $\mathbb Q$-factorial.
        \item $\kappa_{\sigma}(X/U,K_{\Aa})=\kappa_{\iota}(X/U,K_{\Aa})=0$.
    \end{enumerate}
    Then $\Aa/U$ has a $\mathbb Q$-factorial good minimal model $\Aa'/U$ such that $K_{\Aa'}\sim_{\mathbb R,U}0$.
\end{lem}
\begin{proof}
It is a special case of \cite[Proposition 11.2.1(2.b)]{CHLX23} by letting $\Ff=T_X$.
\end{proof}

\begin{lem}\label{lem: abundant+trivial+lcc dominate imply gmm}
    Let $\Aa/U$ be an lc generalized pair with ambient variety $X$ and $f: X\rightarrow Z$ a contraction$/U$. Assume that:
\begin{enumerate}
    \item $\kappa_{\sigma}(X/U,K_{\Aa})=\dim Z-\dim U$.
    \item $K_{\Aa}\sim_{\mathbb R,Z}0$.
    \item Any lc center of $\Aa$ dominates $Z$.
\end{enumerate}
Then $\Aa/U$ has a good log minimal model.
\end{lem}
\begin{proof}
By Lemma \ref{lem: g-pair weak glc imply lmm}, we only need to show that $\Aa/U$ has a bs-semi-ample model. Let $h: X'\rightarrow X$ be a $\mathbb Q$-factorial dlt modification of $\Aa$ and let $\Aa':=h^*\Aa$. By Lemma \ref{lem: mm preserved under dlt model}, possibly replacing $\Aa$ with $h^*\Aa$, we may assume that $\Aa$ is $\mathbb Q$-factorial dlt. By \cite[Theorem 2.3.2]{CHLX23}, there exists a klt generalized pair $\Aa_Z/U$ such that $K_{\Aa}\sim_{\mathbb R}f^*K_{\Aa_Z}$. By Lemma \ref{lem: property of numerical and Iitaka dimension}(1), $K_{\Aa_Z}$ is big$/U$. By \cite[Theorem 2.1.1]{Cas+25a}, we may run a $K_{\Aa_Z}$-MMP$/U$ with scaling of an ample divisor $\phi: Z\dashrightarrow Z'$ which terminates with a good minimal model $\Aa_{Z'}/U$ of $\Aa_Z/U$. We let $p: Y\rightarrow Z$ and $q: Y\rightarrow Z'$ be a resolution of indeterminacy of $\phi$, then we have
$$p^*K_{\Aa_Z}=q^*K_{\Aa_{Z'}}+F$$
for some $F\geq 0$ that is exceptional$/Z'$. We let $h: \Aa'\rightarrow\Aa$ be a proper log smooth model of $\Aa$, $X'$ the ambient variety of $\Aa'$, such that the induced map $f': X'\dashrightarrow Y$ is a morphism. We have
$$\Aa'=(h^*\Aa,E)$$
for some $h$-exceptional $\mathbb R$-divisor $E\geq 0$. Then
\begin{equation}\label{equ: aa' and z'}
    K_{\Aa'}=h^*K_{\Aa}+E\sim_{\mathbb R}(f\circ h)^*K_{\Aa_Z}+E=f'^*p^*K_{\Aa_Z}+E=f'^*q^*K_{\Aa_{Z'}}+f'^*F+E.
\end{equation}
Since $\phi$ does not extract any divisor, there exist open subsets $Z_0\subset Z$ and $Z_0'\subset Z'$ such that $\phi_0:=\phi|_{Z_0}: Z_0\dashrightarrow Z_0'$ is an isomorphism, and $\dim Z'\backslash Z_0'\leq \dim Z-2$ (here we consider $\dim\emptyset=-\infty$). We run a $K_{\Aa'}$-MMP$/Z'$ with scaling of an ample divisor
$$\phi_i: \Aa_i\dashrightarrow \Aa_{i+1},\quad \Aa_1:=\Aa'$$
and let $\Aa_i^0:=\Aa_i\times_{Z'}Z_0$ for each $i$. Then there exists $i\gg 0$ such that $K_{\Aa_i}$ is movable$/Z'$. Let $X_i$ be the ambient variety of $\Aa_i$, $X_i^0:=X_i\times_{Z'}Z_0$, $E_i$ the image of $E$ on $X_i$, $E_i^0:=E_i|_{X_i^0}$, $F_i$ the image of $f'^*F$ on $X_i$, and $f_i: X_i\dashrightarrow Z'$ the induced morphism for each $i$.  

We have that $K_{\Aa_i^0}$ is movable$/Z_0$. By \cite[Lemma 3.4(2)]{LX25}, we have
$$N_{\sigma}(X_1^0/Z_0,K_{\Aa_1^0})=N_{\sigma}\left(X_1/Z_0,\left(h^*K_{\Aa}+E\right)|_{X_1^0}\right)=E^0_1.$$
By \cite[Lemma 2.25]{LMX24}, $E^0_1$ is contracted by the induced birational map $X_1^0\dashrightarrow X_i^0$. Thus $E_i|_{X_i^0}=0$. Since $\dim Z'\backslash Z_0'\leq \dim Z-2$, $\dim f_i(\Supp E_i)\leq\dim Z-2$. Thus $\dim f_i(\Supp F_i\cup\Supp E_i)\leq\dim Z-2$, so $F_i+E_i$ is very exceptional$/Z'$. We have
$$K_{\Aa'}\sim_{\mathbb R,Z'}F_i+E_i\geq 0$$
and $K_{\Aa'}$ is movable$/Z'$. By \cite[Lemma 3.3]{Bir12}, $F_i=E_i=0$ and $K_{\Aa_i}\sim_{\mathbb R,Z'}0$. By (\ref{equ: aa' and z'}), 
$$K_{\Aa_i}\sim_{\mathbb R}f_i^*K_{\Aa_{Z'}}$$
is semi-ample$/U$. Thus $\Aa_i/U$ is a good minimal model of $\Aa/U$. The lemma follows.
\end{proof}

\begin{prop}\label{prop: abundant+lcc dominate imply gmm}
    Let $\Aa/U$ be an lc generalized pair with ambient variety $X$ and $f: X\rightarrow Z$ a contraction$/U$. Assume that
\begin{enumerate}
    \item $\kappa_{\sigma}(X/U,K_{\Aa})=\dim Z-\dim U$.
    \item $\kappa_{\iota}(X/Z,K_{\Aa})=\kappa_{\sigma}(X/Z,K_{\Aa})=0$.
    \item Any lc center of $\Aa$ dominates $Z$.
\end{enumerate}
Then $\Aa/U$ has a good log minimal model.
\end{prop}
\begin{proof}
    By Definition-Lemma \ref{deflem: eoltm}, there exists a proper log toroidal model $h: \Aa'\rightarrow\Aa$ of $\Aa$ with respect to $f$ associated with $f': X'\rightarrow Z'$. By Lemma \ref{lem: property of numerical and Iitaka dimension}(3), we have $$\kappa_{\iota}(X'/Z',K_{\Aa'})=\kappa_{\iota}(X/Z,K_{\Aa})=0=\kappa_{\sigma}(X/Z,K_{\Aa})=\kappa_{\sigma}(X'/Z',K_{\Aa'}).$$
    By Lemma \ref{lem: has19 3.2 step 3 abu ver} and Theorem \ref{thm: bswlc imply mm}, we may run a $K_{\Aa'}$-MMP$/Z'$ with scaling of an ample divisor which terminates with a good minimal model $\Aa''/Z'$ of $\Aa'/Z'$. By Lemma \ref{lem: property of numerical and Iitaka dimension}(5), 
    $$\kappa_{\sigma}(X''/U,K_{\Aa''})=\kappa_{\sigma}(X'/U,K_{\Aa'})=\kappa_{\sigma}(X/U,K_{\Aa})=\dim Z-\dim U.$$
    Since any lc place of $\Aa'$ is an lc place of $\Aa$ hence dominates $Z$, any lc center of $\Aa'$ dominants $Z'$, so any lc center of $\Aa''$ dominants $Z'$. By Lemma \ref{lem: abundant+trivial+lcc dominate imply gmm}, $\Aa''/U$ has a good log minimal model. By Lemmas \ref{lem: minimal model same after running mmp} and \ref{lem: Bir12 2.8}, $\Aa/U$ has a bs-semi-ample model. By Lemma \ref{lem: g-pair weak glc imply lmm}, $\Aa/U$ has a good log minimal model.
\end{proof}

\subsection{Good minimal model for abundant generalized pairs}

The goal of this subsection is to show that klt abundant pseudo-effective generalized pairs have good minimal models. This will be used later in establishing the gluing theory. We refer the reader to \cite[Definition 2.2]{Hu23} for the definition of invariant Iitaka fibrations (see also \cite[Definition 6.1]{CHL24}).

\begin{lem}\label{lem: iitaka fibration numerical abundant divisor gpair dimension}
Let $\Aa/U$ be an lc generalized pair with ambient variety $X$ such that $K_{\Aa}$ is abundant$/U$ and pseudo-effective$/U$. Let $\phi: X\dashrightarrow Y$ be an invariant Iitaka fibration$/U$ of $K_{\Aa}$ and let $h: \Aa'\rightarrow\Aa$ be a proper log toroidal model with respect to $\phi$ and associated with $f: X'\rightarrow Z$. Then
$$\kappa_{\sigma}(X'/U,K_{\Aa'})=\kappa_{\iota}(X'/U,K_{\Aa'})=\dim Z-\dim U,\ \kappa_{\sigma}(X'/Z,K_{\Aa'})=\kappa_{\iota}(X'/Z,K_{\Aa'})=0.$$
\end{lem}
\begin{proof}
Let $\pi: X\rightarrow U$ be the associated morphism and $X\rightarrow U'\rightarrow U$ the Stein factorization of $\pi$. Possibly replacing $U$ by $U'$, we may assume that $\pi$ is a contraction. Let $\kappa:=\kappa_{\sigma}(X/U,K_{\Aa})$. Then $\kappa\geq 0$. Since $K_{\Aa}$ is pseudo-effective$/U$ and abundant$/U$ and $\phi$ is an invariant Iitaka fibration$/U$ of $K_{\Aa}$, we have
$$\dim Z-\dim U=\kappa_{\iota}(X/U,K_{\Aa})=\kappa_{\sigma}(X/U,K_{\Aa})=\kappa.$$
Since $\Aa'$ is a log toroidal model of $\Aa$, by Lemma \ref{lem: property of numerical and Iitaka dimension}(3), we have
$$\kappa_{\sigma}(X'/U,K_{\Aa'})=\kappa_{\sigma}(X/U,K_{\Aa})=\dim Z-\dim U=\kappa_{\iota}(X/U,K_{\Aa})=\kappa_{\iota}(X'/U,K_{\Aa'}).$$
Thus we only need to show that 
$$\kappa_{\sigma}(X'/Z,K_{\Aa'})=\kappa_{\iota}(X'/Z,K_{\Aa'})=0.$$
We have that
$\Aa'=(h^*\Aa,E)$
for some $E\geq 0$ that is exceptional$/X$. Moreover, there exists an ample$/U$ $\mathbb R$-divisor $A\geq 0$ on $Z$ and an exceptional$/X$ $\mathbb R$-divisor $E'\geq 0$ on $X'$, such that
$$h^*K_{\Aa}\sim_{\mathbb R,U}f^*A+E'.$$
Then for any positive real number $k$, we have
$$K_{\Aa'}+kf^*A\sim_{\mathbb R,U}(1+k)f^*A+E+E'.$$
Since
$$(1+k)f^*A+E+E'\geq f^*A+E+E'\geq \frac{1}{1+k}((1+k)f^*A+E+E')),$$
and $E$ is exceptional$/X$, for any $k\geq 0$, by Lemma \ref{lem: property of numerical and Iitaka dimension}(2)(3),
\begin{align*}
    &\kappa_{\sigma}(X'/U,K_{\Aa'}+kf^*A)=\kappa_{\sigma}(X'/U,(1+k)f^*A+E+E')\\
    =&\kappa_{\sigma}(X'/U,f^*A+E+E')=\kappa_{\sigma}(X'/U,h^*K_{\Aa}+E)=\kappa_{\sigma}(X'/U,K_{\Aa'})=\kappa_{\sigma}(X/U,K_{\Aa})=\kappa.
\end{align*}
Let $F$ be a general fiber of $\pi\circ h$, $Z_F:=f(F)$, and $f_F:=f|_F$. By \cite[(3.3)]{Fuj20} and Lemma \ref{lem: property of numerical and Iitaka dimension}(1), we have
 \begin{align*}
\kappa &= \kappa_{\sigma}(X'/U,K_{\Aa'}+kf^*A)=\kappa_{\sigma}((K_{\Aa'}+kf^*A)|_F)\geq \kappa_{\sigma}(F/Z_F,K_{\Aa'}|_F)+\kappa(A|_{F_Z})\\
    &=\kappa_{\sigma}(F/Z_F,K_{\Aa'}|_F)+\dim Z_F=\kappa_{\sigma}(F/Z_F,K_{\Aa'}|_F)+\kappa=\kappa_{\sigma}(X'/Z,K_{\Aa'})+\kappa.
\end{align*}
Therefore, 
$\kappa_{\sigma}(X'/Z,K_{\Aa'})=0$. Since $K_{\Aa}$ is abundant$/U$ and pseudo-effective$/U$, $K_{\Aa'}$ is abundant$/U$ and pseudo-effective$/U$. Thus 
$$0\leq\kappa_{\iota}(X'/Z,K_{\Aa'})\leq\kappa_{\sigma}(X'/Z,K_{\Aa'})=0,$$
so $\kappa_{\iota}(X'/Z,K_{\Aa'})=0$ and the lemma follows.
\end{proof}

\begin{lem}\label{lem: canoincal model implies glmm}
    Let $\Aa/U$ be a klt generalized pair such that $K_{\Aa}$ is abundant$/U$ and pseudo-effective$/U$. Then $\Aa/U$ has a $\mathbb Q$-factorial good minimal model.
\end{lem}
\begin{proof}
Let $\Aa:=(X,B,\Mm)$. Possibly replacing $X$ by a small $\mathbb Q$-factorialization, we may assume that $X$ is $\mathbb Q$-factorial klt. Since $K_{\Aa}$ is abundant$/U$ and pseudo-effective$/U$, $\kappa:=\kappa_{\iota}(X/U,K_{\Aa})\geq 0$. Let $\phi: X\dashrightarrow Y$ be an invariant Iitaka fibration$/U$ (cf. \cite[Definition 6.1]{CHL24}) of $K_{\Aa}$. By Definition-Lemma \ref{deflem: eoltm}, there exists a proper log toroidal model $h: \Aa'\rightarrow\Aa$ with respect to $\phi$ associated with $f: X'\rightarrow Z$. Then $\Aa'$ is klt. By Lemma \ref{lem: iitaka fibration numerical abundant divisor gpair dimension} and Proposition \ref{prop: abundant+lcc dominate imply gmm}, $\Aa'/U$ has a good log minimal model. By Lemma \ref{lem: Bir12 2.8}, $\Aa/U$ has a bs-semi-ample model. By Theorem \ref{thm: bswlc imply mm}, $\Aa/U$ has a $\mathbb Q$-factorial good minimal model.
\end{proof}

\section{Gluing theory for generically trivial generalized pairs}\label{sec: glue}

In this section, we establish the gluing theory for LD generalized pairs. We will adopt the notation as in \cite[Definitions 9.15, 9.16, 9.18]{Kol13} for the definition of stratification, stratified schemes, boundary of stratified schemes, stratified morphisms, stratified relations, and properties (N), (SN), (HN), (HSN).

Our main output is Theorem~\ref{thm: bir12 1.7 g LD}, which can be viewed as an LD analogue of \cite[Theorem~1.7]{Bir12} and will be used in Section~\ref{sec: hx13 1.1}.

\subsection{Gluing theory of lc crepant structures}

\begin{defn}
Let $\Aa/U$ be a generalized pair associated with morphism $\pi: X\rightarrow U$. We say that $\pi: \Aa\rightarrow U$ is an \emph{lc} (resp. \emph{dlt}) \emph{crepant log structure} if $\pi$ is a contraction, $\Aa$ is lc (resp. dlt), and $K_{\Aa}\sim_{\mathbb R,U}0$. An \emph{lc center} of a crepant log structure $\pi: \Aa\rightarrow U$ is the image of an lc center of $\Aa$ in $U$.
\end{defn}

\begin{defn}[{\cite[Definition 13.1.4]{CHLX23}}]
Let $\pi: \Aa\rightarrow U$ be an lc crepant log structure. For any integer $0\leq k\leq\dim U$, we denote by $S_k^*(U,\Aa)$ the union of all lc centers of $\pi: \Aa\rightarrow U$ of dimension $\leq k$. We define $S_{-1}^*(U,\Aa):=\emptyset$ and let $S_k(U,\Aa):=S_k^*(U,\Aa)\backslash S_{k-1}^*(U,\Aa)$ for any $0\leq k\leq \dim U$. 

The \emph{lc stratification} of $U$ induced by $\pi: \Aa\rightarrow U$ is the stratification induced by $S_*(U,\Aa)$ and is denoted by $(U,S_*(\Aa))$, i.e.\ we have $S_k(\Aa)(U)=S_k(U,\Aa)$ for any $k$. We denote by 
$$B(U,\Aa):=U\setminus S_{\dim U}(U,\Aa).$$
\end{defn}

\begin{defn}[{\cite[Definition 13.1.5]{CHLX23}}]\label{defn: of glc origin}
A \emph{semi-normal stratified space} $(Y,S_*)$ is called \emph{of lc origin} if $S_i(Y)$ is unibranch for any $i$, and there exist lc crepant log structures $\pi_j: \Aa_j\rightarrow U_j$ and a finite surjective stratified morphism
$$\phi: \bigsqcup_j(U_j,S_*(\Aa_j))\rightarrow (Y,S_*).$$
\end{defn}

The following two constructions and one lemma are almost verbatim to \cite[Constructions 4.12 and 4.13, Lemma 4.14]{LX23} except that we consider LD generalized pairs instead of NQC generalized pairs. For the reader's convenience, we provide the full statements here.

\begin{cons}\label{cons: gluing part 1}
Let $\Aa/U:=(X,B,\Mm)/U$ be a dlt generalized pair, $W\subset\lfloor B\rfloor$ a reduced divisor, $\pi: W^n\rightarrow W$ the normalization of $W$, $D$ the double locus of $W^n$, $D^n$ the normalization of $D$, $\tau: D^n\rightarrow D^n$ the induced involution, and $(\tau_1,\tau_2): D^n\rightrightarrows W^n$ a finite stratified equivalence relation whose normalization map is given by the quotient morphism $\pi: W^n\to W=W^n/R$, where $R$ is the finite equivalence relation generated by $D^n$. Let $L_W:=K_{\Aa}|_W$, $$L:=K_{\Aa}|_{W^n}=K_{\Aa_W^n}$$
where $\Aa_W^n:=\Aa|_{W^n}:=\left(W^n,B_{W^n},\Mm^{W^n}\right)$, and suppose that $L$ is semi-ample$/U$. Let $g^n: W^n\rightarrow Y^n$ and $h^n: D^n\rightarrow T^n$ be the morphisms$/U$ induced by $L$ and $L|_{D^n}$ respectively so that we have the commutative diagram
\begin{displaymath}
    \xymatrix{ 
        D^n \ar[dd]_{h^n}\ar@<.5ex>[rr]^{\tau_1} \ar@<-.5ex>[rr]_{\tau_2} && W^n \ar[dd]^{g^n} \\
        &&\\
        T^n \ar@<.5ex>[rr]^{\sigma_1}\ar@<-.5ex>[rr]_{\sigma_2} && Y^n 
    }
\end{displaymath}
where $(\sigma_1,\sigma_2): T^n\rightrightarrows Y^n$ are induced by $(\tau_1,\tau_2): D^n\rightrightarrows W^n$. We let 
$$\Aa_D^n:=\Aa_W^n|_{D^n}:=\left(D^n,B_{D^n},\Mm^{D^n}\right).$$
It is clear that $g^n: \Aa_W^n\to Y^n$ and  $h^n: \Aa_D^n\rightarrow T^n$ are dlt crepant log structures. We let $(Y^n,S_*(\Aa_W^n))$ and $(T^n,S_*(\Aa_D^n))$ be their induced stratified schemes respectively.
\end{cons}

\begin{cons}\label{cons: glue part 2}
Notations and conditions as in Construction \ref{cons: gluing part 1}. Assume that $K_{\Aa}$ is $\mathbb Q$-Cartier. Let $m$ be a sufficiently divisible positive integer such that $mL_W$ is Cartier, $|mL/U|$ defines $g^n$, and there exists a very ample$/U$ divisor $H$ on $Y^n$ such that $(g^{n})^*H=M:=mL$.

Let $p_W: W^n_M\to W^n$, $p_Y: Y^n_H\to Y^n$ be the total spaces of the line bundles $M$ and $H$ respectively. Let $B_{W^n_M}:=p_W^{-1}(B_{W^n})$, $\Aa_{W_M}^n:=(W^n_M,B_{W^n_M},p^*_W\Mm^{W^n})$, and $g^n_M: \Aa_{W_M}^n\to Y^n_H$ the dlt crepant log structure with induced stratification $(Y^n_H, S_*(\Aa_{W_M}^n):=p_Y^{-1}S_*(\Aa_{W}^n))$. 

Let $p_D: D^n_M\to D^n$ and $p_T: T^n_H\to T^n$ be the total spaces of the line bundles $M|_{D^n}$ and $H|_{T^n}$. Let $B_{D^n_M}:=p_D^{-1}(B_{D^n})$, $\Aa_{D_M}^n:=(D^n_M,B_{D^n_M},p^*_D\Mm^{D^n})$, and $h^n_M: \Aa_{D_M}^n\to T^n_H$ the dlt crepant log structure with induced stratification $(T^n_H, S_*( \Aa_{D_M}^n):=p_T^{-1}S_*( \Aa_{D}^n))$. 

Then we have a finite pre-relation $(\sigma_{1H},\sigma_{2H}): T^n_H\rightrightarrows Y^n_H$ induced by the finite relation $(\tau_{1M},\tau_{2M}): D^n_M\rightrightarrows W^n_M$, where $\tau_{1M},\tau_{2M}: D^n_M\rightarrow W^n_M$ are liftings of $\tau_{1},\tau_{2}$ respectively.
\end{cons}

\begin{lem}[{cf.~\cite[Lemma 3.11]{HX13}}]\label{lem: induced relation is stratified}
Notations and conditions as in Construction \ref{cons: gluing part 1}. Then
\begin{enumerate}
    \item $(\sigma_1,\sigma_2): T^n\rightrightarrows Y^n$ gives a stratified equivalence relation, and
    \item $(Y^n,S_*(\Aa_W^n))$ and $(T^n,S_*(\Aa_D^n))$ satisfy (HN) and (HSN).
\end{enumerate}
If we have the additional notations and conditions as in Construction \ref{cons: glue part 2}, then\begin{enumerate}
    \item[(3)] $(\sigma_{1H},\sigma_{2H}): T^n_H\rightrightarrows Y^n_H$ gives a stratified equivalence relation, and
    \item[(4)] $(Y^n_H, S_*(\Aa_{W_M}^n))$ and $(T^n_H, S_*(\Aa_{D_M}^n))$ satisfy (HN) and (HSN).
\end{enumerate}
\end{lem}
\begin{proof}
(2)(4) follow from \cite[Theorem 13.2.5]{CHLX23}. We prove (1)(3). For any lc center $V$ of $\Aa_D^n$ (resp. of $\Aa_{D_M}^n$), $\tau(V)$ (resp. $\tau_M(V)$) is also an lc center on $D^n$ (resp. $D^n_M$). Thus the lc stratification induced by $h^n: \Aa_D^n\to T^n$ (resp. $h^n_M: \Aa_{D_M}^n\to T^n_H$) is the same as the lc stratification induced by $h^n\circ\tau: \Aa_D^n\to T^n$ (resp. $h^n_M\circ\tau_M: \Aa_{D_M}^n\to T^n_H$). Hence we only need to check that $\sigma^{-1}S_*(\Aa_W^n)$ (resp. $\sigma_H^{-1}S_*(\Aa_{W_M}^n)$) coincides with $S_*(\Aa_D^n)$ (resp. $S_*(\Aa_{D_M}^n)$), where $\sigma$ (resp. $\sigma_H$) is the canonical morphism $T^n\to Y^n$ (resp. $T^n_H\to Y^n_H$). But this follows directly from \cite[Lemma 13.2.4]{CHLX23}. 
\end{proof}

\subsection{Gluing to lc strata}

\begin{thm}\label{thm: semi-ample over U0 implies semi-ample over U}
Let $\Aa/U=(X,B,\Mm)/U$ be an $\mathbb Q$-factorial LD dlt generalized pair, $U^0$ a non-empty subset of $U$, $W:=\lfloor B\rfloor$, and $\Aa^0:=\Aa\times_UU^0$. Assume that:
\begin{enumerate}
\item $K_{\Aa}$ is a $\mathbb Q$-divisor.
    \item $K_{\Aa^0}$ is semi-ample$/U^0$.
    \item The image of any lc center of $\Aa$ in $U$ intersects $U^0$.
    \item $K_{\Aa}|_S$ is semi-ample$/U$ for any irreducible components $S$ of $W$.
\end{enumerate}
Then $K_{\Aa}|_W$ is semi-ample$/U$.
\end{thm}
\begin{proof}
Let $X^:=X\times_UU^0$, $W^0:=W\cap X^0$, $L_W:=K_{\Aa}|_W$,  $L_{W^0}:=L_W|_{W^0}$, and $L:=L_W|_{W^n}$, where $W^n$ is the normalization of $W$. By assumption (3), $L$ is semi-ample$/U$. Let $g^n: W^n\rightarrow Y^n$ be the contraction$/U$ induced by $L$ and let $g^0: W^0\rightarrow Z^0$ be the contraction$/U$ induced by $L_{W^0}$. Then there exists a sufficiently divisible positive integer $m$, such that
\begin{itemize}
    \item $mK_{\Aa}$ is Cartier, 
    \item $|mK_{\Aa^0}|$ is globally generated$/U^0$,
    \item there exists a very ample$/U$ divisor $H$ on $Y^n$ such that $(g^n)^*H=M=mL$, and
    \item there exists a very ample$/U^0$ divisor $H_{Z^0}$ on $Z^0$ such that $(g^0)^*H_{Z^0}=mL_{W^0}$.
\end{itemize}
By construction, all conditions of Constructions \ref{cons: gluing part 1} and \ref{cons: glue part 2} hold. Therefore, in the following, we will adopt all notations as in Constructions \ref{cons: gluing part 1} and \ref{cons: glue part 2}. By Lemma \ref{lem: induced relation is stratified}, $(\sigma_1,\sigma_2):T^n\rightrightarrows Y^n$ and $(\sigma_{1H},\sigma_{2H}):T^n_H\rightrightarrows Y^n_H$ are stratified equivalence relations, and  $(Y^n,S_*(\Aa_W^n))$, $(T^n,S_*(\Aa_D^n))$, $(Y^n_H, S_*(\Aa_{W_M}^n))$, and $(T^n_H, S_*(\Aa_{D_M}^n))$ satisfy (HN) and (HSN).

We let  $p_{Z^0}: Z^0_{H_{Z_0}}\rightarrow Z^0$ be the total spaces of the line bundle $H_{Z^0}$.

We let $Y^{n,0}=Y^n\times_UU^0$, $T^{n,0}=T^n\times_UU^0$, $Y^{n,0}_H=Y^n_H\times_UU^0$, and $T^{n,0}_H=T^n_H\times_UU^0$. Then the geometric quotients $Z^0=Y^{n,0}/T^{n,0}$ and $Z^0_{H_{Z^0}}=Y^{n,0}_H/T^{n,0}_{H}$ exist by \cite[Lemma 9.8]{Kol13}. In particular, the equivalence relations generated by $(\sigma_1,\sigma_2)|_{T^{n,0}}: T^{n,0}\rightrightarrows Y^{n,0}$ and $(\sigma_{1H},\sigma_{2H})|_{T^{n,0}_H}: T^{n,0}_H\rightrightarrows Y^{n,0}_H$ are finite. By \cite[Lemma 9.55]{Kol13}, the equivalence relations generated by $(\sigma_1,\sigma_2):T^n\rightrightarrows Y^n$ and $(\sigma_{1H},\sigma_{2H}):T^n_H\rightrightarrows Y^n_H$ are finite (cf. \cite[Proposition 3.12]{HX13}). By \cite[Theorem 9.21]{Kol13}, the geometric quotients $Y^n/T^n$ and $Y^n_H/T^n_H$ exist. 

We denote $Z:=Y^n/T^n$ and $Z_{H_Z}:=Y^n_H/T^n_H$. Then we have induced morphisms $p_Z: Z_{H_Z}\rightarrow Z$, $g: W\rightarrow Z$, and $\pi_Z: Y^n\rightarrow Z$, such that
\begin{itemize}
\item  $p_Z: Z_{H_Z}\rightarrow Z$ is a total space of a line bundle $H_Z$ on $Z$,
\item $Z^0=Z\times_UU^0$ and $Z^0_{H_{Z^0}}=Z_{H_Z}\times_UU^0$,
\item $g^0=g|_{W^0}$ and $g^*H_Z=mL_W$, and
\item $\pi_Z^*H_Z=H$.
\end{itemize}
Since $H$ is ample$/U$, $H_Z$ is ample$/U$. Thus $L_W=K_{\Aa}|_W$ is semi-ample$/U$. 
\end{proof}

\subsection{Gluing to high dimensions}

\begin{lem}\label{lem: an abudant lemma}
Let $X\rightarrow U$ be a projective morphism between normal quasi-projective varieties, $f: X\rightarrow Z$ a contraction$/U$, $D$ a vertical$/Z$ $\mathbb R$-Cartier $\mathbb R$-divisor on $X$, and $A$ an ample$/U$ $\mathbb R$-divisor on $Z$. Then $f^*A-tD$ is abundant$/U$ for any $0<t\ll 1$.
\end{lem}
\begin{proof}
    Let $V:=f(\Supp D)$. Then we have $A\sim_{\mathbb R,U}A'+L$ such that $A'$ is ample$/U$, $L\geq 0$ and $V\subset\Supp L$. Therefore, for any $0<t\ll 1$,
    $$f^*A-tD\sim_{\mathbb R,U}f^*A'+(f^*L-tD).$$
    and $f^*L-tD\geq 0$. By Lemma \ref{lem: property of numerical and Iitaka dimension}(1),
    $$\dim Z=\kappa_{\iota}(X/U,f^*A)\geq\kappa_{\iota}(X/U,f^*A-tD)\geq \kappa_{\iota}(X/U,f^*A')=\dim Z$$
    and
    $$\dim Z=\kappa_{\sigma}(X/U,f^*A)\geq\kappa_{\sigma}(X/U,f^*A-tD)\geq \kappa_{\sigma}(X/U,f^*A')=\dim Z$$
    and the lemma follows.
\end{proof}

The goal of this subsection is to prove the following theorem.

\begin{thm}\label{thm: bir12 1.7 g LD}
Let $\Aa/U$ be an LD generalized pair and $U^0\subset U$ a non-empty open subset. Let $\Aa^0:=\Aa\times_UU^0$. Assume that:
\begin{enumerate}
    \item $K_{\Aa}$ is nef$/U$.
    \item $K_{\Aa^0}$ is semi-ample$/U^0$.
    \item The image of any lc center of $\Aa$ in $U$ intersects $U^0$.
\end{enumerate}
Then $K_{\Aa}$ is semi-ample$/U$.	
\end{thm}

To prove Theorem \ref{thm: bir12 1.7 g LD}, we prove it inductively together with the following theorem. It is clear that Theorem \ref{thm: bir12 1.7 g} is a weaker version of Theorem \ref{thm: bir12 1.7 g LD}.

\begin{thm}\label{thm: bir12 1.7 g}
Let $\Aa/U=(X,B,\Mm)/U$ be a $\mathbb Q$-factorial dlt generalized pair, $U^0\subset U$ a non-empty open subset, and $\Aa^0:=\Aa\times_UU^0$. Assume that:
\begin{enumerate}
    \item $K_{\Aa}$ is a nef$/U$ $\mathbb Q$-divisor.
    \item $K_{\Aa^0}$ is semi-ample$/U^0$.
    \item The image of any lc center of $\Aa$ in $U$ intersects $U^0$. 
    \item $K_{\Aa}|_S$ is semi-ample$/U$ for any irreducible component $S$ of $\lfloor B\rfloor$.
    \item Either $K_{\Aa}$ is big$/U$, or there exists an $\mathbb R$-divisor $P\geq 0$ on $X$ such that $\Supp P\subset\Supp\lfloor B\rfloor$ and $K_{\Aa}-tP$ is semi-ample$/U$ for any $0<t\ll 1$.
\end{enumerate}
Then $K_X+B+\Mm_X$ is semi-ample$/U$.	
\end{thm}

\begin{prop}\label{prop: inductive approach to bir12 1.7}
Let $d\geq 2$ be an integer. Then Theorem \ref{thm: bir12 1.7 g} in dimension $d$ and Theorem \ref{thm: bir12 1.7 g LD} in dimension $\leq d-1$ imply Theorem \ref{thm: bir12 1.7 g LD} in dimension $d$.
\end{prop}
\begin{proof}
By Lemmas \ref{lem: LD preserves lc center}, \ref{lem: semi-ampleness ld}, and \ref{lem: ld nef impleis nqc}, we may assume that $K_{\Aa}$ is a $\mathbb Q$-divisor. By Lemma \ref{lem: LD under dlt model}, possibly replacing $\Aa$ with a $\mathbb Q$-factorial dlt modification, we may assume that $\Aa$ is $\mathbb Q$-factorial dlt. Write $\Aa=(X,B,\Mm)$. There are two cases.

\medskip

\noindent\textbf{Case 1.} There exists an irreducible component $S$ of $\lfloor B\rfloor$ such that $K_{\Aa}-\epsilon S$ is not pseudo-effective$/U$ for any $\epsilon>0$. By Lemma \ref{lem: trivial mmp} and \cite[Theorem 2.1.5]{Cas+25a}, we may pick $0<\epsilon\ll 1$ and run a $\left(K_{\Aa}-\epsilon S\right)$-MMP$/U$ with scaling of an ample divisor which terminates with a Mori fiber space$/U$ $f: X'\rightarrow Z$, such that the induced birational map$/U$ $\phi: X\dashrightarrow X'$ is $K_{\Aa}$-trivial. Let $\Aa':=\phi_*\Aa$, $\Aa'^0:=\Aa'\times_UU^0$, $\widetilde{S}:=\phi_*S$, $S'$ the normalization of $\widetilde{S}$, and $S'_0:=S'\times_UU^0$. Let $\Aa_{S'}:=\Aa'|_{S'}$ and $\Aa_{S'}^0:=\Aa_{S'}\times_UU^0$. 

By Lemma \ref{lem: LD preserve under mmp}, $\Aa'$ is LD. By Lemma \ref{lem: ld adjunction}, $\Aa_{S'}$ is LD. Since $\phi$ is $K_{\Aa}$-trivial, $K_{\Aa'}$ is nef$/U$ and $K_{\Aa'^0}$ is semi-ample$/U^0$. Thus $K_{\Aa_{S'}}$ is nef and $K_{\Aa_{S'}^0}$ is semi-ample$/U^0$. By Theorem \ref{thm: bir12 1.7 g LD} in dimension $d-1$, $K_{\Aa_{S'}}$ is semi-ample$/U$. 

Since $f$ is $K_{\Aa'}$-trivial and is a step of a $\left(K_{\Aa'}-\epsilon\widetilde{S}\right)$-MMP, $K_{\Aa'}\sim_{\mathbb R,Z}0$ and $\widetilde{S}$ is horizontal$/Z$. Thus $K_{\Aa_{S'}}\sim_{\mathbb R,Z}0$. Let $f_S: S'\rightarrow Z$ be the associated projective morphism. Then there exists an $\mathbb R$-divisor $L$ on $Z$ such that $$K_{\Aa_{S'}}\sim_{\mathbb R}f_S^*L,$$
hence $K_{\Aa'}\sim_{\mathbb R}f^*L$. Since $K_{\Aa_{S'}}$ is semi-ample$/U$, $L$ is semi-ample$/U$, so $K_{\Aa'}$ is semi-ample$/U$, and so $K_{\Aa}$ is semi-ample$/U$.

\medskip

\noindent\textbf{Case 2.} $K_{\Aa}-t\lfloor B\rfloor$ is pseudo-effective$/U$ for some $t>0$. Let $B_0:=B|_{X^0}$. We let
$$\psi_0: X^0\dashrightarrow Z^0$$
be the ample model$/U^0$ of $K_{\Aa^0}$. Since $K_{\Aa^0}-t\lfloor B_0\rfloor$ is pseudo-effective$/U^0$, any irreducible component of $\lfloor B_0\rfloor$ is vertical$/Z^0$. For any $0<t\ll 1$, by Lemma \ref{lem: an abudant lemma}, $K_{\Aa^0}-t\lfloor B_0\rfloor$ is abundant$/U^0$, hence $K_{\Aa}-t\lfloor B\rfloor$ is abundant$/U$.

Let $\Aa(t):=(\Aa,-t\lfloor B\rfloor)$ for any $t\in\mathbb R$. For any $0<t\ll 1$, since $\Aa(t)$ is klt and $K_{\Aa(t)}$ is abundant$/U$, by Lemma \ref{lem: canoincal model implies glmm}, $\Aa(t)/U$ has a $\mathbb Q$-factorial good minimal model. By Theorem \ref{thm: bswlc imply mm} and Lemma \ref{lem: trivial mmp}, we may fix $0<t_1\ll 1$ and run a $K_{\Aa(t_1)}$-MMP$/U$ with scaling of an ample divisor $\phi: X\dashrightarrow X'$ which terminates with a good minimal model $\Aa'(t_1)/U$ of $\Aa(t_1)/U$. Let $X'$ be the ambient variety of $\Aa'(t_1)$ and $\phi: X\dashrightarrow X'$ the associated birational map. Let $\Aa'(t):=\phi_*\Aa(t)$ for any $t\in\mathbb R$. Then for any $t\in (0,t_1]$, $\phi$ is also a sequence of steps of a $K_{\Aa(t)}$-MMP$/U$ and $K_{\Aa'(t)}$ is nef$/U$, so $\Aa'(t)/U$ is a minimal model of $\Aa(t)/U$. By Lemma \ref{lem: Bir12 2.7}, $\Aa'(t)/U$ is a good minimal model of $\Aa(t)/U$ for any $t\in (0,t_1]$.

Let $h: Y\rightarrow X'$ be a $\mathbb Q$-factorial dlt modification of $\Aa'$, $\Aa_Y:=h^*\Aa'$, $\Aa'^0:=\Aa'\times_UU^0$, and $\Aa_Y^0:=\Aa_Y\times_UU^0$. Since $\phi$ is $K_{\Aa}$-trivial and $K_{\Aa}$ is a nef$/U$ $\mathbb Q$-divisor, $K_{\Aa'}$ is nef$/U$ $\mathbb Q$-divisor and $K_{\Aa'^0}$ is semi-ample$/U^0$, hence
\begin{itemize}
    \item $K_{\Aa_Y}$ is a nef$/U$ $\mathbb Q$-divisor, and
    \item $K_{\Aa_Y^0}$ is semi-ample$/U^0$. 
\end{itemize}
Moreover, since any lc place of $\Aa_Y$ is an lc place of $\Aa$,
\begin{itemize}
    \item the image of any lc center of $\Aa_Y$ in $U$ intersects $U^0$.
\end{itemize}
Write $\Aa_Y=(Y,B_Y,\Mm)$. For any irreducible component $S_Y$ of $\lfloor B_Y\rfloor$, $K_{\Aa_Y|_{S_Y}}=K_{\Aa_Y}|_{S_Y}$ is a nef$/U$ $\mathbb Q$-divisor and so $\Aa|_{S_Y}$ is LD, $K_{\Aa|_{S_Y}\times_UU^0}=K_{\Aa_Y^0}|_{S_Y\times_UU^0}$ is semi-ample$/U^0$, and since any lc center of $\Aa|_{S_Y}$ is an lc center of $\Aa_Y$, the image of any lc center of $\Aa|_{S_Y}$ in $U$ intersects $U^0$. By Theorem \ref{thm: bir12 1.7 g LD} in dimension $d-1$,
\begin{itemize}
    \item $K_{\Aa_Y|_{S_Y}}$ is semi-ample$/U$ for any irreducible component $S_Y$ of $\lfloor B_Y\rfloor$.
\end{itemize}
Let $\Aa_Y(t):=h^*\Aa(t)$ for any $t\in\mathbb R$ and let $P:=h^*\lfloor B'\rfloor$, where $B'$ is the image of $B$ on $X'$. Since $\Aa(t)$ is klt for any $0<t\ll 1$, we have $\Supp P=\Supp \lfloor B_Y\rfloor$. Thus 
\begin{itemize}
    \item $K_{\Aa_Y(t)}-tP$ is semi-ample$/U$ for any $0<t\ll 1$, where $\Supp P=\Supp\lfloor B_Y\rfloor$.
\end{itemize}
By Theorem \ref{thm: bir12 1.7 g} in dimension $d$, $K_{\Aa_Y}$ is semi-ample$/U$, so $K_{\Aa'}$ is semi-ample$/U$, and so $K_{\Aa}$ is semi-ample$/U$. The theorem follows.
\end{proof}

\begin{prop}\label{prop: inductive approach to bir12 1.7 weak version}
Let $d\geq 2$ be an integer. Then Theorem \ref{thm: bir12 1.7 g LD} in dimension $\leq d-1$ implies Theorem \ref{thm: bir12 1.7 g} in dimension $d$.
\end{prop}
\begin{proof}
When $K_{\Aa}$ is big$/U$, by (2-4) and Theorem \ref{thm: semi-ample over U0 implies semi-ample over U}, $K_{\Aa}|_{\lfloor B\rfloor}$ is semi-ample$/U$. By Lemma \ref{lem: reduction to Nlc locus}, $K_{\Aa}$ is semi-ample$/U$. Thus we may assume that there exists an $\mathbb R$-divisor $P\geq 0$ on $X$ such that $\Supp P\subset\Supp\lfloor B\rfloor$ and $K_{\Aa}-tP$ is semi-ample$/U$ for any $0<t\ll 1$, and $K_{\Aa}$ is not big$/U$.

We denote by $\Aa(t):=(\Aa,-tP)$ for any real number $t$. Fix $0<t_0\ll 1$, pick $t_1\in (0,t_0)$ general, and let $g: X\rightarrow Y$ be the ample model$/U$ of $K_{\Aa(t_1)}$. Since $t_1$ is general, we have $P\sim_{\mathbb R,Y}0$ and $K_{\Aa}\sim_{\mathbb R,Y}0$. Since $K_{\Aa}$ is not big$/U$, $K_{\Aa(t_1)}$ is not big$/U$, so $\dim Y<\dim X$. If $\dim Y=0$, then $K_{\Aa}\sim_{\mathbb R}0$ and we are done. Thus we may assume that $\dim Y>0$.

Since $K_{\Aa}$ is a $\mathbb Q$-divisor, by \cite[Lemma 11.4.2, Theorem 11.4.4]{CHLX23}, there exists an lc generalized pair $\Aa_Y/U$ induced by the canonical model formula of $g: \Aa\rightarrow Y$, i.e.
$$K_{\Aa}\sim_{\mathbb Q}g^*K_{\Aa_Y}.$$
Let $h: Z\rightarrow Y$ be a $\mathbb Q$-factorial dlt modification of $\Aa_Y$, $\Aa_Z:=h^*\Aa_Y$, and $\Aa_Z^0:=\Aa_Z\times_UU^0$. Since $K_{\Aa}$ is a nef$/U$ $\mathbb Q$-divisor, $K_{\Aa_Y}$ is a nef$/U$ $\mathbb Q$-divisor. Thus 
\begin{itemize}
    \item $K_{\Aa_Z}$ is a nef$/U$ $\mathbb Q$-divisor.
\end{itemize}
Since $K_{\Aa^0}$ is semi-ample$/U^0$, $K_{\Aa_Y^0}$ is semi-ample$/U^0$, where $\Aa_Y^0:=\Aa_Y\times_UU^0$, hence
\begin{itemize}
\item $K_{\Aa_Z^0}$ is semi-ample$/U^0$.
\end{itemize}
For any lc center $V$ of $\Aa_Z$, $h(V)$ is an lc center of $\Aa_Y$, hence the image of an lc center $W$ of $\Aa$ in $Y$ by \cite[Theorem 11.4.4(4)]{CHLX23}. Thus the image of $V$ in $U$ is the image of $W$ in $U$, which intersects $U^0$. Therefore, 
\begin{itemize}
    \item the image of any lc center of $\Aa_Z$ in $U$ intersects $U^0$.
\end{itemize}
By Theorem \ref{thm: bir12 1.7 g LD} in dimension $\leq d-1$, $K_{\Aa_Z}$ is semi-ample$/U$. Thus $K_{\Aa_Y}$ is semi-ample$/U$, hence $K_{\Aa}$ is semi-ample$/U$.
\end{proof}

\begin{proof}[Proof of Theorems \ref{thm: bir12 1.7 g LD} and \ref{thm: bir12 1.7 g}]
It is obvious that Theorem \ref{thm: bir12 1.7 g LD} holds in dimension $1$. The theorems now follow from Propositions \ref{prop: inductive approach to bir12 1.7} and \ref{prop: inductive approach to bir12 1.7 weak version}.
\end{proof}

\section{Proof of Theorem \ref{thm: hx13 1.1 nonnqc-g}}\label{sec: hx13 1.1}

The goal of this section is twofold: to prove Theorem \ref{thm: hx13 1.1 nonnqc-g}, and to prove Theorem \ref{thm: bir12 1.1 nonnqc-g} in the case when $(X,B,\Mm)$ is LD. Indeed, we shall prove the following theorem which can be seen as a more general version of the LD case of Theorem \ref{thm: bir12 1.1 nonnqc-g} (which corresponds to the case when $U^0=U$), and can be seen as an analogue of \cite[Theorem 4.1]{Has19}.

\begin{thm}\label{thm: bir12 1.1 3non}
Let $\Aa/U$ be an LD generalized pair with ambient variety $X$, $U^0\subset U$ a non-empty open subset, $A\geq 0$ an $\mathbb R$-Cartier $\mathbb R$-divisor on $X$, $\Aa^0:=\Aa\times_UU^0$, and $A^0:=A|_{X^0}$. Assume that
\begin{enumerate}
\item $(\Aa^0,A^0)$ is lc,
\item $K_{\Aa^0}+A^0\sim_{\mathbb R,U^0}0$, 
\item $K_{\Aa}$ is pseudo-effective$/U$, and
\item the image of any lc center of $\Aa$ in $U$ intersects $U^0$.
\end{enumerate}
Then $\Aa/U$ has a good log minimal model.
\end{thm}

The proofs of Theorems \ref{thm: hx13 1.1 nonnqc-g} and \ref{thm: bir12 1.1 3non} use ideas from \cite{Has19}. The proof is divided into two parts: first, we need to construct $\mathbb Q$-factorial dlt modifications satisfying extra properties, which allow us to essentially reduce the question to a special termination problem. The second part is to construct a model that is a good minimal model for several generalized pairs, run a special MMP with scaling, and then use special termination to deduce the result.

\subsection{Construction of a special dlt modification}

In this subsection, we construct $\mathbb Q$-factorial dlt modifications satisfying some extra properties. Ideas of constructing such $\mathbb Q$-factorial dlt modifications originated in \cite{Has18}.

\begin{lem}\label{lem: special proper log smooth model}
Let $\Aa/U$ be an lc generalized pair. Then there exists a proper log smooth model $h: \Aa'\rightarrow\Aa$ with $\Aa'=(X',B',\Mm)$, such that $B'=B^h+B^v$ satisfies the following.
\begin{enumerate}
\item $B^h\geq 0$ and $B^v$ is reduced.
\item $B^v$ is vertical$/U$.
\item For any real number $t>0$, any lc center of $(\Aa',-tB^v)$ dominates $U$.
\end{enumerate}
\end{lem}
\begin{proof}
Let $g: \Aa''\rightarrow\Aa$ be a proper log smooth model and write $\Aa'':=(X'',B'',\Mm)$. By \cite[Lemma 2.10]{Has18}, there exists a proper log smooth model $f: (X',B'=B^h+B^v)\rightarrow (X'',B'')$ such that 
\begin{itemize}
    \item $B^h\geq 0$ and  $B^v$ is reduced,
    \item $B^v$ is vertical$/U$, and
    \item For any real number $t>0$, any lc center of $(X',B'-tB^v)$ dominates $U$.
\end{itemize}
Let $h:=g\circ f$ and let $\Aa':=(X',B',\Mm)$. Then $h: \Aa'\rightarrow\Aa$ is a proper log smooth model of $\Aa$ by definition. Since $\Mm$ descends to $X'$, for any real number $t>0$, any lc center of $(\Aa,-tB^v)$ is an lc center of $(X',B'-tB^v)$, hence dominates $U$. Thus $h: \Aa'\rightarrow\Aa$ satisfies our requirements.
\end{proof}

\begin{lem}\label{lem: special gdlt modification gpair}
Let $\Aa/U$ be an lc generalized pair with ambient variety $X$. Then there exists a $\mathbb Q$-factorial dlt modification $h: X'\rightarrow X$ of $\Aa$ with $\Aa':=h^*\Aa:=(X',B',\Mm)$, such that $B'=B^h+B^v$ satisfies the following.
\begin{enumerate}
\item $B^h\geq 0$ and $B^v$ is reduced.
\item $B^v$ is vertical$/U$.
\item For any real number $t>0$, any lc center of $(\Aa',-tB^v)$ dominates $U$.
\end{enumerate}
\end{lem}
\begin{proof}
By Lemma \ref{lem: special proper log smooth model}, there exists a proper log smooth model $g: \Aa_W\rightarrow\Aa$ with $\Aa_W=(W,B_W,\Mm)$, such that $B_W=B_W^h+B_W^v$ satisfies the following:
\begin{itemize}
    \item $B_W^h\geq 0$ and $B_W^v$ is reduced.
    \item $B_W^v$ is vertical$/U$.
    \item For any real number $t\in (0,1]$, any lc center of $(\Aa_W,-tB_W^v)$ dominates $U$.
\end{itemize}
By Lemma \ref{lem: foliation lsm has lmm}, we may run a $K_{\Aa_W}$-MMP$/X$ with scaling of an ample divisor which terminates with a good minimal model $\Aa'/X$ of $\Aa_W/X$ such that $K_{\Aa'}\sim_{\mathbb R,X}0$ with induced birational morphism $h: X'\rightarrow X$, such that $h$ is a $\mathbb Q$-factorial qdlt modification of $\Aa$ and $\Aa'=h^*\Aa$. Since $\Aa_W$ is dlt, $\Aa'$ is dlt, so $h$ is a $\mathbb Q$-factorial dlt modification of $\Aa$. We let $B',B^h$ and $B^v$ be the images of $B_W,B_W^h$ and $B_W^v$ on $X'$ respectively.

(1) and (2) immediately hold. For any $0<s\ll 1$, the induced birational map $\phi: W\dashrightarrow X'$ is also a sequence of steps of a $(K_{\Aa_W}-sB_W^v)$-MMP$/X$. Since $\Aa'$ is lc, for any $t>0$, any lc place of $(\Aa',-tB^v)$ is an lc place of $(\Aa',-stB^v)$ for $0<s\ll 1$, hence an lc place of $(\Aa_W,-stB_W^v)$, hence dominates $U$. This implies (3), so $h$ and $\Aa'$ satisfies our requirements.
\end{proof}

\begin{lem}\label{lem: special dlt model ii}
Let $\Aa/U=(X,B=B^h+B^v,\Mm)/U$ be an $\mathbb Q$-factorial dlt generalized pair and $\pi: X\rightarrow Z$ a contraction$/U$. Assume that:
\begin{itemize}
\item $B^h\geq 0$ and $B^v$ is reduced.
\item $B^v$ is vertical$/Z$.
\item For any real number $t>0$, any lc center of $(\Aa,-tB^v)$ dominates $Z$.
\item $K_{\Aa}\sim_{\mathbb R,Z}E\geq 0$.
\item $E$ is vertical$/Z$ and $\Supp B^v\subset\Supp E$.
\end{itemize}
Then there exists a $\mathbb Q$-factorial dlt modification $h: Y\rightarrow X$ of $\Aa$ satisfying the following.
\begin{enumerate}
    \item $\Aa_Y:=h^*\Aa:=(Y,B_Y,\Mm)$.
    \item $B_Y=B_Y^h+B_Y^v$ and $K_{\Aa_Y}\sim_{\mathbb R,Z}E_Y^h+E_Y^v$.
    \item $0\leq E_Y^h,E_Y^v$ are vertical$/Z$, $\Supp E_Y^v=\Supp B_Y^v$, and $B_Y^v$ is reduced.
    \item For any $0<t\ll 1$, $(\Aa_Y,tE_Y^h)$ is dlt.
    \item For any $0<t\ll 1$, any lc center of $(\Aa_Y,-tB_Y^v)$ dominate $Z$.
\end{enumerate}
\end{lem}
\begin{proof}
We let $g: \Aa_W\rightarrow\Aa$ be a proper log toroidal model of $\Aa$ and denote by $\Aa':=(W,B_W,\Mm)$. Let 
$$B_W^v:=\left\lfloor B_W\wedge\Supp\left(g^*B^v\right)\right\rfloor\quad \text{and}\quad B_W^h:=B_W-B_W^v.$$
Then $B_W^v$ is reduced, and we may write $g^*E=E_W^h+E_W^v$ such that $E_W^h,E_W^v\geq 0$, $\Supp E_W^v=\Supp B_W^v$, and $E_W^v\wedge E_W^h=0$. Since $E$ is vertical$/Z$, $E_W^h$ and $E_W^v$ are vertical$/Z$.

For any irreducible component $D$ of $g^*E$ that is also an irreducible component of $\lfloor B_W\rfloor$, we have
$$g(D)\subset\Supp g(E)\subset \Supp g(g^*E))=\Supp E,$$
hence $D$ is vertical$/Z$. Since $D$ is an irreducible component of $\lfloor B_W\rfloor$, $D$ is an lc place of $\Aa_W$, so $D$ is an lc place of $\Aa$. Since  any lc center of $(\Aa,-tB^v)$ dominate $Z$ for any $t>0$, $g(D)\subset\Supp B^v$, so $D\subset\Supp\left(g^*B^v\right)$. Thus $D$ is an irreducible component of $\Supp B_W^v$, hence $D$ is not an irreducible component of $\Supp B_W^h=\Supp E_W^h$. Thus $D$ is not an irreducible component of $E_W^h$. Therefore, $E_W^h\wedge \lfloor B_W\rfloor=0$. 

Since the coefficients of $B_W$ are $\leq 1$ and the $B_W^v\subset\lfloor B_W\rfloor$, $B_W^v\wedge B_W^h=0$, hence $E_W^v\wedge E_W^h=0$. By Lemma \ref{lem: foliation lsm has lmm}, we may run a $K_{\Aa_W}$-MMP$/X$ with scaling of an ample divisor which terminates with a good minimal model $\Aa_Y/X$ of $\Aa_W/X$ such that $K_{\Aa_Y}\sim_{\mathbb R,X}0$ with induced birational morphism $h: Y\rightarrow X$, such that $h$ is a $\mathbb Q$-factorial qdlt modification of $\Aa$ and $\Aa_Y=h^*\Aa$.  Since $\Aa_W$ is dlt, $\Aa_Y$ is dlt, so $h$ is a $\mathbb Q$-factorial dlt modification of $\Aa$. 

Since $E^h\wedge\lfloor B_W\rfloor=0$, for any $0<t\ll 1$, $(\Aa_W,tE_W^h)$ is dlt and the induced birational map $\phi: W\dashrightarrow Y$ is sequence of steps of a $(K_{\Aa_W}+tE_W^h)$-MMP$/X$ and a sequence of steps of a $(K_{\Aa_W}-tB_W^v)$-MMP$/X$. Let $B_Y^h,B_Y^v,E_Y^h,E_Y^v$ be the images of $B_W^h,B_W^v,E_W^h,E_W^v$ on $Y$. Then (1-4) hold, and (5) holds as any lc place of $(\Aa_Y,-tB_Y^v)$ is an lc place of $(\Aa_W,-tB_W^v)$.
\end{proof}

\subsection{Proof of the special case}

In this subsection, we prove a special case of Theorems \ref{thm: hx13 1.1 nonnqc-g} and \ref{thm: bir12 1.1 3non}, assuming these two theorems in lower dimensions. The general cases of Theorems \ref{thm: hx13 1.1 nonnqc-g} and \ref{thm: bir12 1.1 3non} will be reduced to this special case. The proof of the special case is very similar to \cite[Proof of Theorems 1.2 after Step 6, Proof of Theorem 4.1]{Has19} and some ideas were adopted later in \cite{Cas+25a,Cas+25b}.

\begin{prop}\label{prop: main theorems for special dlt model}
Assume that Theorems \ref{thm: hx13 1.1 nonnqc-g} and \ref{thm: bir12 1.1 3non} hold in dimension $\leq d-1$.

Let $\Aa/U:=(X,B,\Mm)/U$ be an LD generalized pair of dimension $d$ and $U^0\subset U$ a non-empty open subset. Let $\Aa^0:=\Aa\times_UU^0$. Assume that
    \begin{enumerate}
        \item $\Aa$ is $\mathbb Q$-factorial dlt,
        \item $B=B^h+B^v$ and $K_{\Aa}\sim_{\mathbb R,U}E^h+E^v$.
        \item $E^h,E^v\geq 0$, $\Supp E^v=\Supp B^v$, and $B^v$ is reduced.
        \item For any $0<t\ll 1$, $(\Aa,tE^h)$ is dlt.
        \item For any $0<t\ll 1$, $(\Aa,-tE^v)/U$ has a good minimal model.
        \item For any lc place $D$ of $\Aa$, the image of $D$ on $U$ intersects $U^0$.
        \item Let $X^0:=X\times_UU^0$ and let $A\geq 0$ is an $\mathbb R$-Cartier $\mathbb R$-divisor on $X$ with $A^0:=A|_{X^0}$, such that one of the following two cases hold:
        \begin{enumerate}
            \item[\textbf{\rm\textbf{(Case 1)}}] $\Aa^0/U^0$ is a good minimal model of itself.
            \item[\textbf{\rm\textbf{(Case 2)}}] $(\Aa^0,A^0)$ is lc and $K_{\Aa^0}+A^0\sim_{\mathbb R,U^0}0$.
        \end{enumerate}
    \end{enumerate}
    Then $\Aa/U$ has a good minimal model.
\end{prop}
\begin{proof}
We say a few words about the key ideas of the proof before we start proving the proposition. The key idea is to use the fact that $(\Aa,-tE^v)/U$ is a good minimal model for any $0<t\ll 1$: If we can find $t_1<t_2$ such that $(\Aa,-t_1E^v)/U$ and $(\Aa,-t_2E^v)/U$ are good minimal model of themselves, then we may run a $K_{\Aa}$-MMP$/U$ with scaling of $-E^v$, and the MMP will terminate by special termination and induction hypothesis. For the general case we need to run MMP to obtain models so that the birational transformations of $(\Aa,-t_1E^v)/U$ and $(\Aa,-t_2E^v)/U$ are good minimal models of themselves and apply similar but more delicate arguments.

\medskip

\noindent\textbf{Step 1.} In this step we introduce some basic notations and fix real number $r>0$. We let $\Aa(t):=(\Aa,-tE^v)$, and $\Cc(t):=(\Aa,tE^h)$ for any $t\in\mathbb R$. Then we have $K_{\Aa(0)}=K_{\Cc(0)}=K_{\Aa}$ and
$$K_{\Aa(t)}\sim_{\mathbb R,U}(1-t)K_{\Cc\left(\frac{t}{1-t}\right)}$$
for any $t\in [0,1)$. By condition (2), $K_{\Aa(t)}$ is pseudo-effective$/U$ for any $0\leq t\ll 1$. We let $N_t:=N_{\sigma}(X/U,K_{\Aa(t)})$ for any $t\in\mathbb R$. By condition (7) and \cite[Lemma 3.23]{Cas+25a}, for any $0<t\ll 1$, $\Supp N_t$ is a fixed divisor $F$ and $\Supp N_0\subset F$.

Since $\Aa(t)$ and $\Cc(t)$ are dlt for any $0\leq t\ll 1$, if we are in \textbf{Case 1}, then by Lemma \ref{lem: trivial mmp}, there exists a real number $r$ such that $0<r\ll 1$ and satisfies the following: for any $0<t\leq r$ and any sequence of steps of a $K_{\Aa(t)}$-MMP$/U$ or $K_{\Cc(t)}$-MMP$/U$ $\xi$, $\xi|_{X^0}$ is $K_{\Aa^0}$-trivial. If we are in \textbf{Case 2} then we set $r:=1$.

\medskip

\noindent\textbf{Step 2.} In this step we fix a real number $0<t\ll r$ and construct a birational model $X'$ of $X$ associated with birational map $\phi: X\dashrightarrow X'$ so that $\phi_*\Aa(t)$ is semi-ample$/U$.

Fix $0<t\ll r$. Since $\Aa(t)/U$ has a good minimal model and is $\mathbb Q$-factorial dlt, by Theorem \ref{thm: bswlc imply mm}, we may run a $K_{\Aa(t)}$-MMP$/U$ with scaling of an ample divisor which terminates with a good minimal model $\Aa'(t)/U$ of $\Aa(t)/U$. Let $X'$ be the ambient variety of $\Aa'(t)$, $\phi: X\dashrightarrow X'$ the induced birational map, $\Aa':=\phi_*\Aa,\Aa'(s):=\phi_*\Aa(s)$, and $\Cc'(s):=\phi_*\Cc(s)$ for any $s\in\mathbb R$. Then the reduced divisor contracted by $\phi$ is exactly $F$. By \cite[Lemma 3.21]{Cas+25a}, $K_{\Aa'(s)}$ is movable$/U$ for any $0\leq s\ll 1$. Since $\phi$ is also a sequence of steps of a $K_{\Cc\left(\frac{t}{1-t}\right)}$-MMP$/U$, $\Cc'\left(\frac{t}{1-t}\right)$ is $\mathbb Q$-factorial dlt. Thus $\Cc'(0)=\Aa'(0)=\Aa'$ is $\mathbb Q$-factorial dlt.  Since $0<t\ll r$, any lc place of $\Cc(\frac{t}{1-t})$ is an lc place of $\Aa$, and if we are in \textbf{Case 1}, then by the length of extremal rays, $\phi|_{X^0}$ is $K_{\Aa^0}$-trivial.

\medskip

\noindent\textbf{Step 3.} In this step we fix a real number $0<t-\mu\ll \mu$ and construct a birational model $X''$ of $X$ associated with birational map $\psi: X'\dashrightarrow X''$ so that $\psi_*\Aa'(\mu)$ is semi-ample$/U$.

Fix $\mu<t$ such that $0<t-\mu\ll\mu$. Then $\phi$ is also a sequence of steps of a $K_{\Aa(\mu)}$-MMP$/U$. Since $\Aa(\mu)/U$ has a good minimal model, by Lemma \ref{lem: minimal model same after running mmp}, $\Aa'(\mu)/U$ has a good minimal model. Thus we may run a $K_{\Aa'(\mu)}$-MMP$/U$ with scaling of an ample divisor which terminates with a good minimal model $\Aa''(\mu)/U$ of $\Aa'(\mu)/U$. Let $X''$ be the ambient variety of $\Aa''(\mu)$ and let $\psi: X'\dashrightarrow X''$ be the induced birational map. Since
$$\frac{t}{t-\mu}K_{\Aa'(\mu)}=K_{\Aa'(0)}+\frac{\mu}{t-\mu}K_{\Aa'(t)},$$
$\psi$ is also a sequence of steps of a $\left(K_{\Aa'(0)}+\frac{\mu}{t-\mu}K_{\Aa'(t)}\right)$-MMP$/U$. Since $K_{\Aa'(t)}$ is semi-ample$/U$ and $0<t-\mu\ll\mu$, by Lemma \ref{lem: trivial mmp}, $\psi$ is $K_{\Aa'(t)}$-trivial. Let $\Aa'':=\psi_*\Aa',\Aa''(s):=\psi_*\Aa'(s)$, and $\Cc''(s):=\psi_*\Cc'(s)$ for any $s\in\mathbb R$. Since $K_{\Aa'(\mu)}$ is movable$/U$, $\psi$ is small. Since $\Cc'(\frac{t}{1-t})$ is $\mathbb Q$-factorial dlt, $\Cc'(\frac{\mu}{1-\mu})$ is $\mathbb Q$-factorial dlt as $\mu<t$. Since $\psi$ is also a sequence of steps of a $K_{\Cc'(\frac{\mu}{1-\mu})}$-MMP$/U$, $\Cc''(\frac{\mu}{1-\mu})$ is $\mathbb Q$-factorial dlt. In particular, $\Cc''(0)=\Aa''(0)=\Aa''$ is $\mathbb Q$-factorial dlt. Since $\psi$ is $K_{\Cc'(\frac{t}{1-t})}$-trivial, $\Cc''(\frac{t}{1-t})$ is lc. Since $\mu<t\ll r$, by the construction of $r$, if we are in \textbf{Case 1}, then $\psi|_{X'^0}$ is $K_{\Aa'^0}$-trivial, where $X'^0:=X'\times_UU^0$ and $\Aa'^0:=\Aa'\times_UU^0$.

Write $\Aa''=(X'',B'',\Mm)$ and let $B''^h,B''^v,E''^h,E''^v,A''$ be the images of $B^h,B^v,E^h,E^v$ on $X''$ respectively. 

\medskip

\noindent\textbf{Step 4.} In this step we construct a $K_{\Aa''}$-MMP$/U$ with scaling of $K_{\Aa''(t)}$.

For any $s\in (0,t]$, let $\Aa_{Y_s}(s)/U$ be a good minimal model of $\Aa(s)/U$ associated with birational map $\alpha_s: X\dashrightarrow Y_s$. Then the divisors contracted by $\alpha_s$ are exactly $F$. Thus the associated birational map $\beta_s: X''\dashrightarrow Y_s$ is small. Therefore, $\Aa_{Y_s}(s)/U$ is a good minimal model of $\Aa''(s)/U$ for any $s\in (0,t]$. In particular, $\Aa''(s)/U$ has a good minimal model for any $s\in (0,t]$, so $\Cc''(s)/U$ has a good minimal model for any $s\in\left(0,\frac{t}{1-t}\right]$. Since $K_{\Cc''\left(\frac{t}{1-t}\right)}$ and $K_{\Cc''\left(\frac{\mu}{1-\mu}\right)}$ are semi-ample$/U$, by Lemma \ref{lem: construct mmp with scaling to 0}, we may run a $K_{\Cc''(0)}$-MMP$/U$ with scaling of $K_{\Cc''\left(\frac{t}{1-t}\right)}$ with scaling numbers $\lambda_i$
$$\phi_i: X_i\dashrightarrow X_{i+1},\quad  X_1:=X''.$$
such that either this MMP terminates, or $\lim_{i\rightarrow+\infty}\lambda_i=0$. For any $i$, we let $\Aa_i,\Aa_i(s)$, $\Cc_i(s)$ be the images of $\Aa'',\Aa''(s),\Cc''(s)$ on $X_i$ respectively for any $s\in\mathbb R$, let $B_i^h,B_i^v,E_i^h,E_i^v,A_i$ be the images of $B''^h,B''^v,E''^h,E''^v,A''$ on $X_i$ respectively, $X_i^0:=X_i\times_UU^0$, and $A_i^0:=A_i|_{X^0}$.

Since $\Aa$ is LD, by Lemma \ref{lem: ld and irrationality}(1), $\Supp\Mm_{X}^{\irr}\subset\Supp\{B\}$. Since $0<\mu<t\ll 1$, $$\Supp\Mm_{X}^{\irr}\subset\Supp\left\{B+\frac{\mu}{1-\mu}B^h\right\},$$
so 
$$\Supp\Mm_{X''}^{\irr}\subset\Supp\left\{B''+\frac{\mu}{1-\mu}B''^h\right\}\quad \text{and}\quad \Supp\Mm_{X''}^{\irr}\subset\Supp\{B''\}.$$
Since $\Cc''\left(\frac{\mu}{1-\mu}\right)$ is $\mathbb Q$-factorial dlt, $\Cc''(0)=\Aa''$ is $\mathbb Q$-factorial dlt. By Lemma \ref{lem: ld and irrationality}(2), $\Cc''\left(\frac{\mu}{1-\mu}\right)$ and $\Cc''(0)$ are LD. By Theorem \ref{thm: eomm implies tof with scaling}, $\Cc_i(0)=\Aa_i$ is LD for any $i$. 

\medskip

\noindent\textbf{Step 5.} We show that for any $i$ and any non-trivial lc center $S_i$ of $\Aa_i$, either $\left(\Aa|_{S_i}\right)/U$ has a good minimal model or $K_{\Aa|_{S_i}}$ is not pseudo-effective$/U$. The proof uses induction hypothesis.

Pick a non-trivial lc center $S_i$ of $\Aa_i$ for some $i\gg 0$. Then there exists a non-trivial lc center $S$ of $\Aa''$ such that the induced birational map $\psi_i: X''\dashrightarrow X_i$ is an isomorphism near the generic point of $S$. We let $\Aa_{S_i}:=\Aa_i|_{S_i}$. Since $\Aa_i$ is LD, $\Aa_{S_i}$ is LD. Let $\Aa_{S_i}^0:=\Aa_{S_i}\times_UU^0$ and $S_i^0:=S_i\times_UU^0$.

\medskip

\noindent\textbf{Step 5.1.} We first show that the image of any lc center of $\Aa_{S_i}$ in $U$ intersects $U^0$. Let $V$ be an lc center of $\Aa_{S_i}$. We may assume that $V$ is a non-trivial lc center of $\Aa_{S_i}$. Since $\Aa_i$ is dlt, $V$ is an lc center of $\Aa_i$ that is contained in $S_i$. Let $D$ be an lc place of $\Aa''$ such that $\Center_{X_i}D=V$. Since the induced birational map $X''\dashrightarrow X_i$ is $K_{\Aa''}$-negative, $D$ is also an lc place of $\Aa''$. Thus $D$ is also an lc place of $\Cc''\left(\frac{t}{1-t}\right)$. Since $\psi\circ\phi$ is $K_{\Cc\left(\frac{t}{1-t}\right)}$-non-positive, $D$ is also an lc place of $\Cc\left(\frac{t}{1-t}\right)$. Thus $D$ is also an lc place of $\Aa$. Thus the image of $D$ in $U$ intersects $U^0$, hence the image of $V$ in $U$ intersects $U^0$.

\medskip

\noindent\textbf{Step 5.2.} We deal with \textbf{Case 1}. Suppose that we are in \textbf{Case 1}. Since $K_{\Aa^0}$ is semi-ample$/U^0$ and $(\psi\circ\phi)|_{X^0}$ is $K_{\Aa^0}$-trivial, we have that $K_{\Aa''^0}$ is semi-ample$/U^0$, where $\Aa''^0:=\Aa''\times_UU^0$. Therefore, $\psi_i|_{X''^0}$ is $K_{\Aa''^0}$-trivial, where $X''^0:=X''\times_UU^0$, and $K_{\Aa_i^0}$ is semi-ample$/U^0$, where $\Aa_i^0:=\Aa_i\times_UU^0$. Therefore,
$K_{\Aa_{S_i}^0}=K_{\Aa_i^0}|_{S_i^0}$
is semi-ample$/U^0$. By Theorem \ref{thm: hx13 1.1 nonnqc-g} in dimension $\leq d-1$, $\Aa_{S_i}/U$ has a good minimal model.

\medskip

\noindent\textbf{Step 5.3.} We deal with \textbf{Case 2}. Suppose that we are in \textbf{Case 2}. We may assume that $K_{\Aa_{S_i}}$ is pseudo-effective$/U$. Let $\tau: X\dashrightarrow X_i$ be the induced birational map and $\tau^0:=\tau|_{X^0}$. Since $(\Aa^0,A^0)$ is lc and $K_{\Aa^0}+A^0\sim_{\mathbb R,U^0}0$, $\tau^0_*(\Aa^0,A^0)=(\Aa_i^0,A_i^0)$ is lc and $K_{\Aa_i^0}+A_i^0\sim_{\mathbb R,U^0}0$. Let $A_{S_i}:=A_i|_{S_i}$ and $A_{S_i}^0:=A_{S_i}|_{S_i^0}$, then 
$$(\Aa_{S_i}^0,A_{S_i}^0)=(\Aa_i^0,A_i^0)|_{S_i^0}$$
is lc and
$$K_{\Aa_{S_i}^0}+A_{S_i}^0=\left(K_{\Aa_i^0}+A_i^0\right)\Big|_{S_i^0}\sim_{\mathbb R,U^0}0.$$
By Theorem \ref{thm: bir12 1.1 3non} hold in dimension $\leq d-1$, $\Aa_{S_i}/U$ has a good minimal model.

\medskip

\noindent\textbf{Step 6.} We show that $\{\phi_i\}$ terminates with a weak lc model of $\Aa/U$.

Since $\{\phi_i\}$ is also a sequence of steps of a $K_{\Aa''}$-MMP$/U$ with scaling of $K_{\Cc''(\frac{\mu}{1-\mu})}-K_{\Aa''}$, by Theorem \ref{thm: special termination},  $\{\phi_i\}$ terminates near $\lfloor B''\rfloor$. Since $B''^v$ is reduced, $\{\phi_i\}$ terminates $B''^v$. Since $\phi_i$ is also a sequence of steps of a $K_{\Aa''}$-MMP$/U$ with scaling of $-E''^v$ and $\Supp E''^v=\Supp B''^v$, then MMP terminates.

We let
$$\beta: X''\dashrightarrow Y$$
be this MMP and let $\Aa_Y,\Aa_Y(s),\Cc_Y(s)$ be the images of $\Aa'',\Aa''(s),\Cc''(s)$ on $Y$ respectively for any $s\in\mathbb R$. Let $\lambda$ be the last non-zero scaling number of this MMP, then $K_{\Cc_Y(s)}$ is nef for any $0\leq s\leq\lambda$. In particular, $K_{\Aa_Y(s)}$ is nef for any $0\leq s\ll 1$.

Recall (in \textbf{Step 4}) that $\Aa_{Y_s}(s)/U$ is a good minimal model of $\Aa(s)/U$ for any $0<s\ll 1$. Since $K_{\Aa'(\mu)}$ is movable$/U$, $\psi$ is small. Since $K_{\Aa'}$ is movable$/U$, $K_{\Aa''}$ is movable$/U$, so $\beta$ is small. Then the divisors contracted by the induced birational maps $X\dashrightarrow Y$ and $X\dashrightarrow Y_s$ are both $F$, for any $0<s\ll 1$. We let $p_s: W_s\rightarrow Y_s$ and $q_s: W_s\rightarrow Y$ is a resolution of indeterminacy of the induced birational map $Y_s\dashrightarrow Y$. Then we have that $$p_s^*K_{\Aa_{Y_s}(s)}=q_s^*K_{\Aa_{Y}(s)}+G_s$$
for some $G_s$ that is exceptional over $Y_s$ and $Y$. By applying the negativity lemma twice, we have that $G_s=0$. Thus $\Aa_Y(s)/U$ is a bs-semi-ample model of $\Aa(s)/U$ for any $0<s\ll 1$. Let $p: W\rightarrow X$ and $q: W\rightarrow Y$ be a resolution of indeterminacy of $\beta\circ\psi\circ\phi: X\dashrightarrow Y$, then we have
$$p^*K_{\Aa(s)}=q^*K_{\Aa_Y(s)}+L_s$$
for some $L_s\geq 0$ that is exceptional$/Y$. By linearity of discrepancies, we have that
$$p^*K_{\Aa}=q^*K_{\Aa_Y}+L$$
where $L:=\lim_{s\rightarrow 0}L_s\geq 0$. Therefore, $\Aa_Y/U$ is a weak lc model of $\Aa/U$.

\medskip

\noindent\textbf{Step 7.} We conclude the proof in this step.

If we are in \textbf{Case 1}, then by Theorem \ref{thm: bir12 1.7 g LD}, $\Aa_Y/U$ is a semi-ample model of $\Aa/U$. By Theorem \ref{thm: bswlc imply mm}, $\Aa/U$ has a good minimal model.

If we are in \textbf{Case 2}, then we let $A_Y$ be the image of $A$ on $Y$, $\pi_Y: Y\rightarrow U$ the associated projective morphism, $Y^0:=Y\times_UU^0$, $A_Y^0:=A_Y|_{Y^0}$, $\Aa_Y^0:=\Aa_Y\times_YY^0$, $U^1:=U^0\backslash\pi_Y(\Supp A_Y)$, and $Y_1:=Y\times_UU^1$. Since $K_{\Aa_Y}$ is nef$/U$, $K_{\Aa_Y^0}$ is nef$/U^0$. Since $(K_{\Aa_Y}+A_Y)|_{Y^0}\sim_{\mathbb R,U^0}0$, $-A_Y^0$ is nef$/U^0$. Thus $\Supp A=\pi_Y^{-1}\pi_Y(\Supp A)$. Since $(\Aa_Y^0,A_Y^0)$ is lc, $\Supp A_Y^0$ does not contain any lc center of $\Aa_Y^0$. Therefore, the image of any lc center of $\Aa_Y^0$ in $U^0$ intersects $U^1$. Let $\Aa_Y^1:=\Aa_Y^0\times_{U^0}U^1$, then $\Aa_Y^1/U^1$ is a good minimal model of itself as $\Aa_Y^1$ is $\mathbb Q$-factorial qdlt and $K_{\Aa_Y^1}\sim_{\mathbb R,U^1}0$. By \textbf{Case 1}, $\Aa_Y^0/U^0$ has a good minimal model, hence $\Aa_Y^0/U^0$ is a good minimal model of itself. By Theorem \ref{thm: bir12 1.7 g LD}, $\Aa_Y/U$ is a semi-ample model of $\Aa/U$. By Theorem \ref{thm: bswlc imply mm}, $\Aa/U$ has a good minimal model.
\end{proof}

\subsection{Proof of the general case}

In this subsection, we prove Theorems \ref{thm: hx13 1.1 nonnqc-g} and \ref{thm: bir12 1.1 3non}.

\begin{prop}\label{prop: induction to hx13 1.1}
Assume that Theorems \ref{thm: hx13 1.1 nonnqc-g} and \ref{thm: bir12 1.1 3non} hold in dimension $\leq d-1$. Then Theorem \ref{thm: hx13 1.1 nonnqc-g} holds in dimension $d$.
\end{prop}
\begin{proof}
    Notations and conditions as in Theorem \ref{thm: hx13 1.1 nonnqc-g}. Let $\Aa:=(X,B,\Mm)$, $\Aa^0:=\Aa\times_UU^0$, and let $\pi: X\rightarrow U$ be the associated morphism. Let $\phi: X\dashrightarrow Z_1$ be an invariant Iitaka fibration$/U$ of $K_{\Aa}$. Let $h: \Aa'\rightarrow\Aa$ be a proper LD log toroidal model with respect to $\phi$ associated with $f: X'\rightarrow Z$ whose existence is guaranteed by Definition-Lemma \ref{deflem: eoltm}, and let $\Aa'^0:=\Aa'\times_UU^0$, $X'^0:=X'\times_UU^0$. Let $h^0:=h|_{X'^0}$, then $h^0:\Aa'^0\rightarrow\Aa^0$ is a proper log toroidal model of $\Aa^0$. Since $\Aa^0/U^0$ has a bs-semi-ample model, by Lemma \ref{lem: Cas+25b 3.13}, $\Aa'^0/U^0$ has a good log minimal model, and $K_{\Aa}$ is abundant$/U$ and pseudo-effective$/U$. By Lemma \ref{lem: iitaka fibration numerical abundant divisor gpair dimension},
$$\kappa_{\sigma}(X'/U,K_{\Aa'})=\kappa_{\iota}(X'/U,K_{\Aa'})=\dim Z-\dim U,\ \kappa_{\sigma}(X'/Z,K_{\Aa'})=\kappa_{\iota}(X'/Z,K_{\Aa'})=0.$$
By Lemma \ref{lem: has19 3.2 step 3 abu ver} and Theorem \ref{thm: bswlc imply mm}, we may run a $K_{\Aa'}$-MMP$/Z$ with scaling of an ample divisor which terminates with a good minimal model $\Aa''/Z$ of $\Aa'/Z$ associated with birational map$/Z$ $\psi: X'\dashrightarrow X''$, such that $K_{\Aa''}\sim_{\mathbb R,Z}0$. Let $\psi^0:=\psi|_{X'^0}$ and $\Aa''^0:=\Aa''\times_UU^0$, then $\psi^0$ is $K_{\Aa'^0}$-negative. By Lemma \ref{lem: minimal model same after running mmp}, $\Aa''^0/U^0$ has a bs-good minimal model. Since $\Aa'$ is LD, by Lemma \ref{lem: LD preserve under mmp}, $\Aa''$ is LD.

Since $K_{\Aa''}\sim_{\mathbb R,Z}0$ and $\kappa_{\sigma}(X''/U,K_{\Aa''})=\dim Z-\dim U$, by Lemma \ref{lem: property of numerical and Iitaka dimension}(1)(4), we have $K_{\Aa''}\sim_{\mathbb R,U}f''^*G_Z$ for some $G_Z\geq 0$ that is big$/U$, where $f'': X''\rightarrow Z$ is the associated contraction. By Lemma \ref{lem: special gdlt modification gpair}, there exists a $\mathbb Q$-factorial dlt modification $g: Y\rightarrow X''$ of $\Aa''$ with 
$$\Aa_Y:=g^*\Aa'':=(Y,B_Y=B_Y^h+B_Y^v,\Mm),$$ 
such that
\begin{itemize}
    \item $B^h_Y\geq 0$ and $B_Y^v$ is reduced,
    \item $B_Y^v$ is vertical$/Z$, and
    \item for any real number $t>0$, any lc center of $(\Aa_Y,-tB_Y^v)$ dominates $Z$.
\end{itemize}
Since $G_Z$ is big$/U$, possibly replacing $G_Z$, we may assume that $\pi''\circ f(\Supp B_Y^v)\subset\Supp G_Z$. Therefore,
\begin{itemize}
    \item $K_{\Aa_Y}\sim_{\mathbb R,U}E:=\left(\pi''\circ f\right)^*G_Z\geq 0$, and
    \item $E$ is vertical$/Z$ and $\Supp B^v\subset\Supp E$.
\end{itemize}
By Lemma \ref{lem: special dlt model ii}, there exists a $\mathbb Q$-factorial dlt modification $p: W\rightarrow Y$ of $\Aa_Y$ satisfying the following.
\begin{itemize}
    \item $\Aa_W:=p^*\Aa_Y:=(W,B_W,\Mm)$.
    \item $B_W=B_W^h+B_W^v$ and $K_{\Aa_W}\sim_{\mathbb R,Z}E_W^h+E_W^v$.
    \item $0\leq E_W^h,E_W^v$ are vertical$/Z$, $\Supp E_W^v=\Supp B_W^v$, and $B_W^v$ is reduced.
    \item For any $0<t\ll 1$, $(\Aa_W,tE_W^h)$ is dlt.
    \item For any $0<t\ll 1$, any lc center of $(\Aa_W,-tB_W^v)$ dominate $Z$.
\end{itemize}
For any $0<t\ll 1$, we have
$$0\leq E_W^h+E_W^v-tE_W^v\leq E_W^h+E_W^v\quad \text{and}\quad \Supp(E_W^h+E_W^v-tE_W^v)=\Supp(E_W^h+E_W^v).$$
By Lemma \ref{lem: property of numerical and Iitaka dimension}(2),
\begin{align*}
&\kappa_{\sigma}(W/U,K_{\Aa_W}-tE_W^v)=\kappa_{\sigma}(W/U,E_W^h+E_W^v-tE_W^v)=\kappa_{\sigma}(W/U,E_W^h+E_W^v)\\
=&\kappa_{\sigma}(W/U,K_{\Aa_W})=\kappa_{\sigma}(X''/U,K_{\Aa''})=\kappa_{\sigma}(X'/U,K_{\Aa'})=\dim Z-\dim U.
\end{align*}
Since $E_W^h,E_W^v$ are vertical$/Z$, we have
$$\kappa_{\sigma}(W/Z,K_{\Aa_W}-tE_W^v)=\kappa_{\sigma}(W/Z,E_W^h+E_W^v-tE_W^v)=0$$
and
$$\kappa_{\iota}(W/Z,K_{\Aa_W}-tE_W^v)=\kappa_{\iota}(W/Z,E_W^h+E_W^v-tE_W^v)=0.$$
Since any lc center of $(\Aa_W,-tB_W^v)$ dominates $Z$ for any $t>0$ and $\Supp B_W^v=\Supp E_W^v$, any lc center of $(\Aa_W,-tE_W^v)$ dominates $Z$ for any $t>0$.

By Proposition \ref{prop: abundant+lcc dominate imply gmm}, $(\Aa_W,-tE_W^v)/U$ has a good minimal model for any $0<t\ll 1$.

Since $\Aa''^0/U^0$ has a bs-good minimal model, by Lemma \ref{lem: Cas+25b 3.13}, $$\Aa_W^0/U^0:=\left(\Aa_W\times_UU^0\right)/U^0$$ has a good log minimal model. Since $\Aa''$ is LD, by Lemma \ref{lem: LD under dlt model}, $\Aa_W$ is LD. By Lemma \ref{lem: termination over open subset}, we may run a sequence of steps of a $K_{\Aa_W}$-MMP$/U$ with scaling of an ample divisor
$$\phi_i: \Aa_i\dashrightarrow \Aa_{i+1},\Aa_1:=\Aa_W, 1\leq i\leq n-1$$
such that $\Aa_n^0/U^0:=\left(\Aa_n\times_UU^0\right)/U^0$ is a $\mathbb Q$-factorial good minimal model of $\Aa_W^0/U^0$. Let $X_n$ be the ambient variety of $\Aa_n$ and let $B_n,B_n^h,B_n^v,E_n^h,E_n^v$ be the images of $B_W,B_W^h,B_W^v,E_W^h,E_W^v$ on $X_n$ respectively. Then:
\begin{itemize}
    \item By Lemma \ref{lem: LD preserve under mmp} and since $\Aa_W$ is $\mathbb Q$-factorial LD dlt, $\Aa_n=(X_n,B_n,\Mm)$ is $\mathbb Q$-factorial LD dlt.
    \item $B_n=B_n^h+B_n^v$ and $K_{\Aa_n}\sim_{\mathbb R,U}E_n^h+E_n^v$.
    \item $E_n^h,E_n^v\geq 0$, $\Supp E_n^v=\Supp B_n^v$, an $B_n^v$ is reduced.
    \item Since $\{\phi_i\}_{i=1}^{n-1}$ is also a sequence of steps of a $(K_{\Aa_W}+tE_W^h)$-MMP$/U$ for any $0<t\ll 1$, $(\Aa_n,tE_1^h)$ is dlt for any $0<t\ll 1$.
    \item Since $(\Aa_W,-tB_W^v)/U$ has a good minimal model, and $\{\phi_i\}_{i=1}^{n-1}$ is also a sequence of steps of a $(K_{\Aa_W}-tB_W^v)$-MMP$/U$ for any $0<t\ll 1$, by Lemma \ref{lem: minimal model same after running mmp}, $(\Aa_n,-tB_n^v)/U$ has a bs-good minimal model for any $0<t\ll 1$.
    \item For any lc place $D$ of $\Aa_n$, 
    $$0=a(D,\Aa_n)\geq a(D,\Aa_W)=a(D,\Aa'')\geq a(D,\Aa'),$$
    so $D$ is  lc place of $\Aa'$, hence $D$ is an lc place of $\Aa$. Thus the image of $D$ in $U$ intersects $U^0$.
    \item $\Aa_n^0/U^0$ is a good minimal model of itself.
\end{itemize}
By Proposition \ref{prop: main theorems for special dlt model}, $\Aa_n/U$ has a good minimal model. By Lemma \ref{lem: minimal model same after running mmp}, $\Aa_W/U$ has a good minimal model. By Lemma \ref{lem: mm preserved under dlt model}, $\Aa''/U$ has a bs-good minimal model. By Lemma \ref{lem: minimal model same after running mmp}, $\Aa'/U$ has a bs-good minimal model. By Lemma \ref{lem: Bir12 2.8}, $\Aa/U$ has a bs-good minimal model. By Lemma \ref{lem: g-pair weak glc imply lmm}, $\Aa/U$ has a good log minimal model.
\end{proof}

\begin{prop}\label{prop: induction to bir12 1.1}
Assume that Theorems \ref{thm: hx13 1.1 nonnqc-g} and \ref{thm: bir12 1.1 3non} hold in dimension $\leq d-1$. Then Theorem \ref{thm: bir12 1.1 3non} holds in dimension $d$.
\end{prop}
\begin{proof}
Notations and conditions as in Theorem \ref{thm: bir12 1.1 3non}. Let $\pi: X\rightarrow U$ be the associated morphism and $X\rightarrow U'\rightarrow U$ the Stein factorization of $\pi$. Possibly replacing $U$ with $U'$ and $U^0$ with $U'\times_UU^0$, we may assume that $\pi$ is a contraction. 

By Lemma \ref{lem: special gdlt modification gpair}, there exists a $\mathbb Q$-factorial dlt modification $h: X'\rightarrow X$ of $\Aa$ with $\Aa':=h^*\Aa:=(X',B',\Mm)$, such that $B'=B^h+B^v$ satisfies the following.
\begin{itemize}
\item $B^h\geq 0$ and $B^v$ is reduced.
\item $B^v$ is vertical$/U$.
\item For any real number $t>0$, any lc center of $(\Aa',-tB^v)$ dominates $U$.
\end{itemize}
Let $A':=h^*A$, $X'^0:=X'\times_UU^0$, and $A'^0:=A'|_{X^0}$. Then $(\Aa'^0,A'^0)$ is lc and $K_{\Aa'^0}+A'^0\sim_{\mathbb R,U^0}0$. Since $A\geq 0$, $A'\geq 0$. Since $K_{\Aa'}=h^*K_{\Aa}$ is pseudo-effective$/U$, $A'$ is vertical$/U$. Therefore, $K_{\Aa'}\sim_{\mathbb R}0$ over the generic point of $U$, and we may write 
\begin{itemize}
    \item $K_{\Aa'}\sim_{\mathbb R,U}E\geq 0$
\end{itemize}
for some $E$ that is vertical$/U$ (cf. \cite[Lemma 3.2.1]{BCHM10}, \cite[Lemma 2.3]{HL23}). Let $H\geq 0$ be an ample divisor on $U$ such that $(\pi\circ h)(\Supp B^v)\subset\Supp H$. Possibly replacing $E$ with $E+(\pi\circ h)^*H$, we may assume that 
\begin{itemize}
    \item $E$ that is vertical$/U$ and $\Supp B^v\subset\Supp E$.
\end{itemize}
By Lemma \ref{lem: special dlt model ii}, there exists a $\mathbb Q$-factorial dlt modification $g: Y\rightarrow X'$ of $\Aa'$ satisfying the following.
\begin{itemize}
    \item $\Aa_Y:=g^*\Aa':=(Y,B_Y,\Mm)$. Since $\Aa$ is LD, by Lemma \ref{lem: LD under dlt model}, $\Aa_Y$ is $\mathbb Q$-factorial LD dlt.
    \item $B_Y=B_Y^h+B_Y^v$ and $K_{\Aa_Y}\sim_{\mathbb R,U}E_Y^h+E_Y^v$.
    \item $0\leq E_Y^h,E_Y^v$ are vertical$/U$, $\Supp E_Y^v=\Supp B_Y^v$, and $B_Y^v$ is reduced.
    \item For any $0<t\ll 1$, $(\Aa_Y,tE_Y^h)$ is dlt.
    \item For any $0<t\ll 1$, any lc center of $(\Aa_Y,-tB_Y^v)$ dominates $U$.
\end{itemize}
Since $\Supp B_Y^v=\Supp E_Y^v$, for any $0<t\ll 1$, any lc center of $(\Aa_Y,-tE_Y^v)$ dominates $U$. Since
$$K_{\Aa_Y}-tE_Y^v\sim_{\mathbb R,U}E_Y^h+(1-t)E_Y^v$$
is $\geq 0$ and is vertical$/U$, we have
$$\kappa_{\sigma}(Y/U,K_{\Aa_Y}-tE_Y^v)=\kappa_{\iota}(Y/U,K_{\Aa_Y}-tE_Y^v)=0.$$
By Proposition \ref{prop: abundant+lcc dominate imply gmm}, 
\begin{itemize}
    \item $(\Aa_Y,-tE_Y^v)/U$ has a good log minimal model for any $0<t\ll 1$.
\end{itemize}
Moreover:
\begin{itemize}
\item For any lc place $D$ of $\Aa_Y$, $D$ is also an lc place of $\Aa$, so the image of $D$ on $U$ intersects $U^0$.
\item Let $A_Y:=g^*A'$, $Y^0:=Y\times_UU^0$, $\Aa_Y^0:=\Aa_Y\times_UU^0$, and $A_Y^0:=A_Y|_{Y^0}$. Then
$$\left(\Aa_Y^0,A_Y^0\right)=(h\circ g)^*(\Aa_Y,A)|_{Y^0}=\left((h\circ g)|_{Y^0}\right)^*\left(\Aa^0,A^0\right),$$
so $\left(\Aa_Y^0,A_Y^0\right)$ is lc and $K_{\Aa_Y^0}+A_Y^0\sim_{\mathbb R,U^0}0$.
\end{itemize}
By Proposition \ref{prop: main theorems for special dlt model}, $\Aa_Y/U$ has a good minimal model. By Lemma \ref{lem: mm preserved under dlt model}, $\Aa/U$ has a bs-good minimal model. By Lemma \ref{lem: g-pair weak glc imply lmm}, $\Aa/U$ has a good log minimal model.
\end{proof}

\begin{proof}[Proof of Theorems \ref{thm: hx13 1.1 nonnqc-g} and \ref{thm: bir12 1.1 3non}]
Theorems \ref{thm: hx13 1.1 nonnqc-g} and \ref{thm: bir12 1.1 3non} hold in dimension $1$ trivially. Theorems \ref{thm: hx13 1.1 nonnqc-g} and \ref{thm: bir12 1.1 3non} hold in general by Propositions \ref{prop: main theorems for special dlt model}, \ref{prop: induction to hx13 1.1}, and \ref{prop: induction to bir12 1.1} and induction on dimensions.
\end{proof}

\section{Existence of the minimal model program}\label{sec: eommp}

The goal of this section is to prove Theorems \ref{thm: flip nonnqc-g}, \ref{thm: eommp nonnqc-g}, and \ref{thm: bir12 1.1 nonnqc-g}.

We begin with a variation statement (Theorem~\ref{thm: bir12 1.1 variation1}), which reduces the general case to the LD case established in Section~\ref{sec: hx13 1.1}.  The remaining theorems are then deduced by applying this statement to suitable (q)dlt modifications and running the MMP with scaling.

\begin{thm}\label{thm: bir12 1.1 variation1}
Let $\Aa/U$ be a generalized pair and let $A\geq 0$ be an $\Rr$-divisor such that $(\Aa,A)/U$ is lc and $K_{\Aa}+A\sim_{\mathbb R,U}0$ and $K_{\Aa}$ is pseudo-effective$/U$. Then $\Aa/U$ has a good log minimal model.
\end{thm}
\begin{proof}
Let $h: X'\rightarrow X$ be a $\mathbb Q$-factorial dlt modification of $\Aa$. By Lemma \ref{lem: mm preserved under dlt model}, possibly replacing $\Aa$ with $h^*\Aa$ and $A$ with $h^*A$, we may assume that $\Aa$ is $\mathbb Q$-factorial dlt. 

Write $\Aa=(X,B,\Mm)$. By Lemma \ref{lem: Cas+25b 3.5}, possibly replacing $\Mm$, we may assume that $K_{\Aa}+A=0$. In particular, $(\Aa,A)$ is LD. By Lemma \ref{lem: combination of LD with lc g-pair}, $\left(\Aa,\frac{1}{2}A\right)$ is LD. By Theorem \ref{thm: bir12 1.1 3non}, $\left(\Aa,\frac{1}{2}A\right)/U$ has a good log minimal model. By Theorem \ref{thm: bswlc imply mm},  $\left(\Aa,\frac{1}{2}A\right)/U$ has a $\mathbb Q$-factorial good minimal model. Since
$$K_{\Aa}+\frac{1}{2}A\sim_{\mathbb R,U}\frac{1}{2}K_{\Aa},$$
by Lemma \ref{lem: Cas+25b 3.5}, $\Aa/U$ has a good minimal model. The theorem follows.
\end{proof}

\begin{thm}\label{thm: extremal contraction nonnqc-g}
    Let $\Aa/U$ be a log canonical generalized pair with ambient variety $X$ and let $f: X\rightarrow Z$ be a $K_{\Aa}$-negative extremal contraction. If $\dim X=\dim Z$, then the ample model$/Z$ $\phi: X\dashrightarrow X^+$ of $K_{\Aa}$  exists.
\end{thm}
\begin{proof}
We have that $K_{\Aa}$ is pseudo-effective$/Z$ and $-K_{\Aa}$ is ample$/Z$, so we may choose an ample$/Z$ $\mathbb R$-divisor $0\leq A\sim_{\mathbb R}-K_{\Aa}$ on $X$ such that $(\Aa,A)$ is lc. By Theorem \ref{thm: bir12 1.1 variation1}, $\Aa/Z$ has a good minimal model. Thus we may let $\phi: X\dashrightarrow X^+$ be the ample model$/Z$ of $K_{\Aa}$, and the theorem follows.
\end{proof}

\begin{proof}[Proof of Theorem \ref{thm: flip nonnqc-g}]
It is a special case of Theorem~\ref{thm: extremal contraction nonnqc-g} by considering the case when $\Aa=(X,B,\Mm)$ and $f$ is a $K_{\Aa}$-flipping contraction.
\end{proof}

\begin{proof}[Proof of Theorem \ref{thm: eommp nonnqc-g}]
By the cone theorem and the contraction theorem \cite[Theorem 2.2.1]{CHLX23}, we know that there exists a $(K_X+B+\Mm_X)$-negative extremal ray $R$ and a contraction $X\rightarrow Z$ of $R$ whenever $K_X+B+\Mm_X$ is not nef. If $\dim X>\dim Z$ then we obtain a $(K_X+B+\Mm_X)$-Mori fiber space and we are done. Otherwise, by Theorem \ref{thm: extremal contraction nonnqc-g} we may run a step of the MMP corresponding to the contraction of $R$. By repeating this process, the theorem follows.
\end{proof}

\begin{proof}[Proof of Theorem \ref{thm: bir12 1.1 nonnqc-g}]
Let $\Aa:=(X,B,\Mm)$. If $K_{\Aa}$ is pseudo-effective$/U$, then we are done by Theorem \ref{thm: bir12 1.1 variation1}. If $K_{\Aa}$ is not pseudo-effective$/U$, then we may let $h: X'\rightarrow X$ be a $\mathbb Q$-factorial dlt modification of $\Aa$ and let $\Aa':=h^*\Aa$. We may run a $K_{\Aa'}$-MMP$/U$ with scaling of an ample divisor which terminates with a Mori fiber space $f: X''\rightarrow Z$ associated with birational map $\phi: X'\dashrightarrow X''$. Let $\Aa'':=\phi_*\Aa'$, then $f: \Aa''\rightarrow Z$ is a log Mori fiber space of $\Aa/U$.
\end{proof}

\section{Structure of \texorpdfstring{$\mathbb Q$}{}-factorial generalized pairs}\label{sec: structure theorem}

The goal of this section is to prove Theorem~\ref{thm: structure of q-factorial gpair} and Corollary~\ref{cor: potentially lc base}, and thus complete the proofs of all main results stated in the introduction. We begin with a $P$-trivial MMP statement (Lemma~\ref{lem: P-trivial}) and then follow the strategy of \cite{HL23}, replacing NQC decompositions by LD ones, to deduce the structure theorem and its corollary.

\begin{defn}
		Let $\pi: X\rightarrow U$ be a projective morphism between normal quasi-projective varieties. An \emph{extremal curve of minimal length} in $\overline{NE}(X/U)$ is a curve $C$ on $X$ such that $[C]$ spans an extremal ray $R$ in $\overline{NE}(X/U)$, and for any curve $C'$ on $X$ such that $[C']$ spans $R$, $[C]\equiv\lambda [C']$ for some $\lambda\leq 1$.
\end{defn}
    
\begin{lem}[{cf. \cite[Lemma 3.8]{Hu25}}]\label{lem: P-trivial}
		Let $\Aa/U=(X,B,\Mm)/U$ be a $\mathbb Q$-factorial qdlt generalized pair and let $P$ be a nef$/U$ $\Rr$-divisor on $X$. Suppose that $\Supp\{P\}\subset\Supp\{B\}$. Then for any $\alpha\gg 0$, any sequence of steps of a $(K_{\Aa}+\alpha P)$-MMP$/U$ is $P$-trivial.
\end{lem}
\begin{proof}
Write $\{P\}=\sum_{i=1}^m d_iD_i$ where $D_i$ are prime divisors. Let $Q(\bm{v}):=\sum_{i=1}^mv_iD_i$ for any $\bm{v}=(v_1,\dots,v_m)\in\mathbb R^m$, $\bm{d}:=(d_1,\dots,d_m)$, and $V\ni\bm{d}$ the rational envelope of $\bm{d}$ in $\mathbb R^m$. Let $c:=\dim V+1$, $\bm{v}_1,\dots,\bm{v}_c\in V\cap\mathbb Q^m$ vectors such that $\bm{d}$ is contained in the interior of the convex hull spanned by $\bm{v}_1,\dots,\bm{v}_c$ and $||\bm{v}_j-\bm{d}||_{\infty}\ll 1$ for any $j$. Then there exist unique real numbers $\beta_1,\dots,\beta_c\in (0,1]$ such that $\sum_{j=1}^c\beta_j=1$, $\sum_{j=1}^c\beta_j\bm{v}_j=\bm{d}$, and $\beta_1,\dots,\beta_c$ are $\mathbb Q$-linearly independent.

Let $Q_j:=Q(\bm{v}_j)$, $P_j:=\lfloor P\rfloor+Q_j$, $B_j:=B+P-P_j$, and $\Aa_j:=(X,B_j,\Mm)$ for any $j$. Since $\Aa$ is $\mathbb Q$-factorial qdlt,
$$\Supp(P_j-P)\subset\Supp\{P\}\subset\Supp\{B\},$$
and $||\bm{v}_j-\bm{d}||\ll 1$, we have that $\Aa_j$ is $\mathbb Q$-factorial qdlt for any $j$. Then for any $K_{\Aa}$-negative extremal curve$/U$ $C$, by the length of extremal rays \cite[Theorem 2.2.1(2)]{CHLX23}, we have
$$-2\dim X<K_{\Aa_j}\cdot C=K_{\Aa}\cdot C+(P_j-P)\cdot C<P_j\cdot C.$$
Let $I$ be a positive integer such that $IP_j$ is Cartier for any $j$ and let 
$$\Ii:=\min\left\{\sum\beta_jp_j\middle| Ip_j\in [-2I\dim X,+\infty)\cap\mathbb Z\right\}.$$
		Then $\Ii$ is a discrete set depending only on $I$, hence we may define $$\gamma_0:=\min\{\gamma\in\Ii\mid \gamma>0\}.$$
		Since $P$ is nef$/U$ and $P=\sum_{j=1}^c\beta_jP_j$, we have that
        $$P\cdot C=0\quad  \text{or}\quad  P\cdot C\geq\gamma_0$$ 
        for any $K_{\Aa}$-negative extremal curve of minimal length in $\overline{NE}(X/U)$. By the length of extremal rays, a step of a $(K_{\Aa}+\alpha P)$-MMP$/U$ is $P$-trivial for any $\alpha\geq\alpha_0:=\frac{2\dim X+1}{\gamma_0}$. 

        Since $\beta_1,\dots,\beta_c$ are $\mathbb Q$-linearly independent, for any sequence of steps of a $(K_{\Aa}+\alpha P)$-MMP$/U$ $\phi: X\dashrightarrow X'$ that is $P$-trivial, $\phi$ is $P_j$-trivial for any $j$, hence $\phi$ is also a sequence of steps of a $K_{\Aa_j}$-MMP$/U$. 

        Note that $\alpha_0$ depends only on $\dim X$, $\beta_1,\dots,\beta_c$, and $I$. Let $P':=\phi_*P$ and $P_j':=\phi_*P_j$. Then $\dim X'=\dim X$, $IP_j'$ is Cartier for any $j$ by the contraction theorem \cite[Theorem 2.2.1(4)]{CHLX23}, $P'=\sum_j \beta_jP_j'$, and $\phi_*\Aa_j$ is $\mathbb Q$-factorial qdlt for any $j$. So we may replace $\Aa,P,P_j$ with $\phi_*\Aa,P',P_j'$ and continue the process.
	\end{proof}

\begin{lem}\label{lem: special extraction}
		Let $\Aa/U:=(X,B,\Mm)/U$ be an lc generalized pair such that $\Mm_X$ is $\Rr$-Cartier. Then there exists a projective birational morphism $f: Y\rightarrow X$ and an $\mathbb R$-divisor $L\geq 0$ on $Y$ satisfying the following.
		\begin{enumerate}
			\item $\Mm$ descends to $Y$.
			\item $\Supp(f^*\Mm_X-\Mm_Y)=\Exc(f)$.
			\item $\Supp L=\Exc(f)$ and $-L$ is ample$/X$.
		\end{enumerate}
	\end{lem}
    \begin{proof}
We may replace $U$ by $X$ and assume that the associated morphism $X\rightarrow U$ is the identity morphism. Let $h: W\rightarrow X$ be a log resolution of $(X,B,\Mm)$ and let $E:=h^*\Mm_X-\Mm_W$. Then $-E\sim_{\mathbb R,X}\Mm_W$ is nef$/X$ and exceptional$/X$. By the negativity lemma, $E\geq 0$. Possibly replacing $\Mm$ with $-\overline{E}$, we may assume that $\Mm_W=-E\leq 0$. In particular, $\Mm$ is exceptional$/X$ and $\Mm_X=0$. We may write
$$K_W+B_W(t)+t\Mm_W=K_W+B_W(t)-tE=h^*(K_X+B)=h^*(K_X+B+t\Mm_X)$$
for any $t\in [0,1]$. Pick $0<\delta\ll 1$. Let
$$\Delta_W:=B_W\left(1/2\right)\vee (1-\delta)\Exc(h)$$
and let
$$\Aa_W(t):=(W,\Delta_W,t\Mm)$$
for any real number $t\geq 0$. Then $\Aa_W(t)$ is $\mathbb Q$-factorial dlt log smooth for any $t\geq 0$. Moreover, for any prime divisor $D$ on $W$ that is exceptional$/X$, either $\mult_D\Delta_W=\mult_DB_W(1/2)=1$ or $\mult_D\Delta_W>\mult_DB_W(1/2)$. In particular, we have
$$K_{\Aa_W(t)}\sim_{\mathbb R,X}(t-1/2)\Mm_W+F=F-(t-1/2)E$$
for some $F\geq 0$ that is exceptional$/X$, such that $\Supp F$ is exactly the union of all $h$-exceptional prime divisors $D$ such that $\mult_DB_W(1/2)<1$.

By our construction, 
$$\Supp\Mm_X=\Supp E\subset\Exc(h)\subset\Supp\Delta_W.$$ 
Moreover, for any irreducible component $D$ of $E$, we have that
$$\mult_DB_W(1/2)=\mult_DB_W(1)-t\mult_DE<1,$$ 
so $\mult_D\Delta_W<1$. Since $\Delta_W\geq 0$, we have
$$\Supp\Mm_X=\Supp E\subset\Supp\{\Delta_W\}.$$
Therefore, $\Aa_W(t)$ is $\mathbb Q$-factorial dlt LD log smooth for any real number $t\geq 0$. 

Fix an integer $\alpha\gg 0$. By Lemma \ref{lem: P-trivial}, we may run a $K_{\Aa_W(\alpha)}$-MMP$/X$ with scaling of an ample divisor
$$\phi_i: \Aa_i\dashrightarrow \Aa_{i+1}, \Aa_1:=\Aa_W(\alpha)$$
and this MMP is $E$-trivial. Let $X_i$ be the ambient variety of $\Aa_i$ for each $i$ and let $\Aa_i(t),F_i,E_i$ be the images of $\Aa_W(t),F,E$ on $X_i$ for each $i$. Then there exists $i>0$ such that $K_{\Aa_i}$ is movable$/X$. We have
$$K_{\Aa_i}\sim_{\mathbb R,X}F_i-(\alpha-1/2)E_i.$$
By \cite[Lemma 3.3]{Bir12}, $F_i-(\alpha-1/2)E_i\geq 0$. Let $h_i: X_i\rightarrow X$ be the induced birational morphism. Since the MMP is $E$-trivial, $-E_i$ is nef$/X$. Thus for any closed point $x\in X$ and any irreducible curve $C\subset h_i^{-1}(x)$, we have $-E_i\cdot C\geq 0$, hence either $C$ does not intersect $\Supp E_i$ or $C\subset\Supp E_i$. Since $h_i$ has connected fibers, we have $\Supp E_i=h_i^{-1}(h_i(\Supp E))$. 

We let $X^0:=X\backslash h_i(\Supp E_i)$, $X_i^0:=X_i\times_XX^0$, $\Aa_i^0:=\Aa_i\times_XX^0$, and $\Aa_i^0(t):=\Aa_i(t)\times_XX^0$ for any $t\geq 0$. Write
$$\Aa_i(t):=(X_i,\Delta_i,t\Mm)$$
for any $t\geq 0$. Since $\Supp E$ does not contain any stratum of $\lfloor\Delta_W\rfloor$, $\Supp E_i$ does not contain any stratum of $\lfloor\Delta_i\rfloor$. By Lemma \ref{lem: LD preserve under mmp}, $\Aa_i$ is $\mathbb Q$-factorial qdlt LD, so $\Supp E_i$ does not contain any lc center of $\Aa_i$. Since $\Mm$ descends to $X_i$, $\Aa_i(t)$ is $\mathbb Q$-factorial qdlt for any $t\geq 0$. Thus $\Supp E_i$ does not contain any lc center of $\Aa_i(t)$ for any $t\geq 0$. Therefore, $X_i^0$ intersects all lc centers of $\Aa_i(t)$ for any $t\geq 0$. Moreover, we have
$$\Supp\left(F_i-(\alpha-1/2)E_i\right)=\Supp\left(F_i-(\alpha-1/2)E_i\right)^{\leq 0}=\Supp E_i,$$
so $\Supp F_i\subset\Supp E_i$. We have
$$K_{\Aa_i(t)}\sim_{\mathbb R,X}F_i-(t-1/2)E_i$$
for any $t\geq 0$, so
$$K_{\Aa_i^0(t)}\sim_{\mathbb R,X^0}0,$$
and so $\Aa_i^0(t)/X^0$ is a good minimal model of itself for any $t\geq 0$.

Since $\phi_i$ is $E$-trivial, $\phi_i$ is a sequence of steps of a $K_{\Aa_W(t)}$-MMP$/X$ for any $t\geq 0$, by Lemma \ref{lem: LD preserve under mmp}, $\Aa_i(t)$ is $\mathbb Q$-factorial qdlt LD for any $t\geq 0$. By Theorem \ref{thm: hx13 1.1 nonnqc-g}, $\Aa_i(t)/X$ has a good minimal model for any $t\geq 0$. By Lemma \ref{lem: minimal model same after running mmp}, $\Aa_W(t)/X$ has a good log minimal model for any $t\geq 0$.

We pick a real number $s\gg 0$ that is general in $\mathbb R$ and run a $K_{\Aa_W(s)}$-MMP$/X$ with scaling of an ample divisor. By Theorem \ref{thm: eomm implies tof with scaling}, the MMP terminates with a good minimal model $\Aa_V(s)/X$ of $\Aa_W(s)/X$ associated with birational map$/X$ $\psi: W\dashrightarrow V$. By Lemma \ref{lem: P-trivial}, $\psi$ is $\Mm_W$-trivial. Let $g: V\rightarrow Y$ be the ample model$/X$ of $K_{\Aa_V(s)}$, $f: Y\rightarrow X$ the associated contraction, and let $\Aa_Y(s):=g_*\Aa_V(s)$. Since $s$ is general in $\mathbb R$, $g$ is $\Mm_V$-trivial, so $\Mm$ descends to $Y$. Let $F_Y$ and $E_Y$ be the images of $F,E$ on $Y$, then
$$K_{\Aa_Y(s)}\sim_{\mathbb R,X}F_Y-(s-1/2)E_Y$$
is ample$/X$. Let $L:=(s-1/2)E_Y-F_Y$. Since $L$ is exceptional$/X$, by the negativity lemma, we have that $L\geq 0$ and $$\Exc(f)=\Supp L.$$
Since $s\gg 0$,
$$\Supp L=\Supp E_Y=\Supp\Mm_Y=\Supp(f^*\Mm_X-\Mm_Y).$$
Therefore, $Y$ and $L$ satisfy our requirements.
\end{proof}

\begin{thm}\label{thm: structure of mx rcartier gpair}
    Let $(X,B,\Mm)/U$ be an lc generalized pair such that $\Mm_X$ is $\mathbb R$-Cartier and $A$ an ample$/U$ $\mathbb R$-divisor on $X$. Then there exists an lc pair $(X,\Delta)$ such that
    $$K_X+\Delta\sim_{\mathbb R,U}K_X+B+\Mm_X+A.$$
\end{thm}
\begin{proof}
Possibly replacing $\Mm$ with $(1-\epsilon)\Mm$ and $A$ with $A+\epsilon\Mm_X$ for some $0<\epsilon\ll 1$, we may assume that $\Nklt(X,B,\Mm)=\Nklt(X,B)$. By Lemma \ref{lem: special extraction}, there exists a birational morphism $f: Y\rightarrow X$ an anti-ample$/X$ effective $\Rr$-divisor $E\geq 0$, such that $\Mm$ descends to $Y$ and 
		$$\Supp(f^*\Mm_X-\Mm_Y)=\Exc(f)=\Supp E.$$
		In particular, $\Exc(f)=\Supp E$ does not contain any lc center of $(X,B,\Mm)$.  Let
        $$K_Y+B_Y:=f^*(K_X+B),$$ 
        then we may find $0<\delta\ll 1$ such that $(Y,B_Y+\delta E)$ is sub-lc and $f^*A-\delta E$ is ample$/U$. In particular, there exists an ample$/U$ $\Rr$-divisor $0\leq H_Y\sim_{\mathbb R,U}\Mm_Y+f^*A-\delta E$ such that $(Y,B_Y+H_Y+\delta E)$ is lc. We may take $\Delta:=B+f_*H_Y$.
\end{proof}

\begin{proof}[Proof of Theorem \ref{thm: structure of q-factorial gpair}]
   It is a special case of Theorem~\ref{thm: structure of mx rcartier gpair}. 
\end{proof}

\begin{proof}[Proof of Corollary \ref{cor: potentially lc base}]
Let $A:=-(K_X+B+\Mm_X)$. By Theorem \ref{thm: structure of q-factorial gpair}, there exists an lc pair $(X,\Delta)$ such that
$$K_X+\Delta\sim_{\mathbb R,Z}K_X+B+\Mm_X+A\sim_{\mathbb R,Z}0.$$
The corollary follows from \cite[Theorem 7.6 and 7.2.3]{BFMT25}.
\end{proof}

\end{document}